\documentclass[fleqn]{article}

\usepackage[a4paper, left=2cm, right=2cm, top=3cm, bottom=3cm]{geometry}                           
\usepackage[applemac]{inputenc}                                                                    
\usepackage[T1]{fontenc}                                                                           
\usepackage[colorlinks, linkcolor=black, citecolor=black, urlcolor=black]{hyperref}                
\usepackage{ifthen}                                                                                
\usepackage{amsmath}                                                                               
\usepackage{amssymb}                                                                               
\usepackage{amsthm}                                                                                
\usepackage{upgreek}                                                                               
\usepackage{tikz}                                                                                  

\usetikzlibrary{calc,decorations.markings,matrix,positioning}                                               

\allowdisplaybreaks                                                                                

\swapnumbers
\theoremstyle{definition}
\newtheorem{definition}{Definition}[section]
\newtheorem{corollary}[definition]{Corollary}
\newtheorem{example}[definition]{Example}
\newtheorem{lemma}[definition]{Lemma}
\newtheorem{proposition}[definition]{Proposition}

\newtheorem{remark}[definition]{Remark}
\newtheorem*{remark*}{Remark}
\newtheorem{theorem}[definition]{Theorem}
\newtheorem*{theorem*}{Theorem}

\makeatletter
\renewcommand*\p@enumii{}                                                                          
\makeatother

\def\morchoice#1#2{\begingroup\setbox0=\hbox{$#1\xrightarrow{#2}$}%
        \setbox1=\hbox{$#1\longrightarrow$}%
        \ifdim\wd0<\wd1
        \stackrel{#2}\longrightarrow
        \else
        \xrightarrow{#2}\fi\endgroup}

\newcommand{\booktitle}[1]{\textsl{#1}}                                                            
\newcommand{\eigenname}[1]{\textsc{#1}}                                                            
\newcommand{\newnotion}[1]{\textit{#1}}                                                            

\newenvironment{smallpmatrix}{\left( \begin{smallmatrix}}{\end{smallmatrix} \right)}               
\newcommand{\nbd}{\nobreakdash-\hspace{0pt}}                                                       

\DeclareMathOperator{\Arr}{\mathrm{Arr}}                                                           
\DeclareMathOperator{\Cokernel}{\mathrm{Coker}}                                                    
\DeclareMathOperator{\Cone}{\mathrm{Cone}}                                                         
\DeclareMathOperator{\Denominators}{\mathrm{Den}}                                                  
\DeclareMathOperator{\DiscreteCategory}{\mathrm{Disc}}                                             
\DeclareMathOperator{\FractionCategory}{\mathrm{Frac}}                                             
\DeclareMathOperator{\HomotopyCategory}{\mathrm{Ho}}                                               
\DeclareMathOperator{\Kernel}{\mathrm{Ker}}                                                        
\DeclareMathOperator{\Map}{\mathrm{Map}}                                                           
\DeclareMathOperator{\Mor}{\mathrm{Mor}}                                                           
\DeclareMathOperator{\Ob}{\mathrm{Ob}}                                                             
\DeclareMathOperator{\SDenominators}{\mathrm{SDen}}                                                
\DeclareMathOperator{\Source}{\mathrm{Source}}                                                     
\DeclareMathOperator{\Target}{\mathrm{Target}}                                                     
\DeclareMathOperator{\TDenominators}{\mathrm{TDen}}                                                
\DeclareMathOperator{\threearrowgraph}{\mathrm{AG}}                                                
\newcommand{\Bif}{\mathbf{Bif}}                                                                    
\newcommand{\Cat}{\mathbf{Cat}}                                                                    
\newcommand{\CatD}{\mathbf{CatD}}                                                                  
\newcommand{\categoricalequivalent}{\simeq}                                                        
\newcommand{\Cof}{\mathbf{Cof}}                                                                    
\newcommand{\CohomologyGroup}[1][]{\mathrm{H}^{#1}}                                                
\newcommand{\Com}{\mathrm{C}}                                                                      
\newcommand{\comp}{\circ}                                                                          
\newcommand{\DerivedCategory}{\mathrm{D}}                                                          
\newcommand{\directsum}{\oplus}                                                                    
\newcommand{\doublefrac}[3]{{{#1} \backslash {#2} \slash {#3}}}                                    
\newcommand{\emb}{\mathrm{emb}}                                                                    
\newcommand{\fgMod}{\mathbf{mod}}                                                                  
\newcommand{\Fib}{\mathbf{Fib}}                                                                    
\newcommand{\fibreprod}[3][]{\mathbin{{_{#2}^{}}\times_{#3}^{#1}}}                                 
\newcommand{\fractionequal}{\equiv}                                                                
\newcommand{\homotopic}{\sim}                                                                      
\newcommand{\HomotopyCategoryOfComplexes}{\mathrm{K}}                                              
\newcommand{\id}{\mathrm{id}}                                                                      
\newcommand{\inc}{\mathrm{inc}}                                                                    
\newcommand{\ini}[1][]{\mathrm{ini}^{#1}}                                                          
\newcommand{\initialobject}[1][]{\text{\textexclamdown}^{#1}}                                      
\newcommand{\ins}{\mathrm{ins}}                                                                    
\newcommand{\Integers}{\mathbb{Z}}                                                                 
\newcommand{\intersection}{\cap}                                                                   
\newcommand{\isomorphic}{\cong}                                                                    
\newcommand{\Isomorphisms}{\mathrm{Iso}}                                                           
\newcommand{\kdelta}{\updelta}                                                                     
\newcommand{\LocalisationFunctor}[1][]{\mathrm{loc}^{#1}}                                          
\newcommand{\map}{\rightarrow}                                                                     
\newcommand{\morphism}[1][]{\mathpalette\morchoice{#1}}                                            
\newcommand{\Naturals}{\mathbb{N}}                                                                 
\newcommand{\numeratorshape}{\upnu}                                                                
\newcommand{\op}{\mathrm{op}}                                                                      
\newcommand{\pr}{\mathrm{pr}}                                                                      
\newcommand{\quo}{\mathrm{quo}}                                                                    
\newcommand{\smallcoprod}{\amalg}                                                                  
\newcommand{\smallprod}{\mathbin{\Pi}}                                                             
\newcommand{\sourcedenominatorshape}{\upsigma}                                                     
\newcommand{\targetdenominatorshape}{\uptau}                                                       
\newcommand{\ter}[1][]{\mathrm{ter}^{#1}}                                                          
\newcommand{\terminalobject}[1][]{{!}^{#1}}                                                        
\newcommand{\threearrow}[4]{{#1} \leftarrow {#2} \rightarrow {#3} \leftarrow {#4}}                 
\newcommand{\threearrowshape}{\mathbf{\Theta}}                                                     
\newcommand{\ShiftFunctor}{\mathrm{T}}                                                             
\newcommand{\UFr}{\mathbf{UFr}}                                                                    
\newcommand{\UFrCat}{\mathbf{UFrCat}}                                                              
\newcommand{\union}{\cup}                                                                          

\tikzset{cross line/.style={preaction={-, line width=6pt, draw=white}}}
\tikzset{equality/.style={-, double}}
\tikzset{exists/.style={dotted}}

\tikzset{diagram/.style={matrix of math nodes, row sep=#1, column sep=#1, text height=1.6ex, text depth=0.45ex, inner sep=0pt, nodes={inner sep=0.3333em}}, diagram/.default=2.5em}
\tikzset{diagram without objects/.style={matrix of math nodes, row sep=#1, column sep=#1, nodes in empty cells, inner sep=0pt, nodes={inner sep=0.3333em}}, diagram without objects/.default=2.5em}
\tikzset{diagram mixed/.style={matrix of math nodes, row sep=#1, column sep=#1, nodes in empty cells, text height=1.6ex, text depth=0.45ex, inner sep=0pt, nodes={inner sep=0.3333em}}, diagram mixed/.default=2.5em}

\tikzset{den/.style={decoration={markings, mark=at position #1 with {\node[fill=white, inner sep=0.5pt, transform shape]{\(\scriptstyle{\approx}\)};}}, postaction=decorate}, den/.default=0.5}
\tikzset{sden/.style={decoration={markings, mark=at position #1 with {\draw[solid, fill=white] circle(0.4ex);}}, postaction=decorate}, sden/.default=0.5}
\tikzset{tden/.style={decoration={markings, mark=at position #1 with {\draw[solid, fill=white] (-0.2ex, -0.4ex) rectangle (0.2ex, 0.4ex);}}, postaction=decorate}, tden/.default=0.5}

\title{On the $3$-arrow calculus for homotopy categories}
\author{Sebastian Thomas}
\date{March 30, 2011}

\begin{document}

\maketitle

\renewcommand{\thefootnote}{\fnsymbol{footnote}}
\footnotetext[0]{Mathematics Subject Classification 2010: 18E35, 18G55, 18E30, 55U35. \\ This article has been published in condensed form in Homology, Homotopy and Applications \textbf{13}(1) (2011), pp.~89--119.}
\renewcommand{\thefootnote}{\arabic{footnote}}

\begin{abstract}
We develop a localisation theory for certain categories, yielding a \(3\)-arrow calculus: Every morphism in the localisation is represented by a diagram of length \(3\), and two such diagrams represent the same morphism if and only if they can be embedded in a \(3\)-by-\(3\) diagram in an appropriate way. The method we use to construct this localisation is similar to the Ore localisation for a \(2\)-arrow calculus; in particular, we do not have to use zigzags of arbitrary length. Applications include the localisation of an arbitrary Quillen model category with respect to its weak equivalences as well as the localisation of its full subcategories of cofibrant, fibrant and bifibrant objects, giving the homotopy category in all four cases. In contrast to the approach of \eigenname{Dwyer}, \eigenname{Hirschhorn}, \eigenname{Kan} and \eigenname{Smith}, the Quillen model category under consideration does not need to admit functorial factorisations. Moreover, it follows that the derived category of any abelian (or idempotent splitting exact) category admits a \(3\)-arrow calculus if we localise the category of complexes instead of its homotopy category.
\end{abstract}

\section{Introduction} \label{sec:introduction}

Localisations of categories occur in homological and homotopical algebra. Prominent examples are the construction of the derived category of an abelian category~\cite[ch.~II, \S 1, not.~1.1]{verdier:1963:categories_derivees} (as localisation of the homotopy category of complexes) and -- more generally -- localisations of Verdier triangulated categories with respect to thick subcategories~\cite[ch.~I, \S 2, d{\'e}f.~3-3]{verdier:1963:categories_derivees}, localisations of abelian categories with respect to thick subcategories~\cite[ch.~I, sec.~2]{serre:1953:groupes_d_homotopie_et_classes_de_groupes_abeliens}~\cite[sec.~1.11]{grothendieck:1957:sur_quelques_points_d_algebre_homologique} and the definition of the homotopy category of a Quillen model category~\cite[ch.~I, sec.~1, def.~6]{quillen:1967:homotopical_algebra}. Usually, the construction of the first three examples is done by a procedure known under the name of Ore localisation, which can only be applied in the special case where the denominator set, that is, the subset of morphisms to be formally inverted, fulfills some additional properties. Let us call such a special denominator set a classical denominator set for the moment. The basic ideas of this method have their historical origin in ring theory, in particular in the works of \eigenname{Ore}~\cite[sec.~2]{ore:1931:linear_equations_in_non-commutative_fields} and \eigenname{Asano}~\cite[Satz~1]{asano:1949:ueber_die_quotientenbildung_von_schiefringen}, while the categorical version comes from the Grothendieck school, see \eigenname{Verdier}~\cite[ch.~I, \S 2, sec.~3.2]{verdier:1963:categories_derivees} and \eigenname{Grothendieck} and \eigenname{Hartshorne}~\cite[ch.~I, \S 3, prop.~3.1]{hartshorne:1966:residues_and_duality}, based on the work of \eigenname{Serre}~\cite[ch.~I, sec.~2]{serre:1953:groupes_d_homotopie_et_classes_de_groupes_abeliens}. In contrast, the construction of the homotopy category of a Quillen model category is usually done by a formal construction working for arbitrary denominator sets, which is commonly called Gabriel-Zisman localisation (\footnote{To the authors knowledge, it first explicitly appeared in the monograph of \eigenname{Gabriel} and \eigenname{Zisman}~\cite[sec.~1.1]{gabriel_zisman:1967:calculus_of_fractions_and_homotopy_theory}. One can find earlier mentions, for example in~\cite[ch.~I, \S 3, rem., p.~29]{hartshorne:1966:residues_and_duality} and in~\cite[ch.~I, \S 2, n.~3, p.~17]{verdier:1963:categories_derivees}. In the latter source, one finds moreover a citation ``[C.G.G.]'', which might be the unpublished manuscript \booktitle{Cat{\'e}gories et foncteurs} of \eigenname{Chevalley}, \eigenname{Gabriel} and \eigenname{Grothendieck} occurring in the bibliography of~\cite{schiffmann:1962:theorie_elementaire_des_categories}.}). Of course, if the denominator set under consideration is classical, then the Ore localisation and the Gabriel-Zisman localisation are isomorphic since localisation of categories is defined by a universal property.

The advantage of Ore localisation is in the manageability of morphisms in the localisation: We suppose given a category \(\mathcal{C}\) and a denominator set \(D \subseteq \Mor \mathcal{C}\). The morphisms in the Gabriel-Zisman localisation are represented by zigzags
\[\begin{tikzpicture}[baseline=(m-1-1.base)]
  \matrix (m) [diagram mixed]{
    & & & \dots & & \\};
  \path[->, font=\scriptsize]
    (m-1-1) edge (m-1-2)
    (m-1-3) edge (m-1-4)
            edge[den] (m-1-2)
    (m-1-5) edge (m-1-6)
            edge[den] (m-1-4);
\end{tikzpicture}\]
of finite but arbitrary length, where the ``backward'' arrows are in \(D\). In contrast, if \(D\) is a classical denominator set, then every morphism in the Ore localisation is represented by a diagram
\[\begin{tikzpicture}[baseline=(m-1-1.base)]
  \matrix (m) [diagram]{
    {} & {} & {} \\};
  \path[->, font=\scriptsize]
    (m-1-1) edge (m-1-2)
    (m-1-3) edge[den] (m-1-2);
\end{tikzpicture}.\]
Furthermore, in the Gabriel-Zisman localisation one has, in general, no convenient criterion to decide whether two zigzags represent the same morphism in the localisation, while already from the construction of the Ore localisation it follows that two of these diagrams represent the same morphism if and only if they can be embedded as the top and the bottom row in a commutative diagram of the following form.
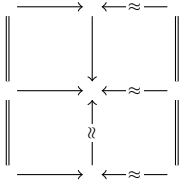
\[\begin{tikzpicture}[baseline=(m-3-1.base)]
  \matrix (m) [diagram without objects]{
    & & \\
    & & \\
    & & \\};
  \path[->, font=\scriptsize]
    (m-1-1) edge (m-1-2)
            edge[equality] (m-2-1)
    (m-1-2) edge (m-2-2)
    (m-1-3) edge[equality] (m-2-3)
            edge[den] (m-1-2)
    (m-2-1) edge (m-2-2)
    (m-2-3) edge[den] (m-2-2)
    (m-3-1) edge (m-3-2)
            edge[equality] (m-2-1)
    (m-3-2) edge[den] (m-2-2)
    (m-3-3) edge[den] (m-3-2)
            edge[equality] (m-2-3);
\end{tikzpicture}\]

Unfortunately, the set of weak equivalences in a Quillen model category \(\mathcal{M}\) is not a classical denominator set in general, and the homotopy category \(\HomotopyCategory \mathcal{M}\), that is, the localisation of \(\mathcal{M}\) with respect to its set of weak equivalences, does in general not fulfill a \(2\)-arrow calculus in the above sense. Instead, \eigenname{Dwyer}, \eigenname{Hirschhorn}, \eigenname{Kan} and \eigenname{Smith} developed in~\cite[sec.~10, sec.~36]{dwyer_hirschhorn_kan_smith:2004:homotopy_limit_functors_on_model_categories_and_homotopical_categories} a \(3\)-arrow calculus for the homotopy category of \(\mathcal{M}\), provided \(\mathcal{M}\) admits functorial factorisations (cf.~\cite[sec.~9.1, ax.~MC5]{dwyer_hirschhorn_kan_smith:2004:homotopy_limit_functors_on_model_categories_and_homotopical_categories}). That is, they showed that each morphism in \(\HomotopyCategory \mathcal{M}\) is represented by a diagram
\[\begin{tikzpicture}[baseline=(m-1-1.base)]
  \matrix (m) [diagram]{
    {} & {} & {} & {} \\};
  \path[->, font=\scriptsize]
    (m-1-2) edge (m-1-3)
            edge[den] (m-1-1)
    (m-1-4) edge[den] (m-1-3);
\end{tikzpicture},\]
and, moreover, that two of these diagrams represent the same morphism if and only if they can be embedded as the top and the bottom row in a commutative diagram of the following form.
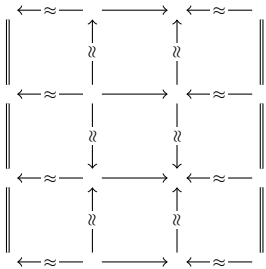
\[\begin{tikzpicture}[baseline=(m-4-1.base)]
  \matrix (m) [diagram without objects]{
    & & & \\
    & & & \\
    & & & \\
    & & & \\};
  \path[->, font=\scriptsize]
    (m-1-2) edge (m-1-3)
            edge[den] (m-1-1)
    (m-1-4) edge[den] (m-1-3)
    (m-2-1) edge[equality] (m-3-1)
            edge[equality] (m-1-1)
    (m-2-2) edge (m-2-3)
            edge[den] (m-3-2)
            edge[den] (m-2-1)
            edge[den] (m-1-2)
    (m-2-3) edge[den] (m-3-3)
            edge[den] (m-1-3)
    (m-2-4) edge[equality] (m-3-4)
            edge[den] (m-2-3)
            edge[equality] (m-1-4)
    (m-3-2) edge (m-3-3)
            edge[den] (m-3-1)
    (m-3-4) edge[den] (m-3-3)
    (m-4-1) edge[equality] (m-3-1)
    (m-4-2) edge (m-4-3)
            edge[den] (m-4-1)
            edge[den] (m-3-2)
    (m-4-3) edge[den] (m-3-3)
    (m-4-4) edge[den] (m-4-3)
            edge[equality] (m-3-4);
\end{tikzpicture}\]
To do this, they introduced the notion of a \emph{homotopical category admitting a \(3\)-arrow calculus}~\cite[sec.~33.1, 36.1]{dwyer_hirschhorn_kan_smith:2004:homotopy_limit_functors_on_model_categories_and_homotopical_categories} and developed a \(3\)-arrow calculus in this context~\cite[sec.~36.3]{dwyer_hirschhorn_kan_smith:2004:homotopy_limit_functors_on_model_categories_and_homotopical_categories}.

In this article, we introduce the concept of a uni-fractionable category, see definition~\ref{def:uni-fractionable_categories_and_their_morphisms}\ref{def:uni-fractionable_categories_and_their_morphisms:uni-fractionable_category}. Our main result is the construction of a localisation of a uni-fractionable category (with respect to its set of denominators) that satisfies a \(3\)-arrow calculus in the sense described above, see theorem~\ref{th:description_of_the_fraction_category}. In contrast to~\cite{dwyer_hirschhorn_kan_smith:2004:homotopy_limit_functors_on_model_categories_and_homotopical_categories}, we will not make use of the Gabriel-Zisman localisation. Instead, we will give an elementary ad hoc construction of a localisation of a uni-fractionable category, in the spirit of the Ore localisation for a \(2\)-arrow calculus. (\footnote{It is easy to show that every morphism in the Gabriel-Zisman localisation of a uni-fractionable category can be represented by a diagram of length \(3\) (cf.\ the definition of the composition in proposition~\ref{prop:welldefinedness_of_the_fraction_category}). However, the author does not know how to prove in that context that two of these diagrams represent the same morphism if and only if they can be embedded in a \(3\)-by-\(3\) diagram as above.})

Both in the approach of~\cite[sec.~36.1]{dwyer_hirschhorn_kan_smith:2004:homotopy_limit_functors_on_model_categories_and_homotopical_categories} and in our uni-fractionable categories, one has three distinguished kinds of morphisms, which, in our terminology, are called denominators, S-denominators and T-denominators. The denominators are the morphisms to be formally inverted, while the S- and T-denominators are particular denominators. The essential stipulations in~\cite[sec.~36.1]{dwyer_hirschhorn_kan_smith:2004:homotopy_limit_functors_on_model_categories_and_homotopical_categories} are that every denominator factors functorially into an S{\nbd}denominator followed by a T-denominator (\footnote{The S resp.\ the T should remind us of the fact that the S-denominator resp.\ the T-denominator in a factorisation has the same source resp.\ the same target as the factorised morphism.}) and that one has functorial Ore completions along S-denominators resp.\ T-denominators. For uni-fractionable categories, we omit the stipulations of functoriality; instead, we require the existence of weakly universal Ore completions along S-denominators resp.\ T-denominators.

The advantage of uni-fractionable categories is that functoriality of factorisations is not needed. On the one hand, this is convenient for applications. On the other hand, the theory developed here can be applied to arbitrary Quillen model categories. Moreover, it can also be applied to the full subcategories of the cofibrant, fibrant resp.\ bifibrant objects of a Quillen model category. As a consequence, all of them admit a \(3\)-arrow calculus.

Furthermore, a derivable category in the sense of \eigenname{Cisinski}~\cite[sec.~2.25]{cisinski:2010:categories_derivables} (\footnote{Also called an Anderson-Brown-Cisinski premodel category by \eigenname{R{\u{a}}dulescu-Banu}~\cite[def.~1.1.3]{radulescu-banu:2006:cofibrations_in_homotopy_theory}.}), which is a self-dual generalisation of a category of fibrant objects in the sense of \eigenname{K. Brown}~\cite[sec.~1]{brown:1974:abstract_homotopy_theory_and_generalized_sheaf_cohomology}, admits a \(3\)-arrow calculus, provided stronger variants of the factorisation axioms and the axioms which ensure stability of acyclic cofibrations under pushouts resp.\ of acyclic fibrations under pullbacks hold. For the relationship of \eigenname{Cisinski}'s approach with other axiom systems, see~\cite[sec.~2]{radulescu-banu:2006:cofibrations_in_homotopy_theory}.

A further example of a uni-fractionable category structure is provided by the category of complexes in an arbitrary abelian category, where the denominators are given by the quasi-isomorphisms, that is, by those morphisms inducing isomorphisms on the homology objects. To obtain the derived category, instead of localising the homotopy category of complexes, we may directly localise the category of complexes itself. The price to pay is that instead of a \(2\)-arrow calculus, we obtain a \(3\)-arrow calculus. Similarly for the derived category of an idempotent splitting exact category.

One feature of this \(3\)-arrow approach is its self-duality. This might be a reason why \(3\)-arrows occurred implicitly in \eigenname{Grothendieck}'s construction of a localisation of an abelian category with respect to a thick subcategory~\cite[sec.~1.11, p.~138]{grothendieck:1957:sur_quelques_points_d_algebre_homologique}, cf.\ example~\ref{ex:thick_subcategories_of_abelian_and_triangulated_categories_yield_uni-fractionable_categories}\ref{ex:thick_subcategories_of_abelian_and_triangulated_categories_yield_uni-fractionable_categories:abelian_category_with_factorised_structure}, although a \(2\)-arrow approach is of course sufficient.

\paragraph{Outline} \label{par:outline}

We recall in section~\ref{sec:preliminaries} some notions of localisation theory and indicate how quotients of (ordered) graphs with respect to so-called graph congruences can be constructed. In section~\ref{sec:uni-fractionable_categories}, uni-fractionable categories are introduced. Recall that the aim of this article is to construct a localisation of a uni-fractionable category with respect to its set of denominators. To this end, we proceed in two steps: In section~\ref{sec:the_3-arrow_graph}, we assign to a uni-fractionable category a certain graph, its \(3\)-arrow graph, and introduce a graph congruence on this graph. Then, in section~\ref{sec:the_fraction_category}, it turns out that the quotient graph has a canonically given category structure, and we will show that this category is a localisation of the uni-fractionable category we started with. Our main theorem~\ref{th:description_of_the_fraction_category} then gives a criterion on when two \(3\)-arrows represent the same morphism in the localisation. In section~\ref{sec:co_products_and_additive_uni-fractionable_categories}, we give a sufficient criterion for the localisation and the localisation functor being additive. Finally, in section~\ref{sec:applications}, we show how Quillen model categories, derivable categories (under additional conditions), complexes and some further classical examples fit into this framework. The example of complexes with entries in an idempotent splitting exact category is best understood when generalised; this requires a little theory of formal cones as provided in appendix~\ref{sec:formal_cones_in_exact_categories}.

\paragraph{Acknowledgements} \label{par:acknowledgements}

I thank \eigenname{Matthias K{\"u}nzer} for many useful discussions on this article, in particular for his ideas leading to example~\ref{ex:exact_categories_with_enough_formal_cones_with_respect_to_subcategories_closed_under_pure_short_exact_sequences}.

This article will be part of my forthcoming doctoral thesis. I thank the RWTH Aachen Graduiertenf{\"o}rderung for financial support.

\subsection*{Conventions and notations}

We use the following conventions and notations.

\begin{itemize}
\item The composite of morphisms \(f\colon X \map Y\) and \(g\colon Y \map Z\) is usually denoted by \(f g\colon X \map Z\). The composite of functors \(F\colon \mathcal{C} \map \mathcal{D}\) and \(G\colon \mathcal{D} \map \mathcal{E}\) is usually denoted by \(G \comp F\colon \mathcal{C} \map \mathcal{E}\).
\item Isomorphy of objects \(X\) and \(Y\) is denoted by \(X \isomorphic Y\). Equivalence between categories \(\mathcal{C}\) and \(\mathcal{D}\) is denoted by \(\mathcal{C} \categoricalequivalent \mathcal{D}\).
\item Given a category \(\mathcal{C}\) and objects \(X\) and \(Y\) in \(\mathcal{C}\), we write \({_{\mathcal{C}}}(X, Y)\) for the set of morphisms from \(X\) to \(Y\).
\item Given a subobject \(U\) of an object \(X\), we denote by \(\inc = \inc^U\colon U \map X\) the inclusion. Dually, given a quotient object \(Q\) of an object \(X\), we denote by \(\quo = \quo^Q\colon X \map Q\) the quotient morphism.
\item Given a coproduct \(C\) of \(X_1\) and \(X_2\), the embedding \(X_k \map C\) is denoted by \(\emb_k = \emb_k^{C}\) for \(k \in \{1, 2\}\). Given morphisms \(f_k\colon X_k \map Y\) for \(k \in \{1, 2\}\), the induced morphism \(C \map Y\) is denoted by \(\begin{smallpmatrix} f_1 \\ f_2 \end{smallpmatrix} = \begin{smallpmatrix} f_1 \\ f_2 \end{smallpmatrix}^{\!C}\). Dually for products: We write \(\pr_k = \pr_k^P\) for the projections of a product of \(X_1\) and \(X_2\) and \(\begin{smallpmatrix} f_1 & f_2 \end{smallpmatrix} = \begin{smallpmatrix} f_1 & f_2 \end{smallpmatrix}^P\) for the induced morphism \(Y \map P\) of morphisms \(f_k\colon Y \map X_k\), \(k \in \{1, 2\}\).
\item By a sum of objects \(X_1\) and \(X_2\), we understand an object \(S\) such that \(S\) carries the structure of a coproduct and a product of \(X_1\) and \(X_2\), and such that \(\emb_k \pr_l = \kdelta_{k, l}\) for \(k, l \in \{1, 2\}\), where \(\kdelta\) denotes the Kronecker delta.
\item Given an initial object \(I\), the unique morphism \(I \map X\) to an object \(X\) will be denoted by \(\ini = \ini_X = \ini[I]_X\). Dually, given a terminal object \(T\), the unique morphism \(X \map T\) from an object \(T\) will be denoted by \(\ter = \ter_X = \ter[T]_X\). Given a zero object \(N\), the unique morphism \(X \map Y\) that factors over \(N\) will be denoted by \(0\).
\item Given a category admitting finite coproducts and objects \(X_1\), \(X_2\), we denote by \(X_1 \smallcoprod X_2\) a chosen coproduct and by \(\initialobject\) a chosen initial object. Analogously, given morphisms \(f_k\colon X_k \map Y_k\) for \(k \in \{1, 2\}\), the coproduct of \(f_1\) and \(f_2\) is denoted by \(f_1 \smallcoprod f_2\). Analogously for finite products resp.\ finite sums, where we write \(X_1 \smallprod X_2\) and \(f_1 \smallprod f_2\) for a chosen product and \(\terminalobject\) for a chosen terminal object resp.\ \(X_1 \directsum X_2\) and \(f_1 \directsum f_2\) for a chosen sum and \(0\) for a chosen zero object.
\item Given a category admitting finite coproducts \(\mathcal{C}\) and a category \(\mathcal{D}\), we say that a functor \(F\colon \mathcal{C} \map \mathcal{D}\) preserves finite coproducts if \(F \initialobject\) is an initial object in \(\mathcal{D}\), and if, given \(X_1, X_2 \in \Ob \mathcal{C}\), the object \(F(X_1 \smallcoprod X_2)\) is a coproduct of \(F X_1\) and \(F X_2\), where the embeddings are given by \(\emb_1^{F(X_1 \smallcoprod X_2)} = F(\emb_1^{X_1 \smallcoprod X_2})\) and \(\emb_2^{F(X_1 \smallcoprod X_2)} = F(\emb_2^{X_1 \smallcoprod X_2})\). Dually for finite products and analogously for finite sums.
\item Given a category admitting kernels, we denote by \(\Kernel f\) a chosen kernel of a morphism \(f\). Given a category admitting cokernels, we denote by \(\Cokernel f\) a chosen cokernel of a morphism \(f\).
\item The opposite category of a category \(\mathcal{C}\) is denoted by \(\mathcal{C}^\op\).
\item By a weak pushout rectangle (resp.\ weak pullback rectangle) we understand a quadrangle having the universal property of a pushout rectangle (resp.\ pullback rectangle) except for the uniqueness of the induced morphism.
\item The category of complexes in an additive category \(\mathcal{A}\) is denoted by \(\Com(\mathcal{A})\), its homotopy category by \(\HomotopyCategoryOfComplexes(\mathcal{A})\). The derived category of an exact category \(\mathcal{E}\) is denoted by \(\DerivedCategory(\mathcal{E})\).
\item Arrows \(a\) and \(b\) in an (oriented) graph are called parallel if \(\Source a = \Source b\) and \(\Target a = \Target b\).
\item In an exact category \(\mathcal{E}\), the distinguished short exact sequences in \(\mathcal{E}\) will be called \newnotion{pure short exact sequences}. Likewise, the monomorphisms occurring in a pure short exact sequence are called \newnotion{pure monomorphisms}, and the epimorphisms occurring in a pure short exact sequence are called \newnotion{pure epimorphisms}.
\item We use the notations \(\Naturals = \{1, 2, 3, \dots\}\) and \(\Naturals_0 = \Naturals \cup \{0\}\).
\item Given integers \(a, b \in \Integers\), we write \([a, b] := \{z \in \Integers \mid a \leq z \leq b\}\) for the set of integers lying between \(a\) and \(b\).
\end{itemize}

\paragraph{A remark on Grothendieck universes} \label{par:a_remark_on_grothendieck_universes}

To avoid set-theoretical difficulties, we work with Grothendieck universes~\cite[exp.~I, sec.~0]{artin_grothendieck_verdier:1972:sga_4_1} in this article. In particular, every category has an object \emph{set} and a morphism \emph{set}. Given a Grothendieck universe \(\mathfrak{U}\), we say that a category \(\mathcal{C}\) is a \newnotion{\(\mathfrak{U}\)-category} if \(\Ob \mathcal{C}\) and \(\Mor \mathcal{C}\) are elements of~\(\mathfrak{U}\). The \newnotion{category of \(\mathfrak{U}\)-categories}, whose object set consists of all \(\mathfrak{U}\)-categories and whose morphism set consists of all functors between \(\mathfrak{U}\)-categories (and source, target, composition and identities given by ordinary source, target, composition of functors and the identity functors, respectively), will be denoted \(\Cat = \Cat_{(\mathfrak{U})}\).

\pagebreak 

\section{Preliminaries} \label{sec:preliminaries}

In this section, we give some preliminaries on localisations of categories and quotient graphs with respect to graph congruences.

\subsection*{Localisations of categories} \label{ssec:localisations_of_categories}

We suppose given a category \(\mathcal{C}\). A \newnotion{denominator set} in \(\mathcal{C}\) is a subset \(D \subseteq \Mor \mathcal{C}\). We will consider denominator sets with special properties later in this article, but at the moment, a denominator set \(D\) is just an arbitrary subset of \(\Mor \mathcal{C}\). Informally, it is a subset singled out with the ``intention of localising with respect to it'', in the following sense.

A \newnotion{localisation} of \(\mathcal{C}\) with respect to a denominator set \(D\) in \(\mathcal{C}\) consists of a category \(\mathcal{L}\) and a functor \(L\colon \mathcal{C} \map \mathcal{L}\) such that the following axioms hold.
\begin{itemize}
\item[(Inv)] \emph{Invertibility}. For all \(d \in D\), the morphism \(L d\) is invertible.
\item[(1-uni)] \emph{1-universality}. Given a category \(\mathcal{D}\) and a functor \(F\colon \mathcal{C} \map \mathcal{D}\) such that \(F d\) is invertible for all \(d \in D\), there exists a unique functor \(\hat F\colon \mathcal{L} \map \mathcal{D}\) with \(F = \hat F \comp L\).
\item[(2-uni)] \emph{2-universality}. We suppose given a category \(\mathcal{D}\) and functors \(F, G\colon \mathcal{C} \map \mathcal{D}\) such that \(F d\) and \(G d\) are invertible for all \(d \in D\), and we denote by \(\hat F\colon \mathcal{L} \map \mathcal{D}\) resp.\ \(\hat G\colon \mathcal{L} \map \mathcal{D}\) the unique functor with \(F = \hat F \comp L\) resp.\ \(G = \hat G \comp L\). Given a transformation \(\alpha\colon F \map G\), there exists a unique transformation \(\hat \alpha\colon \hat F \map \hat G\) such that \(\hat \alpha_{L X} = \alpha_X\) for all \(X \in \Ob \mathcal{C}\).
\end{itemize}
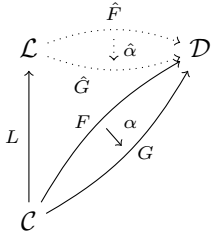
\[\begin{tikzpicture}[baseline=(m-2-1.base)]
  \matrix (m) [diagram=5.0em]{
    \mathcal{L} & \mathcal{D} \\
    \mathcal{C} & \\};
  \path[->, font=\scriptsize]
    (m-1-1) edge[exists, out=20, in=160] node[above] {\(\hat F\)} node (Fhat) {} (m-1-2)
            edge[exists, out=-20, in=-160] node[below, near start] {\(\hat G\)} node (Ghat) {} (m-1-2)
    (m-2-1) edge node[left] {\(L\)} (m-1-1)
            edge[out=60, in=-150] node[left] {\(F\)} node[sloped] (F) {} (m-1-2)
            edge[out=30, in=-120] node[right] {\(G\)} node[sloped] (G) {} (m-1-2);
  \path[->, font=\scriptsize]
    (Fhat) edge[exists] node[right] {\(\hat \alpha\)} (Ghat)
    (F) edge node[above right] {\(\alpha\)} (G);
\end{tikzpicture}\]
By abuse of notation, we refer to the localisation as well as to its underlying category just by \(\mathcal{L}\). The functor \(L\) is said to be the \newnotion{localisation functor} of the localisation \(\mathcal{L}\). Given a localisation \(\mathcal{L}\) of \(\mathcal{C}\) with respect to \(D\) with localisation functor \(L\colon \mathcal{C} \map \mathcal{L}\), we write \(\LocalisationFunctor = \LocalisationFunctor[\mathcal{L}] := L\).

\eigenname{Gabriel} and \eigenname{Zisman} have shown in~\cite[sec.~1.1]{gabriel_zisman:1967:calculus_of_fractions_and_homotopy_theory} that there exists a localisation of every category \(\mathcal{C}\) with respect to an arbitrary denominator set \(D\) in \(\mathcal{C}\). We will not make use of this result. Rather, given a uni-fractionable category, see definition~\ref{def:uni-fractionable_categories_and_their_morphisms}, we construct a localisation directly, cf.\ propositions~\ref{prop:welldefinedness_of_the_fraction_category} and~\ref{prop:universal_property_of_the_fraction_category}.

\subsection*{Saturatedness} \label{ssec:saturatedness}

We suppose given a category \(\mathcal{C}\), a denominator set \(D\) in \(\mathcal{C}\), and a localisation \(\mathcal{L}\) of \(\mathcal{C}\) with respect to \(D\). By definition of a localisation, \(\LocalisationFunctor(d)\) is invertible for every \(d \in D\). But in general, not every morphism \(f\) in \(\mathcal{C}\) for which \(\LocalisationFunctor(f)\) is invertible in \(\mathcal{L}\) has to be an element of \(D\). The denominator set \(D\) is said to be \newnotion{saturated} if \(f \in D\) for all \(f \in \Mor \mathcal{C}\) with \(\LocalisationFunctor(f)\) invertible in \(\mathcal{L}\). We use the following notions to indicate how far \(D\) is away from this property.

The denominator set \(D\) is said to be \newnotion{multiplicative} if it fulfills:
\begin{itemize}
\item[(Cat)] \emph{Multiplicativity}. For all \(d, e \in D\) with \(\Target d = \Source e\), their composite \(d e\) is in \(D\), and for every object \(X\) in \(\mathcal{C}\), the identity \(1_X\) is in \(D\).
\end{itemize}

The denominator set \(D\) is said to be \newnotion{semi-saturated} if it is multiplicative and fulfills:
\begin{itemize}
\item[(2\,of\,3)] \emph{2 out of 3 axiom}. We suppose given morphisms \(f\) and \(g\) in \(\mathcal{C}\) with \(\Target f = \Source g\). If two out of the morphisms \(f\), \(g\), \(f g\) are in \(D\), then so is the third.
\end{itemize}

Finally, the denominator set \(D\) is said to be \newnotion{weakly saturated} if it is multiplicative and fulfills:
\begin{itemize}
\item[(2\,of\,6)] \emph{2 out of 6 axiom}. We suppose given morphisms \(f\), \(g\), \(h\) in \(\mathcal{C}\) with \(\Target f = \Source g\) and \(\Target g = \Source h\). If \(f g, g h \in D\), then \(f, g, h, f g h \in D\).
\end{itemize}

Saturatedness implies weak saturatedness, weak saturatedness implies semi-saturatedness, and semi-saturated\-ness implies multiplicativity (the last impliciation holds by definition).

\subsection*{Categories with denominators} \label{ssec:categories_with_denominators}

A \newnotion{category with denominators} consists of a category \(\mathcal{C}\) together with a denominator set \(D\) in \(\mathcal{C}\). By abuse of notation, we refer to the category with denominators as well as to its underlying category just by \(\mathcal{C}\). The elements of \(D\) are called \newnotion{denominators} in \(\mathcal{C}\).

Given a category with denominators \(\mathcal{C}\) with set of denominators \(D\), we write \(\Denominators \mathcal{C} := D\). In diagrams, a denominator \(d\) in \(\mathcal{C}\) will usually be depicted as
\[\begin{tikzpicture}[baseline=(m-1-1.base)]
  \matrix (m) [diagram]{
    {} & {} \\};
  \path[->, font=\scriptsize]
    (m-1-1) edge[den] node[above] {\(d\)} (m-1-2);
\end{tikzpicture}.\]

We suppose given categories with denominators \(\mathcal{C}\) and \(\mathcal{D}\). A \newnotion{morphism of categories with denominators} from \(\mathcal{C}\) to \(\mathcal{D}\) is a functor \(F\colon \mathcal{C} \map \mathcal{D}\) that \newnotion{preserves denominators}, that is, such that \(F d\) is a denominator in \(\mathcal{D}\) for every denominator \(d\) in \(\mathcal{C}\).

We suppose given a Grothendieck universe \(\mathfrak{U}\). A category with denominators is said to be a \newnotion{\(\mathfrak{U}\)-category with denominators} if its underlying category is in \(\mathfrak{U}\). The category \(\CatD = \CatD_{(\mathfrak{U})}\) consisting of the set of \(\mathfrak{U}\)-categories with denominators as set of objects and the set of morphisms of categories with denominators between \(\mathfrak{U}\)-categories with denominators as set of morphisms (and categorical structure maps induced from \(\Cat_{(\mathfrak{U})}\)) is called the \newnotion{category of categories with denominators} (more precisely, the \newnotion{category of \(\mathfrak{U}\)-categories with denominators}).

Given a category with denominators \(\mathcal{C}\), a \newnotion{localisation} of \(\mathcal{C}\) is defined to be a localisation of (the underlying category of) \(\mathcal{C}\) with respect to its set of denominators \(\Denominators \mathcal{C}\).

A category with denominators \(\mathcal{C}\) is said to be \newnotion{multiplicative} resp.\ \newnotion{semi-saturated} resp.\ \newnotion{weakly saturated} resp.\ \newnotion{saturated} if its set of denominators \(\Denominators \mathcal{C}\) is multiplicative resp.\ semi-saturated resp.\ weakly saturated resp.\ saturated denominator set in the category \(\mathcal{C}\).

\subsection*{Graph congruences and quotient graphs} \label{ssec:graph_congruences_and_quotient_graphs}

We suppose given an (oriented) graph \(\mathcal{G}\). An equivalence relation \({\equiv}\) on \(\Arr \mathcal{G}\) is said to be a \newnotion{graph congruence} on \(\mathcal{G}\) if \(\Source a = \Source \tilde a\) and \(\Target a = \Target \tilde a\) for all \(a, \tilde a \in \Arr \mathcal{G}\) with \(a \equiv \tilde a\). Given a graph congruence \({\equiv}\) on \(\mathcal{G}\), the \newnotion{quotient graph} of \(\mathcal{G}\) with respect to \({\equiv}\) is the graph \(\mathcal{G} / {\equiv}\) with \(\Ob \mathcal{G} / {\equiv} := \Ob \mathcal{G}\), \(\Arr \mathcal{G} / {\equiv} := ({\Arr \mathcal{G}}) / {\equiv}\) and \(\Source{[a]_{\equiv}} := \Source a\), \(\Target{[a]_{\equiv}} := \Target a\) for \(a \in \Arr \mathcal{G}\). The graph morphism \(\quo = \quo^{\mathcal{G} / {\equiv}}\colon \mathcal{G} \map {\mathcal{G} / {\equiv}}\) given by \(\quo(X) := X\) and \(\quo(a) := [a]_{\equiv}\) is called the \newnotion{quotient graph morphism}.

The quotient graph of \(\mathcal{G}\) with respect to a graph congrunce \({\equiv}\) fulfills the following universal property. Given \(a, \tilde a \in \Arr \mathcal{G}\) with \(a \equiv \tilde a\), we have \(\quo(a) = \quo(\tilde a)\). For every graph \(\mathcal{H}\) and every graph morphism \(F\colon \mathcal{G} \map \mathcal{H}\) with \(F a = F \tilde a\) for \(a, \tilde a \in \Arr \mathcal{G}\) with \(a \equiv \tilde a\), there exists a unique graph morphism \(\overline{F}\colon \mathcal{G} / {\equiv} \map \mathcal{H}\) with \(F = \overline{F} \comp \quo\).
\[\begin{tikzpicture}[baseline=(m-2-1.base)]
  \matrix (m) [diagram]{
    \mathcal{G} & \mathcal{H} \\
    \mathcal{G} / {\equiv} & \\};
  \path[->, font=\scriptsize]
    (m-1-1) edge node[above] {\(F\)} (m-1-2)
            edge node[left] {\(\quo\)} (m-2-1)
    (m-2-1) edge[exists] node[right=1pt] {\(\overline{F}\)} (m-1-2);
\end{tikzpicture}\]

\section{Uni-fractionable categories} \label{sec:uni-fractionable_categories}

\begin{definition}[uni-fractionable categories, morphisms of uni-fractionable categories] \label{def:uni-fractionable_categories_and_their_morphisms} \
\begin{enumerate}
\item \label{def:uni-fractionable_categories_and_their_morphisms:uni-fractionable_category} A \newnotion{uni-fractionable category} (\footnote{There exists also the notion of a fractionable category, cf.\ the author's forthcoming doctoral thesis.}) consists of a semi-saturated category with denominators \(\mathcal{C}\) together with multiplicative subsets \(S, T \subseteq \Denominators \mathcal{C}\) such that the following axioms hold.
\begin{itemize}
\item[(WU)] \emph{Weakly universal Ore completions}. Given morphisms \(i\) and \(f\) in \(\mathcal{C}\) with \(i \in S\) and \(\Source i = \Source f\), there exists a weak pushout rectangle
\[\begin{tikzpicture}[baseline=(m-2-1.base)]
  \matrix (m) [diagram without objects]{
    & \\
    & \\};
  \path[->, font=\scriptsize]
    (m-1-1) edge node[above] {\(f'\)} (m-1-2)
    (m-2-1) edge node[above] {\(f\)} (m-2-2)
            edge node[left] {\(i\)} (m-1-1)
    (m-2-2) edge node[right] {\(i'\)} (m-1-2);
\end{tikzpicture}\]
in \(\mathcal{C}\) such that \(i' \in S\). Dually, given morphisms \(p\) and \(f\) in \(\mathcal{C}\) with \(p \in T\) and \(\Target p = \Target f\), there exists a weak pullback rectangle
\[\begin{tikzpicture}[baseline=(m-2-1.base)]
  \matrix (m) [diagram without objects]{
    & \\
    & \\};
  \path[->, font=\scriptsize]
    (m-1-1) edge node[above] {\(f'\)} (m-1-2)
            edge node[left] {\(p'\)} (m-2-1)
    (m-1-2) edge node[right] {\(p\)} (m-2-2)
    (m-2-1) edge node[above] {\(f\)} (m-2-2);
\end{tikzpicture}\]
in \(\mathcal{C}\) such that \(p' \in T\).
\item[(Fac)] \emph{Factorisations}. For every denominator \(d\) in \(\mathcal{C}\), there exist \(i \in S\) and \(p \in T\) with \(d = i p\).
\[\begin{tikzpicture}[baseline=(m-2-1.base)]
  \matrix (m) [diagram without objects=0.9em]{
    & & \\
    & & \\};
  \path[->, font=\scriptsize]
    (m-1-2) edge[exists] node[right=2pt] {\(p\)} (m-2-3)
    (m-2-1) edge node[below] {\(d\)} (m-2-3)
            edge[exists] node[left] {\(i\)} (m-1-2);
\end{tikzpicture}\]
\end{itemize}
By abuse of notation, we refer to the uni-fractionable category as well as to its underlying category with denominators just by~\(\mathcal{C}\). The elements of \(S\) are called \newnotion{S-denominators} in \(\mathcal{C}\), and the elements of \(T\) are called \newnotion{T-denominators} in \(\mathcal{C}\).

Given a uni-fractionable category \(\mathcal{C}\) with set of S-denominators \(S\) and set of T-denominators \(T\), we write \(\SDenominators \mathcal{C} := S\) and \(\TDenominators \mathcal{C} := T\). In diagrams, an S-denominator \(i\) resp.\ a T-denominator \(p\) in \(\mathcal{C}\) will usually be depicted as
\[\begin{tikzpicture}[baseline=(m-1-1.base)]
  \matrix (m) [diagram]{
    {} & {} \\};
  \path[->, font=\scriptsize]
    (m-1-1) edge[sden] node[above] {\(i\)} (m-1-2);
\end{tikzpicture}
\text{ resp.\ }
\begin{tikzpicture}[baseline=(m-1-1.base)]
  \matrix (m) [diagram]{
    {} & {} \\};
  \path[->, font=\scriptsize]
    (m-1-1) edge[tden] node[above] {\(p\)} (m-1-2);
\end{tikzpicture}.\]
\item \label{def:uni-fractionable_categories_and_their_morphisms:morphisms_of_uni-fractionable_categories} We suppose given uni-fractionable categories \(\mathcal{C}\) and \(\mathcal{D}\). A \newnotion{morphism of uni-fractionable categories} from \(\mathcal{C}\) to \(\mathcal{D}\) is a morphism of categories with denominators \(F\colon \mathcal{C} \map \mathcal{D}\) that \newnotion{preserves S-denominators} and \newnotion{T-denominators}, that is, such that \(F i\) is an S-denominator in \(\mathcal{D}\) for every S-denominator \(i\) in \(\mathcal{C}\) and such that \(F p\) is a T-denominator in \(\mathcal{D}\) for every T-denominator \(p\) in \(\mathcal{C}\).
\end{enumerate}
\end{definition}

Some examples of uni-fractionable categories can be found in section~\ref{sec:applications}.

Since the composite of composable morphisms of uni-fractionable categories is again a morphism of uni-fractionable categories and the identity functor on a uni-fractionable category is a morphism of uni-fractionable categories, we get a category of uni-fractionable categories:

\begin{definition}[uni-fractionable category in a Grothendieck universe] \label{def:uni-fractionable_category_in_a_grothendieck_universe}
We suppose given a Grothendieck universe \(\mathfrak{U}\). A uni-fractionable category \(\mathcal{C}\) is said to be a \newnotion{\(\mathfrak{U}\)-uni-fractionable category} if its underlying category with denominators is a category with denominators in \(\mathfrak{U}\).
\end{definition}

\begin{remark} \label{rem:uni-fractionable_categories_and_grothendieck_universes} \
\begin{enumerate}
\item \label{rem:uni-fractionable_categories_and_grothendieck_universes:being_element} We suppose given a Grothendieck universe \(\mathfrak{U}\). A uni-fractionable category \(\mathcal{C}\) is a \(\mathfrak{U}\)-uni-fractionable category if and only if it is an element of \(\mathfrak{U}\).
\item \label{rem:uni-fractionable_categories_and_grothendieck_universes:existence} For every uni-fractionable category \(\mathcal{C}\) there exists a Grothendieck universe \(\mathfrak{U}\) such that \(\mathcal{C}\) is in~\(\mathfrak{U}\).
\end{enumerate}
\end{remark}

\begin{definition}[category of uni-fractionable categories] \label{def:category_of_uni-fractionable_categories}
We suppose given a Grothendieck universe \(\mathfrak{U}\).
\begin{enumerate}
\item The category \(\UFrCat = \UFrCat_{(\mathfrak{U})}\) consisting of the set of \(\mathfrak{U}\)-uni-fractionable categories as set of objects and the set of morphisms of uni-fractionable categories between \(\mathfrak{U}\)-uni-fractionable categories as set of morphisms (and categorical structure maps induced from \(\CatD_{(\mathfrak{U})}\)) is called the \newnotion{category of uni-fractionable categories} (more precisely, the \newnotion{category of \(\mathfrak{U}\)-uni-fractionable categories}).
\item We denote by \(\UFr(\CatD_{(\mathfrak{U})})\) the full subcategory of \(\CatD_{(\mathfrak{U})}\) with
\begin{align*}
\Ob \UFr(\CatD_{(\mathfrak{U})}) & = \{\mathcal{C} \in \Ob \CatD_{(\mathfrak{U})} \mid \text{there exist \(S, T \subseteq \Denominators \mathcal{C}\) such that \(\mathcal{C}\) becomes a} \\
& \qquad \text{uni-fractionable category with \(\SDenominators \mathcal{C} = S\) and \(\TDenominators \mathcal{C} = T\)}\},
\end{align*}
the \newnotion{category of categories with denominators admitting the structure of a uni-fractionable category} (more precisely, the \newnotion{category of \(\mathfrak{U}\)-categories with denominators admitting the structure of a uni-fractionable category}).
\end{enumerate}
\end{definition}

\section{\texorpdfstring{The $3$-arrow graph}{The 3-arrow graph}} \label{sec:the_3-arrow_graph}

We want to construct a localisation \(\FractionCategory \mathcal{C}\) of a uni-fractionable category \(\mathcal{C}\) (with respect to its set of denominators \(\Denominators \mathcal{C}\)). To this end, we begin in this section by introducing its \(3\)-arrow graph \(\threearrowgraph{\mathcal{C}}\) and a graph congruence \({\fractionequal}\) on \(\threearrowgraph{\mathcal{C}}\).

In this section, we suppose given a uni-fractionable category \(\mathcal{C}\).

\begin{definition}[\(3\)-arrow shape] \label{def:3-arrow_shape}
The graph
\[\begin{tikzpicture}[baseline=(m-1-1.base)]
  \matrix (m) [diagram]{
    0 & 1 & 2 & 3 \\};
  \path[->, font=\scriptsize]
    (m-1-2) edge node[above] {\(\numeratorshape\)} (m-1-3)
            edge node[above] {\(\targetdenominatorshape\)} (m-1-1)
    (m-1-4) edge node[above] {\(\sourcedenominatorshape\)} (m-1-3);
\end{tikzpicture}\]
is said to be the \newnotion{\(3\)-arrow shape} and will be denoted by \(\threearrowshape\).
\end{definition}

Recall that a \newnotion{diagram} of \newnotion{shape} \(\threearrowshape\) in \(\mathcal{C}\) is just a graph morphism \(D\colon \threearrowshape \map \mathcal{C}\). Given a diagram \(D\) of shape \(\threearrowshape\) in \(\mathcal{C}\), we write \(D_i := D(i)\) for \(i \in \Ob \threearrowshape\) and \(D_a := D(a)\) for \(a \in \Arr \threearrowshape\). Given diagrams \(D\) and \(E\), a \newnotion{diagram morphism} from \(D\) to \(E\) is a family \(f = (f_i)_{i \in \Ob \threearrowshape}\) in \(\Mor \mathcal{C}\) with \(D_a f_j = f_i E_a\) for all arrows \(a\colon i \map j\) in \(\threearrowshape\). The category consisting of diagrams of shape \(\threearrowshape\) in \(\mathcal{C}\) as objects and diagram morphisms between those diagrams as morphisms will be denoted by \(\mathcal{C}^{\threearrowshape}\). (\footnote{By the adjunction ``free category on a graph -- underlying graph of a category'', diagrams of shape \(\threearrowshape\) in \(\mathcal{C}\) correspond in a unique way to functors from the free category on \(\threearrowshape\) to \(\mathcal{C}\), and diagram morphisms correspond to transformations.})

\begin{definition}[\(3\)-arrow graph] \label{def:3-arrow_graph}
The \newnotion{\(3\)-arrow graph} of \(\mathcal{C}\) is defined to be the graph \(\threearrowgraph{\mathcal{C}}\) with object set
\[\Ob \threearrowgraph{\mathcal{C}} := \Ob \mathcal{C}\]
and arrow set
\[\Arr \threearrowgraph{\mathcal{C}} := \{A \in \Ob \mathcal{C}^\threearrowshape \mid A_{\sourcedenominatorshape}, A_{\targetdenominatorshape} \in \Denominators \mathcal{C}\}.\]
The source resp.\ the target of \(A \in \Arr \threearrowgraph{\mathcal{C}}\) are defined by \(\Source A := A_0\) resp.\ \(\Target A := A_3\).

An arrow \(A\) in \(\threearrowgraph{\mathcal{C}}\) is called a \newnotion{\(3\)-arrow} in \(\mathcal{C}\). Given a denominator \(b\colon \tilde X \map X\), a morphism \(f\colon \tilde X \map \tilde Y\) and a denominator \(a\colon Y \map \tilde Y\) in \(\mathcal{C}\), we abuse notation and denote the unique \(3\)-arrow \(A\) with \(A_{\targetdenominatorshape} = b\), \(A_{\numeratorshape} = f\), \(A_{\sourcedenominatorshape} = a\) by \((b, f, a) := A\). Moreover, we use the notation \((b, f, a)\colon \threearrow{X}{\tilde X}{\tilde Y}{Y}\). 
\[\begin{tikzpicture}[baseline=(m-1-1.base)]
  \matrix (m) [diagram]{
    X & \tilde X & \tilde Y & Y \\};
  \path[->, font=\scriptsize]
    (m-1-2) edge node[above] {\(f\)} (m-1-3)
            edge[den] node[above] {\(b\)} (m-1-1)
    (m-1-4) edge[den] node[above] {\(a\)} (m-1-3);
\end{tikzpicture}\]
\end{definition}

\begin{remark} \label{rem:universe_of_the_3-arrow_graph}
We suppose given a Grothendieck universe \(\mathfrak{U}\) such that \(\threearrowshape\) is in \(\mathfrak{U}\). If \(\mathcal{C}\) is in \(\mathfrak{U}\), then its \(3\)-arrow graph \(\threearrowgraph{\mathcal{C}}\) is in \(\mathfrak{U}\).
\end{remark}
\begin{proof}
We suppose that \(\mathcal{C}\) is in \(\mathfrak{U}\). Then \(\Ob \mathcal{C}\) and \(\Mor \mathcal{C}\) are in \(\mathfrak{U}\) and hence \(\Map(\Arr \threearrowshape, \Mor \mathcal{C})\) is in \(\mathfrak{U}\). But then \(\Ob \threearrowgraph{\mathcal{C}}\) and \(\Arr \threearrowgraph{\mathcal{C}}\) are in \(\mathfrak{U}\), that is, \(\threearrowgraph{\mathcal{C}}\) is in \(\mathfrak{U}\).
\end{proof}

Our next step will be the introduction of an equivalence relation on the arrow set of the \(3\)-arrow graph.

\begin{definition}[fraction equality] \label{def:fraction_equality_relation_on_the_3-arrow_graph}
The equivalence relation \(\fractionequal\) on \(\Arr \threearrowgraph{\mathcal{C}}\) is defined to be generated by the following relation on \(\Arr \threearrowgraph{\mathcal{C}}\): Given \((b, f, a) \in \Arr \threearrowgraph{\mathcal{C}}\) and \(c \in \Mor \mathcal{C}\) with \(a c \in \Denominators \mathcal{C}\), the \(3\)-arrow \((b, f, a)\) is in relation to the \(3\)-arrow \((b, f c, a c)\); and given \((b, f, a) \in \Arr \threearrowgraph{\mathcal{C}}\) and \(c \in \Mor \mathcal{C}\) with \(c b \in \Denominators \mathcal{C}\), the \(3\)-arrow \((b, f, a)\) is in relation to the \(3\)-arrow \((c b, c f, a)\).
\[\begin{tikzpicture}[baseline=(m-2-1.base)]
  \matrix (m) [diagram without objects]{
    & & & \\
    & & & \\};
  \path[->, font=\scriptsize]
    (m-1-1) edge[equality] (m-2-1)
    (m-1-2) edge node[above] {\(f\)} (m-1-3)
            edge[equality] (m-2-2)
            edge[den] node[above] {\(b\)} (m-1-1)
    (m-1-3) edge node[right] {\(c\)} (m-2-3)
    (m-1-4) edge[equality] (m-2-4)
            edge[den] node[above] {\(a\)} (m-1-3)
    (m-2-2) edge node[above] {\(f c\)} (m-2-3)
            edge[den] node[above] {\(b\)} (m-2-1)
    (m-2-4) edge[den] node[above] {\(a c\)} (m-2-3);
\end{tikzpicture}
\qquad
\begin{tikzpicture}[baseline=(m-2-1.base)]
  \matrix (m) [diagram without objects]{
    & & & \\
    & & & \\};
  \path[->, font=\scriptsize]
    (m-1-2) edge node[above] {\(f\)} (m-1-3)
            edge[den] node[above] {\(b\)} (m-1-1)
    (m-1-4) edge[den] node[above] {\(a\)} (m-1-3)
    (m-2-1) edge[equality] (m-1-1)
    (m-2-2) edge node[above] {\(c f\)} (m-2-3)
            edge[den] node[above] {\(c b\)} (m-2-1)
            edge node[right] {\(c\)} (m-1-2)
    (m-2-3) edge[equality] (m-1-3)
    (m-2-4) edge[den] node[above] {\(a\)} (m-2-3)
            edge[equality] (m-1-4);
\end{tikzpicture}\]
Given \((b, f, a), (\tilde b, \tilde f, \tilde a) \in \Arr \threearrowgraph{\mathcal{C}}\) with \((b, f, a) \fractionequal (\tilde b, \tilde f, \tilde a)\), we say that \((b, f, a)\) and \((\tilde b, \tilde f, \tilde a)\) are \newnotion{fraction equal}.
\end{definition}

In practice, it is sometimes convenient to work with different generating sets for fraction equality. These are stated in the following remark.

\pagebreak 

\begin{remark} \label{rem:other_generating_sets_for_the_fraction_equality_relation_on_the_3-arrow_graph} \
\begin{enumerate}
\item \label{rem:other_generating_sets_for_the_fraction_equality_relation_on_the_3-arrow_graph:different_directions} The fraction equality relation \(\fractionequal\) on \(\Arr \threearrowgraph{\mathcal{C}}\) is generated by the following relation: Given \((b, f, a) \in \Arr \threearrowgraph{\mathcal{C}}\) and \(c, c' \in \Mor \mathcal{C}\) with \(a c, c' b \in \Denominators \mathcal{C}\), the \(3\)-arrow \((b, f, a)\) is in relation to the \(3\)-arrow \((c' b, c' f c, a c)\).
\[\begin{tikzpicture}[baseline=(m-2-1.base)]
  \matrix (m) [diagram without objects]{
    & & & \\
    & & & \\};
  \path[->, font=\scriptsize]
    (m-1-2) edge node[above] {\(f\)} (m-1-3)
            edge[den] node[above] {\(b\)} (m-1-1)
    (m-1-3) edge node[right] {\(c\)} (m-2-3)
    (m-1-4) edge[equality] (m-2-4)
            edge[den] node[above] {\(a\)} (m-1-3)
    (m-2-1) edge[equality] (m-1-1)
    (m-2-2) edge node[above] {\(c' f c\)} (m-2-3)
            edge[den] node[above] {\(c' b\)} (m-2-1)
            edge node[right] {\(c'\)} (m-1-2)
    (m-2-4) edge[den] node[above] {\(a c\)} (m-2-3);
\end{tikzpicture}\]
\item \label{rem:other_generating_sets_for_the_fraction_equality_relation_on_the_3-arrow_graph:equal_directions} The fraction equality relation \(\fractionequal\) on \(\Arr \threearrowgraph{\mathcal{C}}\) is generated by the following relation: Given \((b, f, a), (\tilde b, \tilde f, \tilde a)\) 
\(\in \Arr \threearrowgraph{\mathcal{C}}\), the \(3\)-arrow \((b, f, a)\) is in relation to the \(3\)-arrow \((\tilde b, \tilde f, \tilde a)\) if there exist \(c, c' \in \Mor \mathcal{C}\) with \(b = c' \tilde b\), \(f c = c' \tilde f\), \(a c = \tilde a\).
\[\begin{tikzpicture}[baseline=(m-2-1.base)]
  \matrix (m) [diagram without objects]{
    & & & \\
    & & & \\};
  \path[->, font=\scriptsize]
    (m-1-1) edge[equality] (m-2-1)
    (m-1-2) edge node[above] {\(f\)} (m-1-3)
            edge node[right] {\(c'\)} (m-2-2)
            edge[den] node[above] {\(b\)} (m-1-1)
    (m-1-3) edge node[right] {\(c\)} (m-2-3)
    (m-1-4) edge[equality] (m-2-4)
            edge[den] node[above] {\(a\)} (m-1-3)
    (m-2-2) edge node[above] {\(\tilde f\)} (m-2-3)
            edge[den] node[above] {\(\tilde b\)} (m-2-1)
    (m-2-4) edge[den] node[above] {\(\tilde a\)} (m-2-3);
\end{tikzpicture}\]
\end{enumerate}
\end{remark}

As \(\mathcal{C}\) is semi-saturated, the morphisms \(c\) and \(c'\) in definition~\ref{def:fraction_equality_relation_on_the_3-arrow_graph} and remark~\ref{rem:other_generating_sets_for_the_fraction_equality_relation_on_the_3-arrow_graph} are automatically denominators in \(\mathcal{C}\).

\begin{remark} \label{rem:invertibility_of_3-arrows_is_preserved_by_fraction_equality}
We suppose given \(3\)-arrows \((b, f, a)\) and \((\tilde b, \tilde f, \tilde a)\) in \(\mathcal{C}\). If \((b, f, a) \fractionequal (\tilde b, \tilde f, \tilde a)\), then \(f\) is a denominator in \(\mathcal{C}\) if and only if \(\tilde f\) is a denominator in \(\mathcal{C}\).
\end{remark}
\begin{proof}
This follows by the definition of fraction equality~\ref{def:fraction_equality_relation_on_the_3-arrow_graph} and by the semi-saturatedness of \(\mathcal{C}\).
\end{proof}

Before we study a further property of the fraction equality relation, we will show that fraction equality respects the graph structure on the \(3\)-arrow graph.

\begin{remark} \label{rem:fraction_equality_on_the_3-arrow_graph_is_a_congruence}
The fraction equality relation \(\fractionequal\) on \(\Arr \threearrowgraph{\mathcal{C}}\) defines a graph congruence on \(\threearrowgraph{\mathcal{C}}\). In particular, the quotient graph \((\threearrowgraph{\mathcal{C}}) / {\fractionequal}\) is defined.
\end{remark}
\begin{proof}
For \((b, f, a) \in \Arr \threearrowgraph{\mathcal{C}}\), \(c, c' \in \Mor \mathcal{C}\) with \(a c, c' b \in \Denominators \mathcal{C}\), we have
\begin{align*}
& \Source{(c' b, c' f c, a c)} = \Target(c' b) = \Target b = \Source{(b, f, a)} \text{ and} \\
& \Target{(c' b, c' f c, a c)} = \Source(a c) = \Source a = \Target{(b, f, a)}.
\end{align*}
Thus the assertion follows from remark~\ref{rem:other_generating_sets_for_the_fraction_equality_relation_on_the_3-arrow_graph}\ref{rem:other_generating_sets_for_the_fraction_equality_relation_on_the_3-arrow_graph:different_directions}.
\end{proof}

\begin{definition}[double fraction] \label{def:double_fraction}
Given a \(3\)-arrow \((b, f, a)\) in \(\mathcal{C}\), its equivalence class in the quotient graph \((\threearrowgraph \mathcal{C}) / {\fractionequal}\) is denoted by \(\doublefrac{b}{f}{a} := [(b, f, a)]_{\fractionequal}\) and is said to be the \newnotion{double fraction} of \((b, f, a)\).
\end{definition}

Now we will present a certain reduced form for \(3\)-arrows. We will see that every \(3\)-arrow is fraction equal to such a reduced form.

\begin{definition}[normal \(3\)-arrows] \label{def:normal_3-arrows}
A \(3\)-arrow \((p, f, i)\) in \(\mathcal{C}\) is said to be \newnotion{normal} if \(i\) is an S-denominator and \(p\) is a T-denominator in \(\mathcal{C}\).
\[\begin{tikzpicture}[baseline=(m-1-1.base)]
  \matrix (m) [diagram without objects]{
    & & & \\};
  \path[->, font=\scriptsize]
    (m-1-2) edge node[above] {\(f\)} (m-1-3)
            edge[tden] node[above] {\(p\)} (m-1-1)
    (m-1-4) edge[sden] node[above] {\(i\)} (m-1-3);
\end{tikzpicture}\]
\end{definition}

The following lemma and its proof is (essentially) taken from~\cite[sec.~36.5]{dwyer_hirschhorn_kan_smith:2004:homotopy_limit_functors_on_model_categories_and_homotopical_categories}.

\begin{lemma}[normalisation lemma] \label{lem:normalisation_lemma}
Every \(3\)-arrow in \(\mathcal{C}\) is fraction equal to a normal \(3\)-arrow in \(\mathcal{C}\).
\end{lemma}
\begin{proof}
We suppose given an arbitrary \(3\)-arrow \((b, f, a)\) in \(\mathcal{C}\). There exist an S-denominator \(i\) and a T-denominator \(p\) in \(\mathcal{C}\) with \(b = i p\), and there exist an S-denominator \(i'\) and a morphism \(f'\) in \(\mathcal{C}\) with \(i f' = f i'\). By multiplicativity, \(a i'\) is a denominator in \(\mathcal{C}\). Thus there exist an S-denominator \(j\) and a T-denominator \(q\) in \(\mathcal{C}\) with \(a i' = j q\), and there exist a T-denominator \(q'\) and a morphism \(f''\) in \(\mathcal{C}\) with \(f'' q = q' f'\). By multiplicativity, \(q' p\) is a T-denominator.
\[\begin{tikzpicture}[baseline=(m-3-1.base)]
  \matrix (m) [diagram without objects]{
    & & & \\
    & & & \\
    & & & \\};
  \path[->, font=\scriptsize]
    (m-1-1) edge[equality] (m-2-1)
    (m-1-2) edge node[above] {\(f\)} (m-1-3)
            edge[sden] node[right] {\(i\)} (m-2-2)
            edge[den] node[above] {\(b\)} (m-1-1)
    (m-1-3) edge[sden] node[right] {\(i'\)} (m-2-3)
    (m-1-4) edge[equality] (m-2-4)
            edge[den] node[above] {\(a\)} (m-1-3)
    (m-2-2) edge node[above] {\(f'\)} (m-2-3)
            edge[tden] node[above] {\(p\)} (m-2-1)
    (m-2-4) edge[den] node[above] {\(a i'\)} (m-2-3)
    (m-3-1) edge[equality] (m-2-1)
    (m-3-2) edge node[above] {\(f''\)} (m-3-3)
            edge[tden] node[above] {\(q' p\)} (m-3-1)
            edge[tden] node[right] {\(q'\)} (m-2-2)
    (m-3-3) edge[tden] node[right] {\(q\)} (m-2-3)
    (m-3-4) edge[sden] node[above] {\(j\)} (m-3-3)
            edge[equality] (m-2-4);
\end{tikzpicture}\]

Altogether, \((b, f, a) \fractionequal (p, f', a i') \fractionequal (q' p, f'', j)\), and since \(j\) is an S-denominator and \(q' p\) is a T-denominator, the \(3\)-arrow \((q' p, f'', j)\) is normal.
\end{proof}

\begin{corollary} \label{cor:3-arrows_with_common_denominators}
We suppose given a uni-fractionable category \(\mathcal{C}\) and \(3\)-arrows \((b_1, f_1, a_1)\) and \((b_2, f_2, a_2)\) in~\(\mathcal{C}\).
\begin{enumerate}
\item \label{cor:3-arrows_with_common_denominators:equal_source} If \(\Source{(b_1, f_1, a_1)} = \Source{(b_2, f_2, a_2)}\), then there exist normal \(3\)-arrows \((p, \tilde f_1, i_1)\) and \((p, \tilde f_2, i_2)\) in \(\mathcal{C}\) with \((b_1, f_1, a_1) \fractionequal (p, \tilde f_1, i_1)\) and \((b_2, f_2, a_2) \fractionequal (p, \tilde f_2, i_2)\).
\item \label{cor:3-arrows_with_common_denominators:equal_target} If \(\Target{(b_1, f_1, a_1)} = \Target{(b_2, f_2, a_2)}\), then there exist normal \(3\)-arrows \((p_1, \tilde f_1, i)\) and \((p_2, \tilde f_2, i)\) in \(\mathcal{C}\) with \((b_1, f_1, a_1) \fractionequal (p_1, \tilde f_1, i)\) and \((b_2, f_2, a_2) \fractionequal (p_2, \tilde f_2, i)\).
\item \label{cor:3-arrows_with_common_denominators:parallel} If \((b_1, f_1, a_1)\) and \((b_2, f_2, a_2)\) are parallel, then there exist normal \(3\)-arrows \((p, \tilde f_1, i)\) and \((p, \tilde f_2, i)\) in \(\mathcal{C}\) with \((b_1, f_1, a_1) \fractionequal (p, \tilde f_1, i)\) and \((b_2, f_2, a_2) \fractionequal (p, \tilde f_2, i)\).
\end{enumerate}
\end{corollary}
\begin{proof}
By the normalisation lemma~\ref{lem:normalisation_lemma}, there exist normal \(3\)-arrows \((p_k, g_k, i_k)\) in \(\mathcal{C}\) with \((b_k, f_k, a_k) \fractionequal (p_k, g_k, i_k)\) for \(k \in \{1, 2\}\). In particular, we have \(\Source{(p_k, g_k, i_k)} = \Source{(b_k, f_k, a_k)}\) and \(\Target{(p_k, g_k, i_k)} = \Target{(b_k, f_k, a_k)}\) for \(k \in \{1, 2\}\). Hence \(\Source{(b_1, f_1, a_1)} = \Source{(b_2, f_2, a_2)}\) implies that \(\Source{(p_1, g_1, i_1)} = \Source{(p_2, g_2, i_2)}\) and \(\Target{(b_1, f_1, a_1)} = \Target{(b_2, f_2, a_2)}\) implies that \(\Target{(p_1, g_1, i_1)} = \Target{(p_2, g_2, i_2)}\).
\begin{enumerate}
\item There exist a T-denominator \(p_2'\) and a morphism \(p_1'\) in \(\mathcal{C}\) with \(p_2' p_1 = p_1' p_2\). We define \(p := p_2' p_1 = p_1' p_2\), \(\tilde f_1 := p_2' g_1\), \(\tilde f_2 := p_1' g_2\).
\[\begin{tikzpicture}[baseline=(m-3-1.base)]
  \matrix (m) [diagram without objects]{
    & & & \\
    & & & \\
    & & & \\};
  \path[->, font=\scriptsize]
    (m-1-1) edge[equality] (m-2-1)
    (m-1-2) edge node[above] {\(g_1\)} (m-1-3)
            edge[tden] node[above] {\(p_1\)} (m-1-1)
    (m-1-4) edge[sden] node[above] {\(i_1\)} (m-1-3)
    (m-2-2) edge node[left] {\(p_1'\)} (m-3-2)
            edge[tden] node[above] {\(p\)} (m-2-1)
            edge[tden] node[left] {\(p_2'\)} (m-1-2)
    (m-2-2) edge node[right] {\(\tilde f_1\)} (m-1-3)
    (m-2-2) edge node[right=1pt] {\(\tilde f_2\)} (m-3-3)
    (m-3-1) edge[equality] (m-2-1)
    (m-3-2) edge node[above] {\(g_2\)} (m-3-3)
            edge[tden] node[above] {\(p_2\)} (m-3-1)
    (m-3-4) edge[sden] node[above] {\(i_2\)} (m-3-3);
\end{tikzpicture}\]
By multiplicativity, \(p = p_2' p_1\) is a T-denominator in \(\mathcal{C}\), and we have
\begin{align*}
& (p, \tilde f_1, i_1) = (p_2' p_1, p_2' g_1, i_1) \fractionequal (p_1, g_1, i_1) \fractionequal (b_1, f_1, a_1) \text{ and} \\
& (p, \tilde f_2, i_2) = (p_1' p_2, p_1' g_2, i_2) \fractionequal (p_2, g_2, i_2) \fractionequal (b_2, f_2, a_2).
\end{align*}
\item This is dual to~\ref{cor:3-arrows_with_common_denominators:equal_source}.
\item There exist a T-denominator \(p_2'\) and a morphism \(p_1'\) in \(\mathcal{C}\) with \(p_2' p_1 = p_1' p_2\), and there exist an S{\nbd}denominator \(i_1'\) and a morphism \(i_2'\) in \(\mathcal{C}\) with \(i_1 i_2' = i_2 i_1'\). We define \(p := p_2' p_1 = p_1' p_2\), \(i := i_1 i_2' = i_2 i_1'\), \(\tilde f_1 := p_2' g_1 i_2'\), \(\tilde f_2 := p_1' g_2 i_1'\).
\[\begin{tikzpicture}[baseline=(m-3-1.base)]
  \matrix (m) [diagram without objects]{
    & & & \\
    & & & \\
    & & & \\};
  \path[->, font=\scriptsize]
    (m-1-1) edge[equality] (m-2-1)
    (m-1-2) edge node[above] {\(g_1\)} (m-1-3)
            edge[tden] node[above] {\(p_1\)} (m-1-1)
    (m-1-3) edge node[right] {\(i_2'\)} (m-2-3)
    (m-1-4) edge[equality] (m-2-4)
            edge[sden] node[above] {\(i_1\)} (m-1-3)
    (m-2-2) edge node[right] {\(p_1'\)} (m-3-2)
            edge[tden] node[above] {\(p\)} (m-2-1)
            edge[tden] node[right] {\(p_2'\)} (m-1-2)
    (m-2-2.35) edge node[above] {\(\tilde f_1\)} (m-2-3.145)
    (m-2-2.-35) edge node[below] {\(\tilde f_2\)} (m-2-3.-145)
    (m-2-4) edge[sden] node[above] {\(i\)} (m-2-3)
    (m-3-1) edge[equality] (m-2-1)
    (m-3-2) edge node[above] {\(g_2\)} (m-3-3)
            edge[tden] node[above] {\(p_2\)} (m-3-1)
    (m-3-3) edge[sden] node[right] {\(i_1'\)} (m-2-3)
    (m-3-4) edge[sden] node[above] {\(i_2\)} (m-3-3)
            edge[equality] (m-2-4);
\end{tikzpicture}\]
The assertion now follows as in~\ref{cor:3-arrows_with_common_denominators:equal_source} and~\ref{cor:3-arrows_with_common_denominators:equal_target}. \qedhere
\end{enumerate}
\end{proof}

\section{The fraction category} \label{sec:the_fraction_category}

In this section, our main theorem~\ref{th:description_of_the_fraction_category} will be proven. We begin by constructing a localisation of a uni-fractionable category \(\mathcal{C}\) (with respect to its set of denominators \(\Denominators \mathcal{C}\)), see proposition~\ref{prop:welldefinedness_of_the_fraction_category} and proposition~\ref{prop:universal_property_of_the_fraction_category}. To this end, we consider the quotient graph \((\threearrowgraph{\mathcal{C}}) / {\fractionequal}\) of its \(3\)-arrow graph \(\threearrowgraph{\mathcal{C}}\) with respect to fraction equality \({\fractionequal}\). The crucial point in the construction will be the following lemma and its corollaries.

\begin{lemma}[factorisation lemma] \label{lem:factorisation_lemma}
We suppose given a uni-fractionable category \(\mathcal{C}\), denominators \(d\), \(e\) and morphisms \(f\), \(g\) in \(\mathcal{C}\) with \(f e = d g\). Moreover, we suppose given S-denominators \(i\), \(j\) and T-denominators \(p\), \(q\) in \(\mathcal{C}\) with \(d = i p\) and \(e = j q\).
\begin{enumerate}
\item \label{lem:factorisation_lemma:s-2-arrow} There exist S-denominators \(\tilde j\), \(k\), a T-denominator \(\tilde q\) and a morphism \(h\) in \(\mathcal{C}\) such that \(e = \tilde j \tilde q\), \(f \tilde j = i h\), \(p g = h \tilde q\), \(\tilde j = j k\), \(q = k \tilde q\).
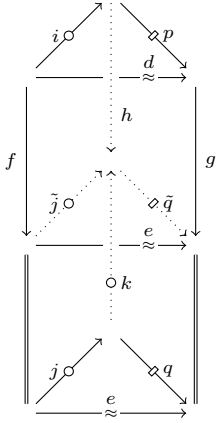
\[\begin{tikzpicture}[baseline=(m-6-1.base)]
  \matrix (m) [diagram without objects]{
    & & \\
    & & \\
    & & \\
    & & \\
    & & \\
    & & \\};
  \path[->, font=\scriptsize]
    (m-2-1) edge[den=0.75] node[above, near end] {\(d\)} (m-2-3)
            edge node[left] {\(f\)} (m-4-1)
            edge[sden] node[left] {\(i\)} (m-1-2)
    (m-2-3) edge node[right] {\(g\)} (m-4-3)
    (m-4-1) edge[den=0.75] node[above, near end] {\(e\)} (m-4-3)
            edge[exists, sden] node[left] {\(\tilde j\)} (m-3-2)
    (m-6-1) edge[den] node[above] {\(e\)} (m-6-3)
            edge[equality] (m-4-1)
            edge[sden] node[left] {\(j\)} (m-5-2)
    (m-6-3) edge[equality] (m-4-3)
    (m-1-2) edge[tden] node[right] {\(p\)} (m-2-3)
            edge[cross line, exists] node[right, near end] {\(h\)} (m-3-2)
    (m-3-2) edge[exists, tden] node[right] {\(\tilde q\)} (m-4-3)
    (m-5-2) edge[tden] node[right] {\(q\)} (m-6-3)
            edge[cross line, exists, sden=0.25] node[right, near start] {\(k\)} (m-3-2);
\end{tikzpicture}\]
\item \label{lem:factorisation_lemma:t-2-arrow} There exist an S-denominator \(\tilde i\), T-denominators \(\tilde p\), \(r\) and a morphism \(h\) in \(\mathcal{C}\) such that \(d = \tilde i \tilde p\), \(f j = \tilde i h\), \(\tilde p g = h q\), \(i = \tilde i r\), \(\tilde p = r p\).
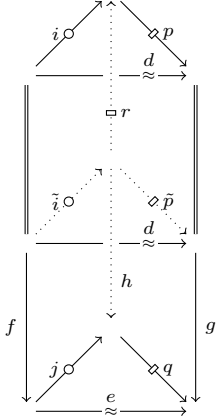
\[\begin{tikzpicture}[baseline=(m-6-1.base)]
  \matrix (m) [diagram without objects]{
    & & \\
    & & \\
    & & \\
    & & \\
    & & \\
    & & \\};
  \path[->, font=\scriptsize]
    (m-2-1) edge[den=0.75] node[above, near end] {\(d\)} (m-2-3)
            edge[sden] node[left] {\(i\)} (m-1-2)
    (m-4-1) edge[den=0.75] node[above, near end] {\(d\)} (m-4-3)
            edge node[left] {\(f\)} (m-6-1)
            edge[equality] (m-2-1)
            edge[exists, sden] node[left] {\(\tilde i\)} (m-3-2)
    (m-4-3) edge node[right] {\(g\)} (m-6-3)
            edge[equality] (m-2-3)
    (m-6-1) edge[den] node[above] {\(e\)} (m-6-3)
            edge[sden] node[left] {\(j\)} (m-5-2)
    (m-1-2) edge[tden] node[right] {\(p\)} (m-2-3)
    (m-3-2) edge[exists, tden] node[right] {\(\tilde p\)} (m-4-3)
            edge[cross line, exists] node[right, near end] {\(h\)} (m-5-2)
            edge[cross line, exists, tden=0.25] node[right, near start] {\(r\)} (m-1-2)
    (m-5-2) edge[tden] node[right] {\(q\)} (m-6-3);
\end{tikzpicture}\]
\end{enumerate}
\end{lemma}
\begin{proof} \
\begin{enumerate}
\item We let
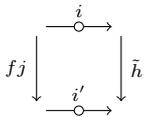
\[\begin{tikzpicture}[baseline=(m-2-1.base)]
  \matrix (m) [diagram without objects]{
    & \\
    & \\};
  \path[->, font=\scriptsize]
    (m-1-1) edge[sden] node[above] {\(i\)} (m-1-2)
            edge node[left] {\(f j\)} (m-2-1)
    (m-1-2) edge node[right] {\(\tilde h\)} (m-2-2)
    (m-2-1) edge[sden] node[above] {\(i'\)} (m-2-2);
\end{tikzpicture}\]
be a weak pushout rectangle in \(\mathcal{C}\) such that \(i'\) is an S-denominator in \(\mathcal{C}\). Since
\[i p g = d g = f e = f j q,\]
there exists an induced morphism \(a\) with \(q = i' a\) and \(p g = \tilde h a\). By semi-saturatedness, \(a\) is a denominator in \(\mathcal{C}\), and thus there exist an S-denominator \(\tilde k\) and a T-denominator \(\tilde q\) with \(a = \tilde k \tilde q\).
\[\begin{tikzpicture}[baseline=(m-5-1.base)]
  \matrix (m) [diagram without objects=0.9em]{
    & & & & \\
    & & & & \\
    & & & & \\
    & & & & \\
    & & & & \\};
  \path[->, font=\scriptsize]
    (m-1-1) edge[sden] node[above] {\(i\)} (m-1-3)
            edge node[left] {\(f j\)} (m-3-1)
    (m-1-3) edge[bend left] node[right] {\(p g\)} (m-5-5)
            edge node[right] {\(\tilde h\)} (m-3-3)
    (m-3-1) edge[sden] node[above] {\(i'\)} (m-3-3)
            edge[bend right, tden] node[below] {\(q\)} (m-5-5)
    (m-3-3) edge[exists, den] node[left] {\(a\)} (m-5-5)
            edge[exists, sden] node[above, near end] {\(\tilde k\)} (m-3-4)
    (m-3-4) edge[exists, tden] node[right=-1pt, near start] {\(\tilde q\)} (m-5-5);
\end{tikzpicture}\]
We set \(h := \tilde h \tilde k\), \(k := i' \tilde k\), \(\tilde j := j i' \tilde k\) and get \(e = \tilde j \tilde q\), \(f \tilde j = i h\), \(p g = h \tilde q\), \(\tilde j = j k\), \(q = k \tilde q\). Moreover, \(k = i' \tilde k\) and \(\tilde j = j i' \tilde k\) are S-denominators in \(\mathcal{C}\) by multiplicativity.
\item This is dual to~\ref{lem:factorisation_lemma:s-2-arrow}. \qedhere
\end{enumerate}
\end{proof}

\begin{corollary} \label{cor:factorisation_lemma_for_one_given_factorisation}
We suppose given a uni-fractionable category \(\mathcal{C}\), denominators \(d\), \(e\) and morphisms \(f\), \(g\) in \(\mathcal{C}\) with \(f e = d g\).
\begin{enumerate}
\item \label{cor:factorisation_lemma_for_one_given_factorisation:s-2-arrow} Given an S-denominator \(i\) and a T-denominator \(p\) in \(\mathcal{C}\) with \(d = i p\), there exist an S-denominator \(j\), a T-denominator \(q\) and a morphism \(h\) in \(\mathcal{C}\) such that \(e = j q\), \(f j = i h\), \(p g = h q\).
\[\begin{tikzpicture}[baseline=(m-4-1.base)]
  \matrix (m) [diagram without objects]{
    & & \\
    & & \\
    & & \\
    & & \\};
  \path[->, font=\scriptsize]
    (m-2-1) edge[den=0.75] node[above, near end] {\(d\)} (m-2-3)
            edge node[left] {\(f\)} (m-4-1)
            edge[sden] node[left] {\(i\)} (m-1-2)
    (m-2-3) edge node[right] {\(g\)} (m-4-3)
    (m-4-1) edge[den] node[above] {\(e\)} (m-4-3)
            edge[exists, sden] node[left] {\(j\)} (m-3-2)
    (m-1-2) edge[tden] node[right] {\(p\)} (m-2-3)
            edge[cross line, exists] node[right, near end] {\(h\)} (m-3-2)
    (m-3-2) edge[exists, tden] node[right] {\(q\)} (m-4-3);
\end{tikzpicture}\]
\item \label{cor:factorisation_lemma_for_one_given_factorisation:t-2-arrow} Given an S-denominator \(j\) and a T-denominator \(q\) in \(\mathcal{C}\) with \(e = j q\), there exist an S-denominator \(i\), a T-denominator \(p\) and a morphism \(h\) in \(\mathcal{C}\) such that \(d = i p\), \(f j = i h\), \(p g = h q\).
\[\begin{tikzpicture}[baseline=(m-4-1.base)]
  \matrix (m) [diagram without objects]{
    & & \\
    & & \\
    & & \\
    & & \\};
  \path[->, font=\scriptsize]
    (m-2-1) edge[den=0.75] node[above, near end] {\(d\)} (m-2-3)
            edge node[left] {\(f\)} (m-4-1)
            edge[exists, sden] node[left] {\(i\)} (m-1-2)
    (m-2-3) edge node[right] {\(g\)} (m-4-3)
    (m-4-1) edge[den] node[above] {\(e\)} (m-4-3)
            edge[sden] node[left] {\(j\)} (m-3-2)
    (m-1-2) edge[exists, tden] node[right] {\(p\)} (m-2-3)
            edge[cross line, exists] node[right, near end] {\(h\)} (m-3-2)
    (m-3-2) edge[tden] node[right] {\(q\)} (m-4-3);
\end{tikzpicture}\]
\end{enumerate}
\end{corollary}
\begin{proof}
This follows from the factorisation axiom and the factorisation lemma~\ref{lem:factorisation_lemma}.
\end{proof}

\begin{corollary} \label{cor:factorisation_lemma_for_no_given_factorisations}
We suppose given a uni-fractionable category \(\mathcal{C}\), denominators \(d\), \(e\) and morphisms \(f\), \(g\) in \(\mathcal{C}\) with \(f e = d g\).
There exist S-denominators \(i\), \(j\), T-denominators \(p\), \(q\) and a morphism \(h\) in \(\mathcal{C}\) with \(d = i p\), \(e = j q\), \(f j = i h\), \(p g = h q\).
\[\begin{tikzpicture}[baseline=(m-4-1.base)]
  \matrix (m) [diagram without objects]{
    & & \\
    & & \\
    & & \\
    & & \\};
  \path[->, font=\scriptsize]
    (m-2-1) edge[den=0.75] node[above, near end] {\(d\)} (m-2-3)
            edge node[left] {\(f\)} (m-4-1)
            edge[exists, sden] node[left] {\(i\)} (m-1-2)
    (m-2-3) edge node[right] {\(g\)} (m-4-3)
    (m-4-1) edge[den] node[above] {\(e\)} (m-4-3)
            edge[exists, sden] node[left] {\(j\)} (m-3-2)
    (m-1-2) edge[exists, tden] node[right] {\(p\)} (m-2-3)
            edge[cross line, exists] node[right, near end] {\(h\)} (m-3-2)
    (m-3-2) edge[exists, tden] node[right] {\(q\)} (m-4-3);
\end{tikzpicture}\]
\end{corollary}
\begin{proof}
This follows from the factorisation axiom and corollary~\ref{cor:factorisation_lemma_for_one_given_factorisation}.
\end{proof}

The following proposition will essentially prove the first part of our main theorem~\ref{th:description_of_the_fraction_category}, cf.\ also proposition~\ref{prop:flexibility_of_composites_and_inverses_in_the_fraction_category} below.

\begin{proposition} \label{prop:welldefinedness_of_the_fraction_category}
For every uni-fractionable category \(\mathcal{C}\), there is a category structure on \((\threearrowgraph{\mathcal{C}}) / {\fractionequal}\), where the composition is constructed by the following procedure.

We suppose given \((b_1, f_1, a_1), (b_2, f_2, a_2) \in \Arr \threearrowgraph{\mathcal{C}}\) with \(\Target{(b_1, f_1, a_1)} = \Source{(b_2, f_2, a_2)}\), that is, with \(\Source a_1 = \Target b_2\). First, we choose an S-denominator \(j\) and a T-denominator \(q\) in \(\mathcal{C}\) with \(b_2 a_1 = j q\). Second, we choose a T-denominator \(q'\) and a morphism \(f_1'\) in \(\mathcal{C}\) with \(f_1' q = q' f_1\), and we choose an S-denominator \(j'\) and a morphism \(f_2'\) in \(\mathcal{C}\) with \(j f_2' = f_2 j'\).
\[\begin{tikzpicture}[baseline=(m-3-1.base)]
  \matrix (m) [diagram without objects=0.9em]{
    & & & & & & & & & & \\
    & & & & & & & & & & \\
    & & & & & & & & & & \\};
  \path[->, font=\scriptsize]
    (m-1-3) edge node[above] {\(f_1'\)} (m-1-6)
            edge[tden] node[left] {\(q'\)} (m-2-2)
    (m-1-6) edge node[above] {\(f_2'\)} (m-1-9)
            edge[tden] node[left] {\(q\)} (m-2-5)
    (m-2-2) edge node[above] {\(f_1\)} (m-2-5)
            edge[den] node[left] {\(b_1\)} (m-3-1)
    (m-2-7) edge node[above] {\(f_2\)} (m-2-10)
            edge[den] node[right] {\(b_2\)} (m-3-6)
            edge[sden] node[right] {\(j\)} (m-1-6)
    (m-2-10) edge[sden] node[right=1pt] {\(j'\)} (m-1-9)
    (m-3-6) edge[den] node[left] {\(a_1\)} (m-2-5)
    (m-3-11) edge[den] node[right] {\(a_2\)} (m-2-10);
\end{tikzpicture}\]
Then
\[(\doublefrac{b_1}{f_1}{a_1}) (\doublefrac{b_2}{f_2}{a_2}) = \doublefrac{q' b_1}{f_1' f_2'}{a_2 j'}.\]

The identity of \(X \in \Ob{(\threearrowgraph{\mathcal{C}}) / {\fractionequal}}\) is given by
\[1_X = \doublefrac{1_X}{1_X}{1_X}.\]
\end{proposition}
\begin{proof}
We suppose given a uni-fractionable category \(\mathcal{C}\). Our first aim is to show that the construction described above is independent of all choices. To this end, we first consider the particular case of choosing a weak pullback of \(f_1\) and \(q\) and a weak pushout of \(f_2\) and \(j\) to obtain a T-denominator \(q'\), an S-denominator \(j'\) and morphisms \(f_1'\), \(f_2'\) in \(\mathcal{C}\).

We suppose given \((b_l, f_l, a_l), (\tilde b_l, \tilde f_l, \tilde a_l) \in \Arr \threearrowgraph{\mathcal{C}}\) and \(c_l, c_l' \in \Mor \mathcal{C}\) with \(b_l = c_l \tilde b_l\), \(f_l c_l' = c_l \tilde f_l\), \(a_l c_l' = \tilde a_l\) for \(l \in \{1, 2\}\), and such that \(\Target{(b_1, f_1, a_1)} = \Source{(b_2, f_2, a_2)}\).
\[\begin{tikzpicture}[baseline=(m-2-1.base)]
  \matrix (m) [diagram without objects]{
    & & & & & & \\
    & & & & & & \\};
  \path[->, font=\scriptsize]
    (m-1-1) edge[equality] (m-2-1)
    (m-1-2) edge node[above] {\(f_1\)} (m-1-3)
            edge node[right] {\(c_1\)} (m-2-2)
            edge[den] node[above] {\(b_1\)} (m-1-1)
    (m-1-3) edge node[right] {\(c_1'\)} (m-2-3)
    (m-1-4) edge[equality] (m-2-4)
            edge[den] node[above] {\(a_1\)} (m-1-3)
    (m-1-5) edge node[above] {\(f_2\)} (m-1-6)
            edge node[right] {\(c_2\)} (m-2-5)
            edge[den] node[above] {\(b_2\)} (m-1-4)
    (m-1-6) edge node[right] {\(c_2'\)} (m-2-6)
    (m-1-7) edge[equality] (m-2-7)
            edge[den] node[above] {\(a_2\)} (m-1-6)
    (m-2-2) edge node[above] {\(\tilde f_1\)} (m-2-3)
            edge[den] node[above] {\(\tilde b_1\)} (m-2-1)
    (m-2-4) edge[den] node[above] {\(\tilde a_1\)} (m-2-3)
    (m-2-5) edge node[above] {\(\tilde f_2\)} (m-2-6)
            edge[den] node[above] {\(\tilde b_2\)} (m-2-4)
    (m-2-7) edge[den] node[above] {\(\tilde a_2\)} (m-2-6);
\end{tikzpicture}\]
We choose S-denominators \(j\), \(\tilde j\) and T-denominators \(q\), \(\tilde q\) in \(\mathcal{C}\) such that \(b_2 a_1 = j q\) and \(\tilde b_2 \tilde a_1 = \tilde j \tilde q\). By the factorisation lemma~\ref{lem:factorisation_lemma}\ref{lem:factorisation_lemma:s-2-arrow}, there exist an S-denominator \(k\), a T-denominator \(r\) and morphisms \(c\), \(\tilde c\) in \(\mathcal{C}\) with \(\tilde b_2 \tilde a_1 = k r\), \(q c_1' = c r\), \(c_2 k = j c\), \(\tilde q = \tilde c r\), \(k = \tilde j \tilde c\). Next, we choose weak pullback rectangles
\[\begin{tikzpicture}[baseline=(m-2-1.base)]
  \matrix (m) [diagram without objects]{
    & \\
    & \\};
  \path[->, font=\scriptsize]
    (m-1-1) edge node[above] {\(f_1'\)} (m-1-2)
            edge[tden] node[left] {\(q'\)} (m-2-1)
    (m-1-2) edge[tden] node[right] {\(q\)} (m-2-2)
    (m-2-1) edge node[above] {\(f_1\)} (m-2-2);
\end{tikzpicture}
\text{ and }
\begin{tikzpicture}[baseline=(m-2-1.base)]
  \matrix (m) [diagram without objects]{
    & \\
    & \\};
  \path[->, font=\scriptsize]
    (m-1-1) edge node[above] {\(g_1\)} (m-1-2)
            edge[tden] node[left] {\(r'\)} (m-2-1)
    (m-1-2) edge[tden] node[right] {\(r\)} (m-2-2)
    (m-2-1) edge node[above] {\(\tilde f_1\)} (m-2-2);
\end{tikzpicture}
\text{ and }
\begin{tikzpicture}[baseline=(m-2-1.base)]
  \matrix (m) [diagram without objects]{
    & \\
    & \\};
  \path[->, font=\scriptsize]
    (m-1-1) edge node[above] {\(\tilde f_1'\)} (m-1-2)
            edge[tden] node[left] {\(\tilde q'\)} (m-2-1)
    (m-1-2) edge[tden] node[right] {\(\tilde q\)} (m-2-2)
    (m-2-1) edge node[above] {\(\tilde f_1\)} (m-2-2);
\end{tikzpicture}\]
in \(\mathcal{C}\) such that \(q'\), \(r'\), \(\tilde q'\) are T-denominators, and we choose weak pushout rectangles
\[\begin{tikzpicture}[baseline=(m-2-1.base)]
  \matrix (m) [diagram without objects]{
    & \\
    & \\};
  \path[->, font=\scriptsize]
    (m-1-1) edge node[above] {\(f_2'\)} (m-1-2)
    (m-2-1) edge node[above] {\(f_2\)} (m-2-2)
            edge[sden] node[left] {\(j\)} (m-1-1)
    (m-2-2) edge[sden] node[right] {\(j'\)} (m-1-2);
\end{tikzpicture}
\text{ and }
\begin{tikzpicture}[baseline=(m-2-1.base)]
  \matrix (m) [diagram without objects]{
    & \\
    & \\};
  \path[->, font=\scriptsize]
    (m-1-1) edge node[above] {\(g_2\)} (m-1-2)
    (m-2-1) edge node[above] {\(\tilde f_2\)} (m-2-2)
            edge[sden] node[left] {\(k\)} (m-1-1)
    (m-2-2) edge[sden] node[right] {\(k'\)} (m-1-2);
\end{tikzpicture}
\text{ and }
\begin{tikzpicture}[baseline=(m-2-1.base)]
  \matrix (m) [diagram without objects]{
    & \\
    & \\};
  \path[->, font=\scriptsize]
    (m-1-1) edge node[above] {\(\tilde f_2'\)} (m-1-2)
    (m-2-1) edge node[above] {\(\tilde f_2\)} (m-2-2)
            edge[sden] node[left] {\(\tilde j\)} (m-1-1)
    (m-2-2) edge[sden] node[right] {\(\tilde j'\)} (m-1-2);
\end{tikzpicture}\]
in \(\mathcal{C}\) such that \(j'\), \(k'\), \(\tilde j'\) are S-denominators. We obtain induced morphisms \(c'\) and \(\tilde c'\) on the weak pullbacks, that is, with \(q' c_1 = c' r'\), \(f_1' c = c' g_1\) and \(\tilde q' = c' r'\), \(\tilde f_1' \tilde c = \tilde c' g_1\), and induced morphisms \(c''\) and \(\tilde c''\) on the weak pushouts, that is, with \(c_2' k' = j' c''\), \(c g_2 = f_2' c''\) and \(k' = \tilde j' \tilde c''\), \(\tilde c g_2 = \tilde f_2' \tilde c''\).
\[\begin{tikzpicture}[baseline=(m-9-1.base)]
  \matrix (m) [diagram without objects=2.0em]{
    & & & & & & & & & & \\
    & & & & & & & & & & \\
    & & & & & & & & & & \\
    & & & & & & & & & & \\
    & & & & & & & & & & \\
    & & & & & & & & & & \\
    & & & & & & & & & & \\
    & & & & & & & & & & \\
    & & & & & & & & & & \\};
  \path[->, font=\scriptsize]
    (m-1-3) edge node[above] {\(f_1'\)} (m-1-6)
            edge node[right, near end] {\(c'\)} (m-4-3)
            edge[tden] node[left] {\(q'\)} (m-2-2)
    (m-1-6) edge node[above] {\(f_2'\)} (m-1-9)
            edge node[right, near end] {\(c\)} (m-4-6)
            edge[tden] node[left] {\(q\)} (m-2-5)
    (m-1-9) edge node[right, near end] {\(c''\)} (m-4-9)
    (m-4-3) edge node[above, near start] {\(g_1\)} (m-4-6)
            edge[tden] node[left, pos=0.4] {\(r'\)} (m-5-2)
    (m-4-6) edge node[above, near end] {\(g_2\)} (m-4-9)
            edge[tden] node[left] {\(r\)} (m-5-5)
    (m-7-3) edge node[above, near start] {\(\tilde f_1'\)} (m-7-6)
            edge[tden] node[left, pos=0.4] {\(\tilde q'\)} (m-8-2)
            edge node[right, near start] {\(\tilde c'\)} (m-4-3)
    (m-7-6) edge node[above, near end] {\(\tilde f_2'\)} (m-7-9)
            edge[tden] node[left] {\(\tilde q\)} (m-8-5)
            edge node[right, near start] {\(\tilde c\)} (m-4-6)
    (m-7-9) edge node[right, near start] {\(\tilde c''\)} (m-4-9)
    (m-2-2) edge[cross line] node[above] {\(f_1\)} (m-2-5)
            edge node[right] {\(c_1\)} (m-5-2)
            edge[den] node[left] {\(b_1\)} (m-3-1)
    (m-2-5) edge[cross line] node[right] {\(c_1'\)} (m-5-5)
    (m-2-7) edge[cross line] node[above] {\(f_2\)} (m-2-10)
            edge[cross line] node[right] {\(c_2\)} (m-5-7)
            edge[cross line] node[above, fill=white] {\(b_2 a_1\)} (m-2-5)
            edge[cross line, den] (m-2-5) 
            edge[sden] node[right] {\(j\)} (m-1-6)
    (m-2-10) edge node[right] {\(c_2'\)} (m-5-10)
             edge[sden] node[right=1pt] {\(j'\)} (m-1-9)
    (m-3-1) edge[equality] (m-6-1)
    (m-3-11) edge[equality] (m-6-11)
             edge[den] node[right] {\(a_2\)} (m-2-10)
    (m-5-2) edge[cross line] node[above] {\(\tilde f_1\)} (m-5-5)
            edge[den] node[left] {\(\tilde b_1\)} (m-6-1)
    (m-5-7) edge[cross line] node[above] {\(\tilde f_2\)} (m-5-10)
            edge[cross line] node[above, fill=white] {\(\tilde b_2 \tilde a_1\)} (m-5-5)
            edge[cross line, den] (m-5-5) 
            edge[sden] node[right] {\(k\)} (m-4-6)
    (m-5-10) edge[sden] node[right=1pt, pos=0.6] {\(k'\)} (m-4-9)
    (m-6-11) edge[den] node[right] {\(\tilde a_2\)} (m-5-10)
    (m-8-2) edge node[above] {\(\tilde f_1\)} (m-8-5)
            edge[den] node[left] {\(\tilde b_1\)} (m-9-1)
            edge[equality] (m-5-2)
    (m-8-5) edge[cross line, equality] (m-5-5)
    (m-8-7) edge node[above] {\(\tilde f_2\)} (m-8-10)
            edge[den] node[above] {\(\tilde b_2 \tilde a_1\)} (m-8-5)
            edge[sden] node[right=1pt] {\(\tilde j\)} (m-7-6)
            edge[cross line, equality] (m-5-7)
    (m-8-10) edge[sden] node[right=1pt, pos=0.6] {\(\tilde j'\)} (m-7-9)
             edge[equality] (m-5-10)
    (m-9-1) edge[equality] (m-6-1)
    (m-9-11) edge[den] node[right] {\(\tilde a_2\)} (m-8-10)
             edge[equality] (m-6-11);
\end{tikzpicture}\]
Hence we have \(q' b_1 = c' r' \tilde b_1\), \(f_1' f_2' c'' = c' g_1 g_2\), \(a_2 j' c'' = \tilde a_2 k'\) and therefore \((q' b_1, f_1' f_2', a_2 j') \fractionequal (r' \tilde b_1, g_1 g_2, \tilde a_2 k')\), and we get \(\tilde q' \tilde b_1 = \tilde c' r' \tilde b_1\), \(\tilde f_1' \tilde f_2' \tilde c'' = \tilde c' g_1 g_2\), \(\tilde a_2 \tilde j' \tilde c'' = \tilde a_2 k'\) and therefore \((r' \tilde b_1, g_1 g_2, \tilde a_2 k') \fractionequal (\tilde q' \tilde b_1, \tilde f_1' \tilde f_2', \tilde a_2 \tilde j')\).
\[\begin{tikzpicture}[baseline=(m-3-1.base)]
  \matrix (m) [diagram without objects]{
    & & & \\
    & & & \\
    & & & \\};
  \path[->, font=\scriptsize]
    (m-1-1) edge[equality] (m-2-1)
    (m-1-2) edge node[above] {\(f_1' f_2'\)} (m-1-3)
            edge node[right] {\(c'\)} (m-2-2)
            edge[den] node[above] {\(q' b_1\)} (m-1-1)
    (m-1-3) edge node[right] {\(c''\)} (m-2-3)
    (m-1-4) edge[equality] (m-2-4)
            edge[den] node[above] {\(a_2 j'\)} (m-1-3) 
    (m-2-2) edge node[above] {\(g_1 g_2\)} (m-2-3)
            edge[den] node[above] {\(r' \tilde b_1\)} (m-2-1)
    (m-2-4) edge[den] node[above] {\(\tilde a_2 k'\)} (m-2-3)
    (m-3-1) edge[equality] (m-2-1)
    (m-3-2) edge node[above] {\(\tilde f_1' \tilde f_2'\)} (m-3-3)
            edge[den] node[above] {\(\tilde q' \tilde b_1\)} (m-3-1)
            edge node[right] {\(\tilde c'\)} (m-2-2)
    (m-3-3) edge node[right] {\(\tilde c''\)} (m-2-3)
    (m-3-4) edge[den] node[above] {\(\tilde a_2 \tilde j'\)} (m-3-3)
            edge[equality] (m-2-4);
\end{tikzpicture}\]
Altogether, we have \((q' b_1, f_1' f_2', a_2 j') \fractionequal (\tilde q' \tilde b_1, \tilde f_1' \tilde f_2', \tilde a_2 \tilde j')\) in \(\threearrowgraph \mathcal{C}\), that is, we have
\[\doublefrac{q' b_1}{f_1' f_2'}{a_2 j'} = \doublefrac{\tilde q' \tilde b_1}{\tilde f_1' \tilde f_2'}{\tilde a_2 \tilde j'}\]
in \((\threearrowgraph \mathcal{C}) / {\fractionequal}\).

In the special case where \(c_1 = 1\), \(c_1' = 1\), \(c_2 = 1\), \(c_2' = 1\), we see that different choices of constructions via weak pullback and weak pushout rectangles lead to the same double fraction \(\doublefrac{q' b_1}{f_1' f_2'}{a_2 j'} = \doublefrac{\tilde q' b_1}{\tilde f_1' \tilde f_2'}{\tilde a_2 \tilde j'}\). Hence we obtain a well-defined map
\begin{align*}
& c\colon \Arr \threearrowgraph{\mathcal{C}} \fibreprod{\Target}{\Source} \Arr \threearrowgraph{\mathcal{C}} \map \Arr{(\threearrowgraph{\mathcal{C}}) / {\fractionequal}}, \\
& \qquad ((b_1, f_1, a_1), (b_2, f_2, a_2)) \mapsto \doublefrac{q' b_1}{f_1' f_2'}{a_2 j'},
\end{align*}
where \(q'\), \(f_1'\), \(f_2'\), \(j'\) are constructed as described above. Now the general case shows that \(c\) is independent of the choice of the representatives in the equivalence classes with respect to \({\fractionequal}\), and thus we obtain an induced map
\[\overline{c}\colon \Arr{(\threearrowgraph{\mathcal{C}}) / {\fractionequal}} \fibreprod{\Target}{\Source} \Arr{(\threearrowgraph{\mathcal{C}}) / {\fractionequal}} \map \Arr{(\threearrowgraph{\mathcal{C}}) / {\fractionequal}}\]
given by
\[\overline{c}(\doublefrac{b_1}{f_1}{a_1}, \doublefrac{b_2}{f_2}{a_2}) = c((b_1, f_1, a_1), (b_2, f_2, a_2)) = \doublefrac{q' b_1}{f_1' f_2'}{a_2 j'}\]
for \((b_1, f_1, a_1), (b_2, f_2, a_2) \in \Arr \threearrowgraph \mathcal{C}\) with \(\Target{(b_1, f_1, a_1)} = \Source{(b_2, f_2, a_2)}\).

We claim that arbitrary commutative quadrangles may be used instead of weak pullback and weak pushout rectangles to compute \(\overline{c}\). Indeed, given a weak pullback rectangle
\[\begin{tikzpicture}[baseline=(m-2-1.base)]
  \matrix (m) [diagram without objects]{
    & \\
    & \\};
  \path[->, font=\scriptsize]
    (m-1-1) edge node[above] {\(f_1'\)} (m-1-2)
            edge[tden] node[left] {\(q'\)} (m-2-1)
    (m-1-2) edge[tden] node[right] {\(q\)} (m-2-2)
    (m-2-1) edge node[above] {\(f_1\)} (m-2-2);
\end{tikzpicture}\]
and a weak pushout rectangle
\[\begin{tikzpicture}[baseline=(m-2-1.base)]
  \matrix (m) [diagram without objects]{
    & \\
    & \\};
  \path[->, font=\scriptsize]
    (m-1-1) edge node[above] {\(f_2'\)} (m-1-2)
    (m-2-1) edge node[above] {\(f_2\)} (m-2-2)
            edge[sden] node[left] {\(j\)} (m-1-1)
    (m-2-2) edge[sden] node[right] {\(j'\)} (m-1-2);
\end{tikzpicture}\]
and arbitrary commutative quadrangles
\[\begin{tikzpicture}[baseline=(m-2-1.base)]
  \matrix (m) [diagram without objects]{
    & \\
    & \\};
  \path[->, font=\scriptsize]
    (m-1-1) edge node[above] {\(\tilde f_1'\)} (m-1-2)
            edge[tden] node[left] {\(\tilde q'\)} (m-2-1)
    (m-1-2) edge[tden] node[right] {\(q\)} (m-2-2)
    (m-2-1) edge node[above] {\(f_1\)} (m-2-2);
\end{tikzpicture}
\text{ and }
\begin{tikzpicture}[baseline=(m-2-1.base)]
  \matrix (m) [diagram without objects]{
    & \\
    & \\};
  \path[->, font=\scriptsize]
    (m-1-1) edge node[above] {\(\tilde f_2'\)} (m-1-2)
    (m-2-1) edge node[above] {\(f_2\)} (m-2-2)
            edge[sden] node[left] {\(j\)} (m-1-1)
    (m-2-2) edge[sden] node[right] {\(\tilde j'\)} (m-1-2);
\end{tikzpicture}\]
such that \(q'\), \(\tilde q'\) are T-denominators and \(j'\), \(\tilde j'\) are S-denominators in \(\mathcal{C}\), we obtain induced morphisms \(c\) and \(c'\) such that \(\tilde q' = c q'\), \(\tilde f_1' = c f_1'\), \(\tilde f_2' = f_2' c'\), \(\tilde j' = j' c'\).
\[\begin{tikzpicture}[baseline=(m-5-1.base)]
  \matrix (m) [diagram without objects=0.9em]{
    & & & & & & & & & & \\
    & & & & & & & & & & \\
    & & & & & & & & & & \\
    & & & & & & & & & & \\
    & & & & & & & & & & \\};
  \path[->, font=\scriptsize]
    (m-1-3) edge[out=-10, in=125] node[above] {\(\tilde f_1'\)} (m-3-6)
            edge[exists] node[right] {\(c\)} (m-3-3)
            edge[tden, out=-145, in=100] node[left] {\(\tilde q'\)} (m-4-2)
    (m-3-3) edge node[above] {\(f_1'\)} (m-3-6)
            edge[tden] node[right] {\(q'\)} (m-4-2)
    (m-3-6) edge node[above] {\(f_2'\)} (m-3-9)
            edge[tden] node[left] {\(q\)} (m-4-5)
            edge[out=55, in=-170] node[above] {\(\tilde f_2'\)} (m-1-9)
    (m-3-9) edge[exists] node[right] {\(c'\)} (m-1-9)
    (m-4-2) edge node[above, near end] {\(f_1\)} (m-4-5)
            edge[den] node[left] {\(b_1\)} (m-5-1)
    (m-4-7) edge node[above, near start] {\(f_2\)} (m-4-10)
            edge[den] node[right] {\(b_2\)} (m-5-6)
            edge[sden] node[right] {\(j\)} (m-3-6)
    (m-4-10) edge[sden] node[left=1pt] {\(j'\)} (m-3-9)
            edge[sden, out=80, in=-35] node[right] {\(\tilde j'\)} (m-1-9)
    (m-5-6) edge[den] node[left] {\(a_1\)} (m-4-5)
    (m-5-11) edge[den] node[right] {\(a_2\)} (m-4-10);
\end{tikzpicture}\]
Hence \((\tilde q' b_1, \tilde f_1' \tilde f_2', a_2 \tilde j') = (c q' b_1, c f_1' f_2' c', a_2 j' c') \fractionequal (q' b_1, f_1' f_2', a_2 j')\) and therefore
\[\overline{c}(\doublefrac{b_1}{f_1}{a_1}, \doublefrac{b_2}{f_2}{a_2}) = \doublefrac{q' b_1}{f_1' f_2'}{a_2 j'} = \doublefrac{\tilde q' b_1}{\tilde f_1' \tilde f_2'}{a_2 \tilde j'}.\]
This proves the claim.

In addition to \(\overline{c}\), we define the map
\[e\colon \Ob{(\threearrowgraph{\mathcal{C}}) / {\fractionequal}} \map \Arr{(\threearrowgraph{\mathcal{C}}) / {\fractionequal}}, X \mapsto \doublefrac{1_X}{1_X}{1_X}.\]
To show that \((\threearrowgraph{\mathcal{C}}) / {\fractionequal}\) is a category with composition \(\overline{c}\) and identity map \(e\), it remains to verify the category axioms. By the definitions of \(\overline{c}\) and \(e\), we have
\[\Source{\overline{c}(\doublefrac{b_1}{f_1}{a_1}, \doublefrac{b_2}{f_2}{a_2})} = \Source{\doublefrac{q' b_1}{f_1' f_2'}{a_2 j'}} = \Target(q' b_1) = \Target b_1 = \Source{\doublefrac{b_1}{f_1}{a_1}}\]
and analogously \(\Target{\overline{c}(\doublefrac{b_1}{f_1}{a_1}, \doublefrac{b_2}{f_2}{a_2})} = \Target{\doublefrac{b_2}{f_2}{a_2}}\) for all \((b_1, f_1, a_1), (b_2, f_2, a_2) \in \Arr \threearrowgraph{\mathcal{C}}\) with \(\Target{\doublefrac{b_1}{f_1}{a_1}} = \Source{\doublefrac{b_2}{f_2}{a_2}}\), as well as
\[\Source{e(X)} = \Source{\doublefrac{1_X}{1_X}{1_X}} = \Target 1_X = X\]
and analogously \(\Target{e(X)} = X\) for all \(X \in \Ob{(\threearrowgraph{\mathcal{C}}) / {\fractionequal}}\).

For the associativity of \(\overline{c}\), we suppose given \((b_l, f_l, a_l) \in \Arr \threearrowgraph{\mathcal{C}}\) for \(l \in \{1, 2, 3\}\) such that \(\Target{\doublefrac{b_1}{f_1}{a_1}} = \Source{\doublefrac{b_2}{f_2}{a_2}}\) and \(\Target{\doublefrac{b_2}{f_2}{a_2}} = \Source{\doublefrac{b_3}{f_3}{a_3}}\). We choose S-denominators \(j\), \(\tilde j\) and T-denominators \(q\), \(\tilde q\) with \(b_2 a_1 = j q\) and \(b_3 a_2 = \tilde j \tilde q\). Then we choose T-denominators \(q'\), \(\tilde q'\) and morphisms \(f_1'\), \(\tilde f_2'\) in \(\mathcal{C}\) with \(f_1' q = q' f_1\) and \(\tilde f_2' \tilde q = \tilde q' f_2\), and we choose S-denominators \(j'\), \(\tilde j'\) and morphisms \(f_2'\), \(\tilde f_3'\) in \(\mathcal{C}\) with \(j f_2' = f_2 j'\) and \(\tilde j \tilde f_3' = f_3 \tilde j'\). By definition of \(\overline{c}\), we obtain
\begin{align*}
& \overline{c}(\doublefrac{b_1}{f_1}{a_1}, \doublefrac{b_2}{f_2}{a_2}) = \doublefrac{q' b_1}{f_1' f_2'}{a_2 j'}, \\
& \overline{c}(\doublefrac{b_2}{f_2}{a_2}, \doublefrac{b_3}{f_3}{a_3}) = \doublefrac{\tilde q' b_2}{\tilde f_2' \tilde f_3'}{a_3 \tilde j'}.
\end{align*}
Moreover, we have \(\tilde q' j f_2' = \tilde f_2' \tilde q j'\), and thus by corollary~\ref{cor:factorisation_lemma_for_no_given_factorisations} there exist S-denominators \(k\), \(\tilde k\), T-denominators \(r\), \(\tilde r\) and a morphism \(f_2''\) in \(\mathcal{C}\) with \(\tilde q j' = k r\), \(\tilde q' j = \tilde k \tilde r\), \(\tilde r f_2' = f_2'' r\), \(\tilde f_2' k = \tilde k f_2''\). We choose a T-denominator \(\tilde r'\) and a morphism \(f_1''\) in \(\mathcal{C}\) with \(f_1'' \tilde r = \tilde r' f_1'\), and we choose an S-denominator \(k'\) and a morphism \(f_3''\) in \(\mathcal{C}\) with \(k f_3'' = \tilde f_3' k'\). Then we obtain \(\tilde r' f_1' f_2' = f_1'' f_2'' r\), \(\tilde j k f_3'' = f_3 \tilde j' k'\), \(\tilde f_2' \tilde f_3' k' = \tilde k f_2'' f_3''\), \(f_1'' \tilde r q = \tilde r' q' f_1\), and therefore
\begin{align*}
\overline{c}(\overline{c}(\doublefrac{b_1}{f_1}{a_1}, \doublefrac{b_2}{f_2}{a_2}), \doublefrac{b_3}{f_3}{a_3}) & = \overline{c}(\doublefrac{q' b_1}{f_1' f_2'}{a_2 j'}, \doublefrac{b_3}{f_3}{a_3}) = \doublefrac{\tilde r' q' b_1}{f_1'' f_2'' f_3''}{a_3 \tilde j' k'} \\
& = \overline{c}(\doublefrac{b_1}{f_1}{a_1}, \doublefrac{\tilde q' b_2}{\tilde f_2' \tilde f_3'}{a_3 \tilde j'}) = \overline{c}(\doublefrac{b_1}{f_1}{a_1}, \overline{c}(\doublefrac{b_2}{f_2}{a_2}, \doublefrac{b_3}{f_3}{a_3})).
\end{align*}
Thus \(\overline{c}\) is associative.
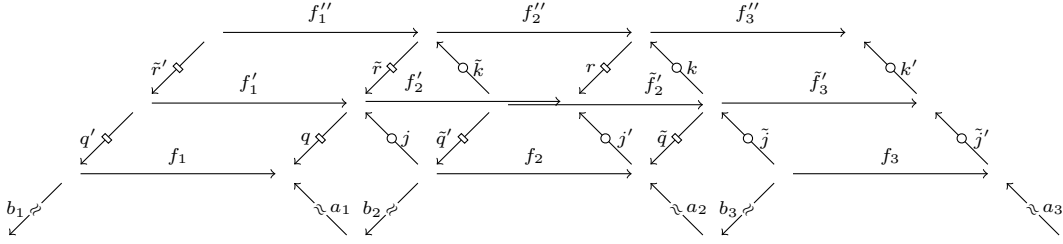
\[\begin{tikzpicture}[baseline=(m-4-1.base)]
  \matrix (m) [diagram without objects=2.0em]{
    & & & & & & & & & & & & & & & \\
    & & & & & & & & & & & & & & & \\
    & & & & & & & & & & & & & & & \\
    & & & & & & & & & & & & & & & \\};
  \path[->, font=\scriptsize]
    (m-1-4) edge node[above] {\(f_1''\)} (m-1-7)
            edge[tden] node[left] {\(\tilde r'\)} (m-2-3)
    (m-1-7) edge node[above] {\(f_2''\)} (m-1-10)
            edge[tden] node[left] {\(\tilde r\)} (m-2-6)
    (m-1-10) edge node[above] {\(f_3''\)} (m-1-13)
            edge[tden] node[left] {\(r\)} (m-2-9)
    (m-2-3) edge node[above] {\(f_1'\)} (m-2-6)
            edge[tden] node[left] {\(q'\)} (m-3-2)
    (m-2-6) edge[tden] node[left] {\(q\)} (m-3-5)
    (m-2-6.10) edge node[above, near start] {\(f_2'\)} (m-2-9.170)
    (m-2-8) edge[tden] node[left] {\(\tilde q'\)} (m-3-7)
            edge[sden] node[right] {\(\tilde k\)} (m-1-7)
    (m-2-8.-10) edge node[above, near end] {\(\tilde f_2'\)} (m-2-11.-170)
    (m-2-11) edge node[above] {\(\tilde f_3'\)} (m-2-14)
             edge[tden] node[left] {\(\tilde q\)} (m-3-10)
             edge[sden] node[right] {\(k\)} (m-1-10)
    (m-2-14) edge[sden] node[right] {\(k'\)} (m-1-13)
    (m-3-2) edge node[above] {\(f_1\)} (m-3-5)
            edge[den] node[left] {\(b_1\)} (m-4-1)
    (m-3-7) edge node[above] {\(f_2\)} (m-3-10)
            edge[den] node[left] {\(b_2\)} (m-4-6)
            edge[sden] node[right] {\(j\)} (m-2-6)
    (m-3-10) edge[sden] node[right=1pt] {\(j'\)} (m-2-9)
    (m-3-12) edge node[above] {\(f_3\)} (m-3-15)
             edge[den] node[left] {\(b_3\)} (m-4-11)
             edge[sden] node[right=1pt] {\(\tilde j\)} (m-2-11)
    (m-3-15) edge[sden] node[right=1pt] {\(\tilde j'\)} (m-2-14)
    (m-4-6) edge[den] node[right] {\(a_1\)} (m-3-5)
    (m-4-11) edge[den] node[right] {\(a_2\)} (m-3-10)
    (m-4-16) edge[den] node[right] {\(a_3\)} (m-3-15);
\end{tikzpicture}\]

Finally, we suppose given \((b, f, a) \in \Arr \threearrowgraph{\mathcal{C}}\). We want to show that \(\overline{c}(\doublefrac{b}{f}{a}, e(\Target{\doublefrac{b}{f}{a}})) = \doublefrac{b}{f}{a}\). By the normalisation lemma~\ref{lem:normalisation_lemma}, there exists a normal arrow \((p, g, i) \in \Arr \threearrowgraph{\mathcal{C}}\) with \((b, f, a) \fractionequal (p, g, i)\). Since a factorisation of \(i\) into an S-denominator followed by a T-denominator is given by \(i = i 1\), we obtain
\[\overline{c}(\doublefrac{b}{f}{a}, e(\Target{\doublefrac{b}{f}{a}})) = \overline{c}(\doublefrac{p}{g}{i}, \doublefrac{1}{1}{1}) = \doublefrac{1 p}{g 1}{1 i} = \doublefrac{p}{g}{i} = \doublefrac{b}{f}{a}.\]
Analogously, we have \(\overline{c}(e(\Source{\doublefrac{b}{f}{a}}), \doublefrac{b}{f}{a}) = \doublefrac{b}{f}{a}\).
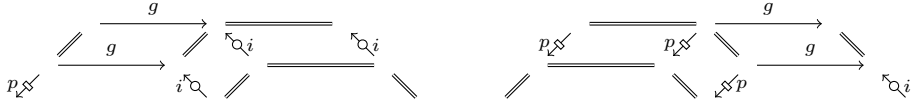
\[\begin{tikzpicture}[baseline=(m-3-1.base)]
  \matrix (m) [diagram without objects=0.9em]{
    & & & & & & & & & & \\
    & & & & & & & & & & \\
    & & & & & & & & & & \\};
  \path[->, font=\scriptsize]
    (m-1-3) edge node[above] {\(g\)} (m-1-6)
            edge[equality] (m-2-2)
    (m-1-6) edge[equality] (m-1-9)
            edge[equality] (m-2-5)
    (m-2-2) edge node[above] {\(g\)} (m-2-5)
            edge[tden] node[left] {\(p\)} (m-3-1)
    (m-2-7) edge[equality] (m-2-10)
            edge[equality] (m-3-6)
            edge[sden] node[right] {\(i\)} (m-1-6)
    (m-2-10) edge[sden] node[right] {\(i\)} (m-1-9)
    (m-3-6) edge[sden] node[left=1pt] {\(i\)} (m-2-5)
    (m-3-11) edge[equality] (m-2-10);
\end{tikzpicture}
\qquad
\begin{tikzpicture}[baseline=(m-3-1.base)]
  \matrix (m) [diagram without objects=0.9em]{
    & & & & & & & & & & \\
    & & & & & & & & & & \\
    & & & & & & & & & & \\};
  \path[->, font=\scriptsize]
    (m-1-3) edge[equality] (m-1-6)
            edge[tden] node[left] {\(p\)} (m-2-2)
    (m-1-6) edge node[above] {\(g\)} (m-1-9)
            edge[tden] node[left] {\(p\)} (m-2-5)
    (m-2-2) edge[equality] (m-2-5)
            edge[equality] (m-3-1)
    (m-2-7) edge node[above] {\(g\)} (m-2-10)
            edge[tden] node[right] {\(p\)} (m-3-6)
            edge[equality] (m-1-6)
    (m-2-10) edge[equality] (m-1-9)
    (m-3-6) edge[equality] (m-2-5)
    (m-3-11) edge[sden] node[right] {\(i\)} (m-2-10);
\end{tikzpicture}\]
Altogether, \((\threearrowgraph{\mathcal{C}}) / {\fractionequal}\) becomes a category with \((\doublefrac{b_1}{f_1}{a_1}) (\doublefrac{b_2}{f_2}{a_2}) = \overline{c}(\doublefrac{b_1}{f_1}{a_1}, \doublefrac{b_2}{f_2}{a_2})\) for \((b_1, f_1, a_1)\), 
\((b_2, f_2, a_2) \in \Arr \threearrowgraph{\mathcal{C}}\) with \(\Target{\doublefrac{b_1}{f_1}{a_1}} = \Source{\doublefrac{b_2}{f_2}{a_2}}\) and \(1_X = e(X)\) for \(X \in \Ob {(\threearrowgraph{\mathcal{C}}) / {\fractionequal}}\).
\end{proof}

\begin{definition}[fraction category] \label{def:fraction_category}
The \newnotion{fraction category} of a uni-fractionable category \(\mathcal{C}\) is defined to be the category \(\FractionCategory \mathcal{C}\), whose underlying graph is given by the quotient graph \((\threearrowgraph{\mathcal{C}}) / {\fractionequal}\) and whose composition and identities are given as in proposition~\ref{prop:welldefinedness_of_the_fraction_category}.
\end{definition}

Our next aim is to show that the fraction category of a uni-fractionable category is a localisation, which is going to be the second part of our main theorem~\ref{th:description_of_the_fraction_category}.

\begin{remark} \label{rem:splitting_double_fractions_in_fractions}
Given a uni-fractionable category \(\mathcal{C}\), we have
\[(\doublefrac{b_1}{f_1}{1}) (\doublefrac{1}{f_2}{a_2}) = \doublefrac{b_1}{f_1 f_2}{a_2}\]
for all \(3\)-arrows \((b_1, f_1, 1)\) and \((1, f_2, a_2)\) in \(\mathcal{C}\).
\end{remark}
\begin{proof}
This follows using the definition of the composition in proposition~\ref{prop:welldefinedness_of_the_fraction_category}.
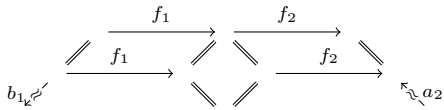
\[\begin{tikzpicture}[baseline=(m-3-1.base)]
  \matrix (m) [diagram without objects=0.9em]{
    & & & & & & & & & & \\
    & & & & & & & & & & \\
    & & & & & & & & & & \\};
  \path[->, font=\scriptsize]
    (m-1-3) edge node[above] {\(f_1\)} (m-1-6)
            edge[equality] (m-2-2)
    (m-1-6) edge node[above] {\(f_2\)} (m-1-9)
            edge[equality] (m-2-5)
    (m-2-2) edge node[above] {\(f_1\)} (m-2-5)
            edge[den] node[left] {\(b_1\)} (m-3-1)
    (m-2-7) edge node[above] {\(f_2\)} (m-2-10)
            edge[equality] (m-3-6)
            edge[equality] (m-1-6)
    (m-2-10) edge[equality] (m-1-9)
    (m-3-6) edge[equality] (m-2-5)
    (m-3-11) edge[den] node[right] {\(a_2\)} (m-2-10);
\end{tikzpicture} \qedhere\]
\end{proof}

\begin{proposition}[universal property of the fraction category] \label{prop:universal_property_of_the_fraction_category}
The fraction category \(\FractionCategory \mathcal{C}\) of a uni-fractionable category \(\mathcal{C}\) is a localisation of \(\mathcal{C}\), where the localisation functor \(\LocalisationFunctor\colon \mathcal{C} \map \FractionCategory \mathcal{C}\) is given on the objects by
\[\LocalisationFunctor(X) = X\]
for \(X \in \Ob \mathcal{C}\) and on the morphisms by
\[\LocalisationFunctor(f) = \doublefrac{1}{f}{1}\]
for \(f \in \Mor \mathcal{C}\). The inverse of \(\LocalisationFunctor(d)\) for \(d \in \Denominators \mathcal{C}\) is given by
\[(\LocalisationFunctor(d))^{- 1} = \doublefrac{d}{1}{1} = \doublefrac{1}{1}{d}.\]
Given a category \(\mathcal{D}\) and a functor \(F\colon \mathcal{C} \map \mathcal{D}\) such that \(F d\) is invertible for all \(d \in \Denominators \mathcal{C}\), the unique functor \(\hat F\colon \FractionCategory \mathcal{C} \map \mathcal{D}\) with \(F = \hat F \comp \LocalisationFunctor\) is given by
\[\hat F(\doublefrac{b}{f}{a}) = (F b)^{- 1} (F f) (F a)^{- 1}.\]
Given a category \(\mathcal{D}\) and functors \(F, G\colon \mathcal{C} \map \mathcal{D}\) such that \(F d\) and \(G d\) are invertible for all \(d \in \Denominators \mathcal{C}\) and a transformation \(\alpha\colon F \map G\), the unique transformation \(\hat \alpha\colon \hat F \map \hat G\) with \(\alpha_X = \hat \alpha_{\LocalisationFunctor(X)}\) for \(X \in \Ob \mathcal{C}\) is given by \(\hat \alpha_X = \alpha_X\) for \(X \in \Ob \FractionCategory \mathcal{C} = \Ob \mathcal{C}\).
\end{proposition}
\begin{proof}
We suppose given a uni-fractionable category \(\mathcal{C}\). We define a graph morphism \(L\colon \mathcal{C} \map \FractionCategory \mathcal{C}\) on the objects by \(L X := X\) for \(X \in \Ob \mathcal{C}\) and on the arrows by \(L f := \doublefrac{1}{f}{1}\) for \(f \in \Mor \mathcal{C}\). By remark~\ref{rem:splitting_double_fractions_in_fractions}, we get
\[L(f g) = \doublefrac{1}{f g}{1} = (\doublefrac{1}{f}{1}) (\doublefrac{1}{g}{1}) = (L f) (L g)\]
for all \(f, g \in \Mor \mathcal{C}\) with \(\Target f = \Source g\) and
\[L 1_X = \doublefrac{1_X}{1_X}{1_X} = 1_{L X}\]
for all \(X \in \Ob \mathcal{C}\), that is, \(L\) is a functor.

We want to show that \(\FractionCategory \mathcal{C}\) is a localisation of \(\mathcal{C}\) with localisation functor \(L\).
\begin{itemize}
\item[(Inv)] We let \(d \in \Denominators \mathcal{C}\) be given. By remark~\ref{rem:splitting_double_fractions_in_fractions}, we have
\begin{align*}
& (L d) (\doublefrac{1}{1}{d}) = (\doublefrac{1}{d}{1}) (\doublefrac{1}{1}{d}) = \doublefrac{1}{d}{d} = \doublefrac{1}{1}{1} = 1 \text{ and} \\
& (\doublefrac{d}{1}{1}) (L d) = (\doublefrac{d}{1}{1}) (\doublefrac{1}{d}{1}) = \doublefrac{d}{d}{1} = \doublefrac{1}{1}{1} = 1,
\end{align*}
that is, \(L d\) has a right inverse \(\doublefrac{1}{1}{d}\) and a left inverse \(\doublefrac{d}{1}{1}\). But then \(L d\) is invertible and the left and the right inverse coincide as the unique inverse of \(L d\), that is,
\[(L d)^{- 1} = \doublefrac{d}{1}{1} = \doublefrac{1}{1}{d}.\]
\item[(1-uni)] We let \(\mathcal{D}\) be a category and \(F\colon \mathcal{C} \map \mathcal{D}\) be a functor such that \(F d\) is invertible for all \(d \in \Denominators \mathcal{C}\). Since
\begin{align*}
& \Source((F b)^{- 1} (F f) (F a)^{- 1}) = \Source{(F b)^{- 1}} = \Target(F b) = F(\Target b) = F(\Source{(b, f, a)}), \\
& \Target((F b)^{- 1} (F f) (F a)^{- 1}) = \Target{(F a)^{- 1}} = \Source(F a) = F(\Source a) = F(\Target{(b, f, a)})
\end{align*}
for \((b, f, a) \in \Arr \threearrowgraph{\mathcal{C}}\), there is a graph morphism \(F'\colon \threearrowgraph{\mathcal{C}} \map \mathcal{D}\) given on the objects by \(F' X = F X\) for \(X \in \Ob \threearrowgraph{\mathcal{C}}\) and on the arrows by \(F'{(b, f, a)} = (F b)^{- 1} (F f) (F a)^{- 1}\) for \((b, f, a) \in \Arr \threearrowgraph{\mathcal{C}}\). Moreover, given \((b, f, a) \in \Arr \threearrowgraph{\mathcal{C}}\) and \(c, c' \in \Denominators \mathcal{C}\) with \(\Target c' = \Source b\), \(\Source c = \Target a\), we obtain
\begin{align*}
F'{(c' b, c' f c, a c)} & = (F(c' b))^{- 1} (F(c' f c)) (F(a c))^{- 1} = ((F c') (F b))^{- 1} ((F c') (F f) (F c)) ((F a) (F c))^{- 1} \\
& = (F b)^{- 1} (F c')^{- 1} (F c') (F f) (F c) (F c)^{- 1} (F a)^{- 1} = (F b)^{- 1} (F f) (F a)^{- 1} = F'{(b, f, a)}.
\end{align*}
Hence \(F'\) maps fraction equal \(3\)-arrows to the same morphism and we obtain an induced graph morphism \(\hat F\colon (\threearrowgraph{\mathcal{C}}) / {\fractionequal} \map \mathcal{D}\) with \(F' = \hat F \comp \quo\).
\[\begin{tikzpicture}[baseline=(m-2-1.base)]
  \matrix (m) [diagram]{
    \threearrowgraph{\mathcal{C}} & \mathcal{D} \\
    (\threearrowgraph{\mathcal{C}}) / {\fractionequal} & \\};
  \path[->, font=\scriptsize]
    (m-1-1) edge node[above] {\(F'\)} (m-1-2)
            edge node[left] {\(\quo\)} (m-2-1)
    (m-2-1) edge[exists] node[right] {\(\hat F\)} (m-1-2);
\end{tikzpicture}\]
Given \((b_1, f_1, a_1), (b_2, f_2, a_2) \in \Arr \threearrowgraph{\mathcal{C}}\) with \(\Target{(b_1, f_1, a_1)} = \Source{(b_2, f_2, a_2)}\), we have
\begin{align*}
\hat F((\doublefrac{b_1}{f_1}{a_1}) (\doublefrac{b_2}{f_2}{a_2})) & = \hat F(\doublefrac{q' b_1}{f_1' f_2'}{a_2 j'}) = (F(q' b_1))^{- 1} (F(f_1' f_2')) (F(a_2 j'))^{- 1} \\
& = (F b_1)^{- 1} (F q')^{- 1} (F f_1') (F f_2') (F j')^{- 1} (F a_2)^{- 1} \\
& = (F b_1)^{- 1} (F f_1) (F q)^{- 1} (F j)^{- 1} (F f_2) (F a_2)^{- 1} \\
& = (F b_1)^{- 1} (F f_1) (F a_1)^{- 1} (F b_2)^{- 1} (F f_2) (F a_2)^{- 1} \\
& = \hat F(\doublefrac{b_1}{f_1}{a_1}) \hat F(\doublefrac{b_2}{f_2}{a_2}),
\end{align*}
where \(j\), \(j'\), \(q\), \(q'\), \(f_1'\), \(f_2'\) are supposed to be constructed as in proposition~\ref{prop:welldefinedness_of_the_fraction_category}.
\[\begin{tikzpicture}[baseline=(m-3-1.base)]
  \matrix (m) [diagram without objects=0.9em]{
    & & & & & & & & & & \\
    & & & & & & & & & & \\
    & & & & & & & & & & \\};
  \path[->, font=\scriptsize]
    (m-1-3) edge node[above] {\(f_1'\)} (m-1-6)
            edge[tden] node[left] {\(q'\)} (m-2-2)
    (m-1-6) edge node[above] {\(f_2'\)} (m-1-9)
            edge[tden] node[left] {\(q\)} (m-2-5)
    (m-2-2) edge node[above] {\(f_1\)} (m-2-5)
            edge[den] node[left] {\(b_1\)} (m-3-1)
    (m-2-7) edge node[above] {\(f_2\)} (m-2-10)
            edge[den] node[right] {\(b_2\)} (m-3-6)
            edge[sden] node[right] {\(j\)} (m-1-6)
    (m-2-10) edge[sden] node[right=1pt] {\(j'\)} (m-1-9)
    (m-3-6) edge[den] node[left] {\(a_1\)} (m-2-5)
    (m-3-11) edge[den] node[right] {\(a_2\)} (m-2-10);
\end{tikzpicture}\]
Moreover, we have
\[\hat F(1_X) = \hat F(\doublefrac{1_X}{1_X}{1_X}) = (F 1_X)^{- 1} (F 1_X) (F 1_X)^{- 1} = 1_{F X}^{- 1} 1_{F X} 1_{F X}^{- 1} = 1_{F X} = 1_{\hat F X}\]
for \(X \in \Ob \FractionCategory \mathcal{C}\). This implies that \(\hat F\colon \FractionCategory \mathcal{C} \map \mathcal{D}\) is a functor, given by
\[\hat F(\doublefrac{b}{f}{a}) = F'{(b, f, a)} = (F b)^{- 1} (F f) (F a)^{- 1}\]
for every \((b, f, a) \in \Arr \threearrowgraph{\mathcal{C}}\). In particular,
\[\hat F L f = \hat F(\doublefrac{1}{f}{1}) = (F 1)^{- 1} (F f) (F 1)^{- 1} = 1^{- 1} (F f) 1^{- 1} = F f\]
for all \(f \in \Mor \mathcal{C}\), that is, \(\hat F \comp L = F\).

Conversely, given an arbitrary functor \(G\colon \FractionCategory \mathcal{C} \map \mathcal{D}\) with \(F = G \comp L\), we conclude by remark~\ref{rem:splitting_double_fractions_in_fractions} that
\begin{align*}
G(\doublefrac{b}{f}{a}) & = G((\doublefrac{b}{1}{1}) (\doublefrac{1}{f}{1}) (\doublefrac{1}{1}{a})) = G((L b)^{- 1} (L f) (L a)^{- 1}) = (G L b)^{- 1} (G L f) (G L a)^{- 1} \\
& = (F b)^{- 1} (F f) (F a)^{- 1}
\end{align*}
for \((b, f, a) \in \Arr \threearrowgraph{\mathcal{C}}\).
\item[(2-uni)] We suppose given a category \(\mathcal{D}\) and functors \(F, G\colon \mathcal{C} \map \mathcal{D}\) such that \(F d\) and \(G d\) are invertible for all \(d \in \Denominators \mathcal{C}\), and we let \(\hat F, \hat G\colon \FractionCategory \mathcal{C} \map \mathcal{D}\) be the unique functors with \(F = \hat F \comp L\) resp.\ \(G = \hat G \comp L\). Moreover, we suppose given a transformation \(\alpha\colon F \map G\). We define a family \(\hat \alpha := (\hat \alpha_X)_{X \in \Ob \FractionCategory \mathcal{C}}\) by \(\hat \alpha_X := \alpha_X\) for \(X \in \Ob \FractionCategory \mathcal{C} = \Ob \mathcal{C}\). Then \(\hat \alpha_{L X} = \hat \alpha_X = \alpha_X\) for \(X \in \Ob \mathcal{C}\). Moreover, \(\hat \alpha\) is a transformation from \(\hat F\) to \(\hat G\) since for every \(3\)-arrow \((b, f, a)\colon \threearrow{X}{\tilde X}{\tilde Y}{Y}\) in \(\mathcal{C}\), we have
\begin{align*}
\hat \alpha_X (\hat G(\doublefrac{b}{f}{a})) & = \alpha_X (G b)^{- 1} (G f) (G a)^{- 1} = (F b)^{- 1} \alpha_{\tilde X} (G f) (G a)^{- 1} = (F b)^{- 1} (F f) \alpha_{\tilde Y} (G a)^{- 1} \\
& = (F b)^{- 1} (F f) (F a)^{- 1} \alpha_Y = (\hat F(\doublefrac{b}{f}{a})) \hat \alpha_Y.
\end{align*}
Conversely, given an arbitrary transformation \(\beta\colon \hat F \map \hat G\) such that \(\beta_{L X} = \alpha_X\) for all \(X \in \Ob \mathcal{C}\), we necessarily have \(\beta_X = \beta_{L X} = \alpha_X\) for all \(X \in \Ob \FractionCategory \mathcal{C} = \Ob \mathcal{C}\).
\end{itemize}

Altogether, \(\FractionCategory \mathcal{C}\) is a localisation of \(\mathcal{C}\) with localisation functor \(\LocalisationFunctor[\FractionCategory \mathcal{C}] = L\).
\end{proof}

\begin{corollary}[splitting double fractions] \label{cor:splitting_double_fractions}
Given a uni-fractionable category \(\mathcal{C}\), we have
\[\doublefrac{b}{f}{a} = (\LocalisationFunctor(b))^{- 1} \LocalisationFunctor(f) (\LocalisationFunctor(a))^{- 1}\]
for each \(3\)-arrow \((b, f, a)\) in \(\mathcal{C}\).
\end{corollary}
\begin{proof}
By proposition~\ref{prop:universal_property_of_the_fraction_category}, the fraction category \(\FractionCategory \mathcal{C}\) is a localisation of \(\mathcal{C}\). In particular, \(\LocalisationFunctor(d)\) is invertible for all \(d \in \Denominators \mathcal{C}\), and hence there exists a unique functor \(\hat L\colon \FractionCategory \mathcal{C} \map \FractionCategory \mathcal{C}\) with \(\LocalisationFunctor = \hat L \comp \LocalisationFunctor\), which is given by
\[\hat L(\doublefrac{b}{f}{a}) = (\LocalisationFunctor(b))^{- 1} \LocalisationFunctor(f) (\LocalisationFunctor(a))^{- 1}\]
for all \((b, f, a) \in \Arr \threearrowgraph{\mathcal{C}}\). But since \(\LocalisationFunctor = \id_{\FractionCategory \mathcal{C}} \comp \LocalisationFunctor\), we necessarily must have \(\hat L = \id_{\FractionCategory \mathcal{C}}\) and therefore the assertion holds.
\end{proof}

In the construction of the composition of the fraction category in proposition~\ref{prop:welldefinedness_of_the_fraction_category}, the occurring morphisms \(j\), \(j'\) were S-denominators, and \(q\), \(q'\) were T-denominators. We shall now show that it suffices to have a diagram with arbitrary denominators at their places to get the correct composite.

\begin{proposition} \label{prop:flexibility_of_composites_and_inverses_in_the_fraction_category}
We suppose given a uni-fractionable category \(\mathcal{C}\).
\begin{enumerate}
\item \label{prop:flexibility_of_composites_and_inverses_in_the_fraction_category:composition} We suppose given \(3\)-arrows \((b_1, f_1, a_1)\) and \((b_2, f_2, a_2)\) in \(\mathcal{C}\) with \(\Target{(b_1, f_1, a_1)} = \Source{(b_2, f_2, a_2)}\). Moreover, we suppose given denominators \(d\), \(d'\), \(e\), \(e'\) and morphisms \(g_1\), \(g_2\) in \(\mathcal{C}\) with \(b_2 a_1 = d e\), \(g_1 e = e' f_1\), \(d g_2 = f_2 d'\).
\[\begin{tikzpicture}[baseline=(m-3-1.base)]
  \matrix (m) [diagram without objects=0.9em]{
    & & & & & & & & & & \\
    & & & & & & & & & & \\
    & & & & & & & & & & \\};
  \path[->, font=\scriptsize]
    (m-1-3) edge node[above] {\(g_1\)} (m-1-6)
            edge[den] node[left] {\(e'\)} (m-2-2)
    (m-1-6) edge node[above] {\(g_2\)} (m-1-9)
            edge[den] node[left] {\(e\)} (m-2-5)
    (m-2-2) edge node[above] {\(f_1\)} (m-2-5)
            edge[den] node[left] {\(b_1\)} (m-3-1)
    (m-2-7) edge node[above] {\(f_2\)} (m-2-10)
            edge[den] node[right] {\(b_2\)} (m-3-6)
            edge[den] node[right] {\(d\)} (m-1-6)
    (m-2-10) edge[den] node[right] {\(d'\)} (m-1-9)
    (m-3-6) edge[den] node[left] {\(a_1\)} (m-2-5)
    (m-3-11) edge[den] node[right] {\(a_2\)} (m-2-10);
\end{tikzpicture}\]
Then we have
\[(\doublefrac{b_1}{f_1}{a_1}) (\doublefrac{b_2}{f_2}{a_2}) = \doublefrac{e' b_1}{g_1 g_2}{a_2 d'}.\]
\item \label{prop:flexibility_of_composites_and_inverses_in_the_fraction_category:inverse} Given a \(3\)-arrow \((b, d, a)\) in \(\mathcal{C}\) with a denominator \(d\), the double fraction \(\doublefrac{b}{d}{a}\) is invertible in \(\FractionCategory \mathcal{C}\), and the inverse of \(\doublefrac{b}{d}{a}\) can be constructed as follows. We choose denominators \(d_1\), \(d_1'\), \(d_2\), \(d_2'\), \(a'\), \(b'\) in \(\mathcal{C}\) with \(d = d_1 d_2\), \(d_1 b' = b d_1'\), \(a' d_2 = d_2' a\) (\footnote{For example, one can factorise \(d\) into a composite of an S-denominator \(d_1\) and a T-denominator \(d_2\) and then construct weakly universal Ore completions.}).
\[\begin{tikzpicture}[baseline=(m-2-1.base)]
  \matrix (m) [diagram without objects=0.9em]{
    & & & & & & & & \\
    & & & & & & & & \\};
  \path[->, font=\scriptsize]
    (m-1-5) edge[den] node[right=1pt] {\(d_2\)} (m-2-6)
            edge[den] node[above] {\(b'\)} (m-1-2)
    (m-1-8) edge[den] node[right=1pt] {\(d_2'\)} (m-2-9)
            edge[den] node[above] {\(a'\)} (m-1-5)
    (m-2-1) edge[den] node[left=1pt] {\(d_1'\)} (m-1-2)
    (m-2-4) edge[den] node[above] {\(d\)} (m-2-6)
            edge[den] node[above] {\(b\)} (m-2-1)
            edge[den] node[left] {\(d_1\)} (m-1-5)
    (m-2-9) edge[den] node[above] {\(a\)} (m-2-6);
\end{tikzpicture}\]
Then we have
\[(\doublefrac{b}{d}{a})^{- 1} = \doublefrac{d_2'}{a' b'}{d_1'}.\]
\end{enumerate}
\end{proposition}
\begin{proof} \
\begin{enumerate}
\item We compute
\begin{align*}
(\doublefrac{b_1}{f_1}{a_1}) (\doublefrac{b_2}{f_2}{a_2}) & = (\LocalisationFunctor(b_1))^{- 1} \LocalisationFunctor(f_1) (\LocalisationFunctor(a_1))^{- 1} (\LocalisationFunctor(b_2))^{- 1} \LocalisationFunctor(f_2) (\LocalisationFunctor(a_2))^{- 1} \\
& = (\LocalisationFunctor(b_1))^{- 1} \LocalisationFunctor(f_1) (\LocalisationFunctor(e))^{- 1} (\LocalisationFunctor(d))^{- 1} \LocalisationFunctor(f_2) (\LocalisationFunctor(a_2))^{- 1} \\
& = (\LocalisationFunctor(b_1))^{- 1} (\LocalisationFunctor(e'))^{- 1} \LocalisationFunctor(g_1) \LocalisationFunctor(g_2) (\LocalisationFunctor(d'))^{- 1} (\LocalisationFunctor(a_2))^{- 1} \\
& = (\LocalisationFunctor(e' b_1))^{- 1} \LocalisationFunctor(g_1 g_2) (\LocalisationFunctor(a_2 d'))^{- 1} = \doublefrac{e' b_1}{g_1 g_2}{a_2 d'}.
\end{align*}
\item The double fraction \(\doublefrac{b}{d}{a} = (\LocalisationFunctor(b))^{- 1} \LocalisationFunctor(d) (\LocalisationFunctor(a))^{- 1}\) is invertible in \(\FractionCategory \mathcal{C}\) since the localisation functor \(\LocalisationFunctor\colon \mathcal{C} \map \FractionCategory \mathcal{C}\) maps denominators in \(\mathcal{C}\) to isomorphisms in \(\FractionCategory \mathcal{C}\).

Given denominators \(d_1\), \(d_1'\), \(d_2\), \(d_2'\), \(a'\), \(b'\) in \(\mathcal{C}\) with \(d = d_1 d_2\), \(d_1 b' = b d_1'\), \(a' d_2 = d_2' a\), we obtain
\begin{align*}
(\doublefrac{b}{d}{a})^{- 1} & = ((\LocalisationFunctor(b))^{- 1} \LocalisationFunctor(d) (\LocalisationFunctor(a))^{- 1})^{- 1} = \LocalisationFunctor(a) (\LocalisationFunctor(d))^{- 1} \LocalisationFunctor(b) = \LocalisationFunctor(a) (\LocalisationFunctor(d_1 d_2))^{- 1} \LocalisationFunctor(b) \\
& = \LocalisationFunctor(a) (\LocalisationFunctor(d_2))^{- 1} (\LocalisationFunctor(d_1))^{- 1} \LocalisationFunctor(b) = (\LocalisationFunctor(d_2'))^{- 1} \LocalisationFunctor(a') \LocalisationFunctor(b') (\LocalisationFunctor(d_1'))^{- 1} \\
& = (\LocalisationFunctor(d_2'))^{- 1} \LocalisationFunctor(a' b') (\LocalisationFunctor(d_1'))^{- 1} = \doublefrac{d_2'}{a' b'}{d_1'}. \qedhere
\end{align*}
\end{enumerate}
\end{proof}

The preceding proposition shows that the fraction category of a uni-fractionable category does not depend on the choice of S-denominators and T-denominators, as to be expected by the universal property of a localisation, cf.\ proposition~\ref{prop:universal_property_of_the_fraction_category}:

\begin{corollary} \label{cor:fraction_category_does_not_depend_on_s-_and_t-denominators}
Given uni-fractionable categories \(\mathcal{C}\) and \(\mathcal{C}'\) such that their underlying categories with denominators coincide, we have \(\FractionCategory \mathcal{C} = \FractionCategory \mathcal{C}'\).
\end{corollary}
\begin{proof}
By the definition of the category structure of \(\FractionCategory \mathcal{C}\), see proposition~\ref{prop:welldefinedness_of_the_fraction_category}, only the definition of the composition depends on the definition of \(\SDenominators \mathcal{C}\) and \(\TDenominators \mathcal{C}\), and proposition~\ref{prop:flexibility_of_composites_and_inverses_in_the_fraction_category}\ref{prop:flexibility_of_composites_and_inverses_in_the_fraction_category:composition} shows that this composition is in fact independent of \(\SDenominators \mathcal{C}\) and \(\TDenominators \mathcal{C}\). Analogously for \(\mathcal{C}'\), and thus we have \(\FractionCategory \mathcal{C} = \FractionCategory \mathcal{C}'\).
\end{proof}

Next, we want to turn the construction of the fraction category into a functor.

\begin{remark} \label{rem:universe_of_the_fraction_category}
We suppose given a Grothendieck universe \(\mathfrak{U}\) such that \(\threearrowshape\) is in \(\mathfrak{U}\) and a uni-fractionable category \(\mathcal{C}\). If \(\mathcal{C}\) is in \(\mathfrak{U}\), then its fraction category \(\FractionCategory \mathcal{C}\) is in \(\mathfrak{U}\).
\end{remark}
\begin{proof}
Since the underlying graph of \(\FractionCategory \mathcal{C}\) is \(\threearrowgraph{\mathcal{C}} / {\fractionequal}\) and this graph is a quotient of \(\threearrowgraph{\mathcal{C}}\), the assertion follows from remark~\ref{rem:universe_of_the_3-arrow_graph}.
\end{proof}

\begin{proposition} \label{prop:fraction_category_as_a_functor} \
\begin{enumerate}
\item \label{prop:fraction_category_as_a_functor:induced_functor} Given uni-fractionable categories \(\mathcal{C}\) and \(\mathcal{D}\) and a morphism of categories with denominators \(F\colon \mathcal{C} \map \mathcal{D}\), there exists a unique induced functor
\[{\FractionCategory F}\colon \FractionCategory \mathcal{C} \map \FractionCategory \mathcal{D}\]
with \(\LocalisationFunctor[\FractionCategory \mathcal{D}] \comp F = (\FractionCategory F) \comp \LocalisationFunctor[\FractionCategory \mathcal{C}]\). It is given on the objects by
\[(\FractionCategory F) X = F X\]
for \(X \in \Ob \FractionCategory \mathcal{C}\) and on the morphisms by
\[(\FractionCategory F)(\doublefrac{b}{f}{a}) = \doublefrac{F b}{F f}{F a}\]
for \((b, f, a) \in \Arr \threearrowgraph{\mathcal{C}}\).
\item \label{prop:fraction_category_as_a_functor:functoriality} We suppose given a Grothendieck universe \(\mathfrak{U}\) such that \(\threearrowshape\) is in \(\mathfrak{U}\). The construction in~\ref{prop:fraction_category_as_a_functor:induced_functor} yields functors
\begin{align*}
& \FractionCategory\colon \UFrCat_{(\mathfrak{U})} \map \Cat_{(\mathfrak{U})} \text{ and} \\
& \FractionCategory\colon \UFr(\CatD_{(\mathfrak{U})}) \map \Cat_{(\mathfrak{U})}.
\end{align*}
\end{enumerate}
\end{proposition}
\begin{proof} \
\begin{enumerate}
\item Since \(F\) preserves denominators and \(\LocalisationFunctor[\FractionCategory \mathcal{D}]\) maps denominators in \(\mathcal{D}\) to isomorphisms in \(\FractionCategory \mathcal{D}\), the composite \(\LocalisationFunctor[\FractionCategory \mathcal{D}] \comp F\) maps denominators in \(\mathcal{C}\) to isomorphisms in \(\FractionCategory \mathcal{D}\). Hence, by the universal property of \(\FractionCategory \mathcal{C}\), there exists a unique functor \({\FractionCategory F}\colon \FractionCategory \mathcal{C} \map \FractionCategory \mathcal{D}\) with \(\LocalisationFunctor[\FractionCategory \mathcal{D}] \comp F = (\FractionCategory F) \comp \LocalisationFunctor[\FractionCategory \mathcal{C}]\). It follows that
\[(\FractionCategory F) X = (\FractionCategory F) \LocalisationFunctor(X) = \LocalisationFunctor(F X) = F X\]
for \(X \in \Ob \mathcal{C}\) as well as
\begin{align*}
(\FractionCategory F)(\doublefrac{b}{f}{a}) & = (\FractionCategory F)((\LocalisationFunctor(b))^{- 1} \LocalisationFunctor(f) (\LocalisationFunctor(a))^{- 1}) \\
& = ((\FractionCategory F) \LocalisationFunctor(b))^{- 1} ((\FractionCategory F) \LocalisationFunctor(f)) ((\FractionCategory F) \LocalisationFunctor(a))^{- 1} \\
& = (\LocalisationFunctor(F b))^{- 1} \LocalisationFunctor(F f) (\LocalisationFunctor(F a))^{- 1} = \doublefrac{F b}{F f}{F a}
\end{align*}
for \((b, f, a) \in \Arr \threearrowgraph{\mathcal{C}}\).
\item We suppose given uni-fractionable categories \(\mathcal{C}\), \(\mathcal{D}\), \(\mathcal{E}\) in \(\mathfrak{U}\) and morphisms of categories with denominators \(F\colon \mathcal{C} \map \mathcal{D}\) and \(G\colon \mathcal{D} \map \mathcal{E}\). Then we have
\begin{align*}
& \LocalisationFunctor[\FractionCategory \mathcal{E}] \comp G \comp F = (\FractionCategory G) \comp \LocalisationFunctor[\FractionCategory \mathcal{D}] \comp F = (\FractionCategory G) \comp (\FractionCategory F) \comp \LocalisationFunctor[\FractionCategory \mathcal{C}] \text{ and} \\
& \LocalisationFunctor[\FractionCategory \mathcal{C}] \comp \id_{\mathcal{C}} = \LocalisationFunctor[\FractionCategory \mathcal{C}] = \id_{\FractionCategory \mathcal{C}} \comp \LocalisationFunctor[\FractionCategory \mathcal{C}],
\end{align*}
so by the uniqueness of the induced functor in~\ref{prop:fraction_category_as_a_functor:induced_functor}, we obtain \(\FractionCategory(G \comp F) = (\FractionCategory G) \comp (\FractionCategory F)\) and \(\FractionCategory \id_{\mathcal{C}} = \id_{\FractionCategory \mathcal{C}}\). Since every morphism of uni-fractionable categories is in particular a morphism of categories with denominators, we have a functor
\[\FractionCategory\colon \UFrCat \map \Cat.\]
Moreover, since the fraction category of a uni-fractionable category does not depend on the choice of S-denominators and T-denominators by corollary~\ref{cor:fraction_category_does_not_depend_on_s-_and_t-denominators}, we even have a functor
\[\FractionCategory\colon \UFr(\CatD) \map \Cat. \qedhere\]
\end{enumerate}
\end{proof}

Here is another elementary property of the fraction category, which will be needed in section~\ref{sec:co_products_and_additive_uni-fractionable_categories}, when we deal with (co)products.

\begin{proposition} \label{prop:morphisms_in_the_fraction_category_can_have_the_same_denominators}
We suppose given a uni-fractionable category \(\mathcal{C}\) and morphisms \(\varphi_1\) and \(\varphi_2\) in \(\FractionCategory \mathcal{C}\).
\begin{enumerate}
\item \label{prop:morphisms_in_the_fraction_category_can_have_the_same_denominators:equal_source} If \(\Source \varphi_1 = \Source \varphi_2\), then there exist normal \(3\)-arrows \((p, f_1, i_1)\) and \((p, f_2, i_2)\) in \(\mathcal{C}\) with \(\varphi_1 = \doublefrac{p}{f_1}{i_1}\) and \(\varphi_2 = \doublefrac{p}{f_2}{i_2}\).
\item \label{prop:morphisms_in_the_fraction_category_can_have_the_same_denominators:equal_target} If \(\Target \varphi_1 = \Target \varphi_2\), then there exist normal \(3\)-arrows \((p_1, f_1, i)\) and \((p_2, f_2, i)\) in \(\mathcal{C}\) with \(\varphi_1 = \doublefrac{p_1}{f_1}{i}\) and \(\varphi_2 = \doublefrac{p_2}{f_2}{i}\).
\item \label{prop:morphisms_in_the_fraction_category_can_have_the_same_denominators:parallel} If \(\varphi_1\) and \(\varphi_2\) are parallel, then there exist normal \(3\)-arrows \((p, f_1, i)\) and \((p, f_2, i)\) in \(\mathcal{C}\) with \(\varphi_1 = \doublefrac{p}{f_1}{i}\) and \(\varphi_2 = \doublefrac{p}{f_2}{i}\).
\end{enumerate}
\end{proposition}
\begin{proof}
This follows from corollary~\ref{cor:3-arrows_with_common_denominators}.
\end{proof}

Our next aim is to give a sufficient (and necessary) criterion for saturatedness.

\begin{proposition}[{cf.~\cite[sec.~36.4]{dwyer_hirschhorn_kan_smith:2004:homotopy_limit_functors_on_model_categories_and_homotopical_categories}}] \label{prop:weakly_saturated_uni-fractionable_category_is_a_saturated_uni-fractionable_category}
A uni-fractionable category \(\mathcal{C}\) is saturated if and only if it is weakly saturated.
\end{proposition}
\begin{proof}
We suppose given a uni-fractionable category \(\mathcal{C}\). Since saturatedness always implies weak saturatedness, it suffices to show that if \(\mathcal{C}\) is weakly saturated, then it is already saturated. So we suppose that \(\mathcal{C}\) is weakly saturated and we suppose given a morphism \(f\) in \(\mathcal{C}\) such that \(\LocalisationFunctor(f)\) is invertible in \(\FractionCategory \mathcal{C}\). We let \((p, g, i)\) be a normal \(3\)-arrow in \(\mathcal{C}\) with \((\LocalisationFunctor(f))^{- 1} = \doublefrac{p}{g}{i}\). Moreover, we choose a T-denominator \(p'\) and a morphism \(f'\) in \(\mathcal{C}\) with \(f' p = p' f\), and we choose an S-denominator \(i'\) and a morphism \(f''\) in \(\mathcal{C}\) with \(i f'' = f i'\).
\[\begin{tikzpicture}[baseline=(m-3-1.base)]
  \matrix (m) [diagram without objects=0.9em]{
    & & & & & & & & & & & & & & & \\
    & & & & & & & & & & & & & & & \\
    & & & & & & & & & & & & & & & \\};
  \path[->, font=\scriptsize]
    (m-1-3) edge node[above] {\(f'\)} (m-1-6)
            edge[tden] node[left] {\(p'\)} (m-2-2)
    (m-1-6) edge node[above] {\(g\)} (m-1-11)
            edge[tden] node[left] {\(p\)} (m-2-5)
    (m-1-11) edge node[above] {\(f''\)} (m-1-14)
             edge[equality] (m-2-10)
    (m-2-2) edge node[above] {\(f\)} (m-2-5)
            edge[equality] (m-3-1)
    (m-2-7) edge node[above] {\(g\)} (m-2-10)
            edge[tden] node[right] {\(p\)} (m-3-6)
            edge[equality] (m-1-6)
    (m-2-12) edge node[above] {\(f\)} (m-2-15)
             edge[equality] (m-3-11)
             edge[sden] node[right] {\(i\)} (m-1-11)
    (m-2-15) edge[sden] node[right] {\(i'\)} (m-1-14)
    (m-3-6) edge[equality] (m-2-5)
    (m-3-11) edge[sden] node[left=1pt] {\(i\)} (m-2-10)
    (m-3-16) edge[equality] (m-2-15);
\end{tikzpicture}\]
Then we have
\begin{align*}
& \doublefrac{1}{1}{1} = (\doublefrac{1}{f}{1}) (\doublefrac{p}{g}{i}) = \doublefrac{p'}{f' g}{i} \text{ and} \\
& \doublefrac{1}{1}{1} = (\doublefrac{p}{g}{i}) (\doublefrac{1}{f}{1}) = \doublefrac{p}{g f''}{i'}.
\end{align*}
We conclude that \(f' g\) and \(g f''\) must be denominators by remark~\ref{rem:invertibility_of_3-arrows_is_preserved_by_fraction_equality}. Hence (2\,of\,6) implies that \(f'\) and thus \(f\) is a denominator. Altogether, \(\mathcal{C}\) is saturated.
\end{proof}

\begin{corollary} \label{cor:description_of_isomorphisms_in_the_fraction_category_of_a_weakly_saturated_uni-fractionable_category}
The set of isomorphisms in the fraction category of a weakly saturated uni-fractionable category \(\mathcal{C}\) is given by
\[\Isomorphisms \FractionCategory \mathcal{C} = \{\doublefrac{b}{f}{a} \mid \text{\((b, f, a) \in \Arr \threearrowgraph{\mathcal{C}}\) with \(f \in \Denominators \mathcal{C}\)}\}.\]
\end{corollary}
\begin{proof}
Given a \(3\)-arrow \((b, f, a) \in \Arr \threearrowgraph{\mathcal{C}}\) with \(f \in \Denominators \mathcal{C}\), we have \(\LocalisationFunctor(b), \LocalisationFunctor(f), \LocalisationFunctor(a) \in \Isomorphisms \FractionCategory \mathcal{C}\) and hence \(\doublefrac{b}{f}{a} = (\LocalisationFunctor(b))^{- 1} \LocalisationFunctor(f) (\LocalisationFunctor(a))^{- 1} \in \Isomorphisms \FractionCategory \mathcal{C}\). Conversely, we suppose given an isomorphism \(\varphi \in \Isomorphisms \FractionCategory \mathcal{C}\) and we choose a \(3\)-arrow \((b, f, a) \in \Arr \threearrowgraph{\mathcal{C}}\) with \(\varphi = \doublefrac{b}{f}{a}\). Since \(a, b \in \Denominators \mathcal{C}\), we also have \(\LocalisationFunctor(b), \LocalisationFunctor(a) \in \Isomorphisms \FractionCategory \mathcal{C}\) and thus \(\LocalisationFunctor(f) = \LocalisationFunctor(b) \, \varphi \, \LocalisationFunctor(a) \in \Isomorphisms \FractionCategory \mathcal{C}\). But \(\mathcal{C}\) is saturated by proposition~\ref{prop:weakly_saturated_uni-fractionable_category_is_a_saturated_uni-fractionable_category}, whence \(f \in \Denominators \mathcal{C}\) follows.
\end{proof}

Now we come to the last part of the main theorem of this article, that is, we want to show that the uni-fractionable category \(\mathcal{C}\) admits a \(3\)-arrow calculus. It can be found in proposition~\ref{prop:3-arrow_calculus}. The key step of its proof is treated in the following lemma.

\begin{lemma}[flipping lemma] \label{lem:flipping_lemma}
We suppose given a uni-fractionable category \(\mathcal{C}\). Moreover, we suppose given \(3\)-arrows \((b_1, f_1, a_1)\), \((b_2, f_2, a_2)\), \((v_1, h_1, u_1)\), \((v_2, h_2, u_2)\), morphisms \(g_1\), \(g_1'\), \(g_1''\), \(g_2\), \(g_2'\), \(g_2''\), denominators \(d\), \(e\), an S{\nbd}denominator \(i_2\) and a T-denominator \(p_1\) in \(\mathcal{C}\), fitting into the following commutative diagram in~\(\mathcal{C}\).
\[\begin{tikzpicture}[baseline=(m-4-1.base)]
  \matrix (m) [diagram without objects]{
    & & & \\
    & & & \\
    & & & \\
    & & & \\};
  \path[->, font=\scriptsize]
    (m-1-1) edge[equality] (m-2-1)
    (m-1-2) edge node[above] {\(f_1\)} (m-1-3)
            edge node[right] {\(g_2''\)} (m-2-2)
            edge[den] node[above] {\(b_1\)} (m-1-1)
    (m-1-3) edge node[right] {\(g_2'\)} (m-2-3)
    (m-1-4) edge node[right] {\(g_2\)} (m-2-4)
            edge[den] node[above] {\(a_1\)} (m-1-3)
    (m-2-2) edge node[above] {\(h_1\)} (m-2-3)
            edge[den] node[above] {\(v_1\)} (m-2-1)
    (m-2-4) edge[den] node[above] {\(u_1\)} (m-2-3)
    (m-3-1) edge node[right] {\(g_1\)} (m-4-1)
            edge[tden] node[right] {\(p_1\)} (m-2-1)
    (m-3-2) edge node[above] {\(h_2\)} (m-3-3)
            edge node[right] {\(g_1'\)} (m-4-2)
            edge[den] node[above] {\(v_2\)} (m-3-1)
            edge[den] node[right] {\(d\)} (m-2-2)
    (m-3-3) edge node[right] {\(g_1''\)} (m-4-3)
            edge[den] node[right] {\(e\)} (m-2-3)
    (m-3-4) edge[equality] (m-4-4)
            edge[den] node[above] {\(u_2\)} (m-3-3)
            edge[sden] node[right] {\(i_2\)} (m-2-4)
    (m-4-2) edge node[above] {\(f_2\)} (m-4-3)
            edge[den] node[above] {\(b_2\)} (m-4-1)
    (m-4-4) edge[den] node[above] {\(a_2\)} (m-4-3);
\end{tikzpicture}\]
Then there exist \(3\)-arrows \((\tilde b_1, \tilde f_1, \tilde a_1)\), \((\tilde b_2, \tilde f_2, \tilde a_2)\) and normal \(3\)-arrows \((\tilde p_1, \tilde g_1, \tilde i_1)\), \((\tilde p_2, \tilde g_2, \tilde i_2)\) in \(\mathcal{C}\), fitting into the following commutative diagram in \(\mathcal{C}\).
\[\begin{tikzpicture}[baseline=(m-4-1.base)]
  \matrix (m) [diagram without objects]{
    & & & \\
    & & & \\
    & & & \\
    & & & \\};
  \path[->, font=\scriptsize]
    (m-1-2) edge node[above] {\(f_1\)} (m-1-3)
            edge[den] node[above] {\(b_1\)} (m-1-1)
    (m-1-4) edge[den] node[above] {\(a_1\)} (m-1-3)
    (m-2-1) edge node[right] {\(g_1\)} (m-3-1)
            edge[tden] node[right] {\(p_1\)} (m-1-1)
    (m-2-2) edge[exists] node[above] {\(\tilde f_1\)} (m-2-3)
            edge[exists] node[right] {\(\tilde g_1\)} (m-3-2)
            edge[exists, den] node[above] {\(\tilde b_1\)} (m-2-1)
            edge[exists, tden] node[right] {\(\tilde p_1\)} (m-1-2)
    (m-2-3) edge[exists] node[right] {\(\tilde g_2\)} (m-3-3)
            edge[exists, tden] node[right] {\(\tilde p_2\)} (m-1-3)
    (m-2-4) edge node[right] {\(g_2\)} (m-3-4)
            edge[exists, den] node[above] {\(\tilde a_1\)} (m-2-3)
            edge[equality] (m-1-4)
    (m-3-2) edge[exists] node[above] {\(\tilde f_2\)} (m-3-3)
            edge[exists, den] node[above] {\(\tilde b_2\)} (m-3-1)
    (m-3-4) edge[exists, den] node[above] {\(\tilde a_2\)} (m-3-3)
    (m-4-1) edge[equality] (m-3-1)
    (m-4-2) edge node[above] {\(f_2\)} (m-4-3)
            edge[den] node[above] {\(b_2\)} (m-4-1)
            edge[exists, sden] node[right] {\(\tilde i_1\)} (m-3-2)
    (m-4-3) edge[exists, sden] node[right] {\(\tilde i_2\)} (m-3-3)
    (m-4-4) edge[den] node[above] {\(a_2\)} (m-4-3)
            edge[sden] node[right] {\(i_2\)} (m-3-4);
\end{tikzpicture}\]
\end{lemma}
\begin{proof}
By corollary~\ref{cor:factorisation_lemma_for_one_given_factorisation}\ref{cor:factorisation_lemma_for_one_given_factorisation:s-2-arrow}, there exist S-denominators \(j_1\), \(\tilde j_2\), T-denominators \(q_1\), \(\tilde q_2\) and morphisms \(b\), \(\tilde a\) in \(\mathcal{C}\) with \(d = j_1 q_1\), \(e = \tilde j_2 \tilde q_2\), \(q_1 v_1 = b p_1\), \(v_2 = j_1 b\), \(u_1 = \tilde a \tilde q_2\), \(u_2 \tilde j_2 = i_2 \tilde a\).
\[\begin{tikzpicture}[baseline=(m-3-1.base)]
  \matrix (m) [diagram without objects]{
    & & & & & & & \\
    & & & & & & & \\
    & & & & & & & \\};
  \path[->, font=\scriptsize]
    (m-1-4) edge node[above] {\(h_1\)} (m-1-6)
            edge[den] node[above] {\(v_1\)} (m-1-2)
    (m-1-8) edge[den] node[above] {\(u_1\)} (m-1-6)
    (m-3-2) edge[equality] (m-2-1)
            edge[tden=0.25] node[right, near start] {\(p_1\)} (m-1-2)
    (m-3-4) edge node[above] {\(h_2\)} (m-3-6)
            edge[den] node[above] {\(v_2\)} (m-3-2)
            edge[exists, sden] node[left] {\(j_1\)} (m-2-3)
            edge[den] node[right] {\(d\)} (m-1-4)
    (m-3-6) edge[exists, sden] node[left] {\(\tilde j_2\)} (m-2-5)
            edge[den=0.25] node[right, near start] {\(e\)} (m-1-6)
    (m-3-8) edge[den] node[above] {\(u_2\)} (m-3-6)
            edge[sden] node[left] {\(i_2\)} (m-2-7)
            edge[sden] node[right] {\(i_2\)} (m-1-8)
    (m-2-1) edge[tden] node[left] {\(p_1\)} (m-1-2)
    (m-2-3) edge[exists, tden] node[left] {\(q_1\)} (m-1-4)
            edge[cross line, exists, den=0.25] node[above, near start] {\(b\)} (m-2-1)
    (m-2-5) edge[exists, tden] node[left] {\(\tilde q_2\)} (m-1-6)
    (m-2-7) edge[equality] (m-1-8)
            edge[cross line, exists, den=0.25] node[above, near start] {\(\tilde a\)} (m-2-5);
\end{tikzpicture}\]
Next, using the factorisation lemma~\ref{lem:factorisation_lemma}\ref{lem:factorisation_lemma:s-2-arrow}, there exist an S-denominator \(j_2\), a T-denominator \(q_2\), a morphism \(f\) and a denominator \(\tilde a'\) in \(\mathcal{C}\) with \(e = j_2 q_2\), \(q_1 h_1 = f q_2\), \(j_1 f = h_2 j_2\), \(\tilde q_2 = \tilde a' q_2\), \(j_2 = \tilde j_2 \tilde a'\).
\[\begin{tikzpicture}[baseline=(m-3-1.base)]
  \matrix (m) [diagram without objects]{
    & & & & & \\
    & & & & & \\
    & & & & & \\};
  \path[->, font=\scriptsize]
    (m-1-2) edge node[above] {\(h_1\)} (m-1-4)
    (m-1-6) edge[equality] (m-1-4)
    (m-3-2) edge node[above] {\(h_2\)} (m-3-4)
            edge[sden] node[left] {\(j_1\)} (m-2-1)
            edge[den=0.25] node[right, near start] {\(d\)} (m-1-2)
    (m-3-4) edge[exists, sden] node[left] {\(j_2\)} (m-2-3)
            edge[den=0.25] node[right, near start] {\(e\)} (m-1-4)
    (m-3-6) edge[equality] (m-3-4)
            edge[sden] node[left] {\(\tilde j_2\)} (m-2-5)
            edge[den] node[right] {\(e\)} (m-1-6)
    (m-2-1) edge[cross line, exists] node[above, near end] {\(f\)} (m-2-3)
            edge[tden] node[left] {\(q_1\)} (m-1-2)
    (m-2-3) edge[exists, tden] node[left] {\(q_2\)} (m-1-4)
    (m-2-5) edge[tden] node[left] {\(\tilde q_2\)} (m-1-6)
            edge[cross line, exists, den=0.25] node[above, near start] {\(\tilde a'\)} (m-2-3);
\end{tikzpicture}\]
We set \(a := \tilde a \tilde a'\) and obtain \(u_1 = a q_2\) and \(u_2 j_2 = i_2 a\).

Next, we choose weak pullback rectangles
\[\begin{tikzpicture}[baseline=(m-2-1.base)]
  \matrix (m) [diagram without objects]{
    & \\
    & \\};
  \path[->, font=\scriptsize]
    (m-1-1) edge node[above] {\(\tilde g_1'\)} (m-1-2)
            edge[tden] node[left] {\(\tilde p_1\)} (m-2-1)
    (m-1-2) edge[tden] node[right] {\(q_1\)} (m-2-2)
    (m-2-1) edge node[above] {\(g_2''\)} (m-2-2);
\end{tikzpicture}
\text{ and }
\begin{tikzpicture}[baseline=(m-2-1.base)]
  \matrix (m) [diagram without objects]{
    & \\
    & \\};
  \path[->, font=\scriptsize]
    (m-1-1) edge node[above] {\(\tilde g_2'\)} (m-1-2)
            edge[tden] node[left] {\(\tilde p_2\)} (m-2-1)
    (m-1-2) edge[tden] node[right] {\(q_2\)} (m-2-2)
    (m-2-1) edge node[above] {\(g_2'\)} (m-2-2);
\end{tikzpicture}\]
in \(\mathcal{C}\) such that \(\tilde p_1\) and \(\tilde p_2\) are T-denominators, and we choose weak pushout rectangles
\[\begin{tikzpicture}[baseline=(m-2-1.base)]
  \matrix (m) [diagram without objects]{
    & \\
    & \\};
  \path[->, font=\scriptsize]
    (m-1-1) edge node[above] {\(\tilde g_1''\)} (m-1-2)
    (m-2-1) edge node[above] {\(g_1'\)} (m-2-2)
            edge[sden] node[left] {\(j_1\)} (m-1-1)
    (m-2-2) edge[sden] node[right] {\(\tilde i_1\)} (m-1-2);
\end{tikzpicture}
\text{ and }
\begin{tikzpicture}[baseline=(m-2-1.base)]
  \matrix (m) [diagram without objects]{
    & \\
    & \\};
  \path[->, font=\scriptsize]
    (m-1-1) edge node[above] {\(\tilde g_2''\)} (m-1-2)
    (m-2-1) edge node[above] {\(g_1''\)} (m-2-2)
            edge[sden] node[left] {\(j_2\)} (m-1-1)
    (m-2-2) edge[sden] node[right] {\(\tilde i_2\)} (m-1-2);
\end{tikzpicture}\]
in \(\mathcal{C}\) such that \(\tilde i_1\) and \(\tilde i_2\) are S-denominators. We obtain induced morphisms \(\tilde b_1\), \(\tilde f_1\), \(\tilde a_1\) on the weak pullbacks, that is, with \(\tilde p_1 b_1 = \tilde b_1 p_1\), \(\tilde b_1 = \tilde g_1' b\), \(\tilde p_1 f_1 = \tilde f_1 \tilde p_2\), \(\tilde f_1 \tilde g_2' = \tilde g_1' f\), \(a_1 = \tilde a_1 \tilde p_2\), \(\tilde a_1 \tilde g_2' = g_2 a\), and induced morphisms \(\tilde b_2\), \(\tilde f_2\), \(\tilde a_2\) on the weak pushouts, that is, with \(b g_1 = \tilde g_1'' \tilde b_2\), \(\tilde i_1 \tilde b_2 = b_2\), \(f \tilde g_2'' = \tilde g_1'' \tilde f_2\), \(\tilde i_1 \tilde f_2 = f_2 \tilde i_2\), \(a \tilde g_2'' = \tilde a_2\), \(i_2 \tilde a_2 = a_2 \tilde i_2\).
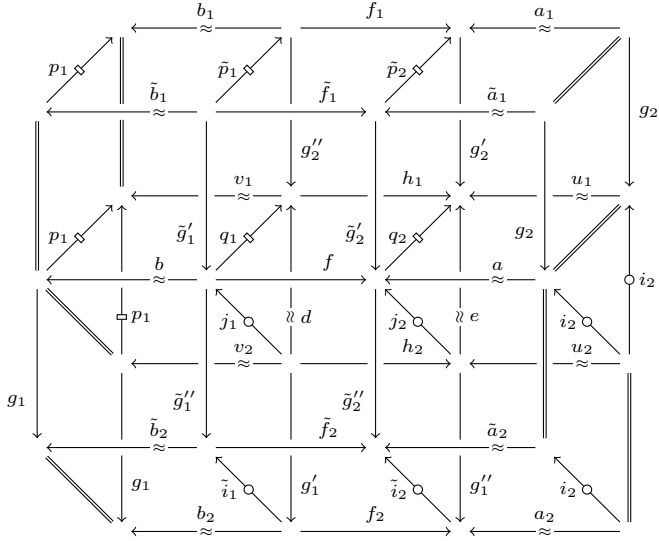
\[\begin{tikzpicture}[baseline=(m-7-1.base)]
  \matrix (m) [diagram without objects]{
    & & & & & & & \\
    & & & & & & & \\
    & & & & & & & \\
    & & & & & & & \\
    & & & & & & & \\
    & & & & & & & \\
    & & & & & & & \\};
  \path[->, font=\scriptsize]
    (m-1-2) edge[equality] (m-3-2)
    (m-1-4) edge node[above] {\(f_1\)} (m-1-6)
            edge node[right, near end] {\(g_2''\)} (m-3-4)
            edge[den] node[above] {\(b_1\)} (m-1-2)
    (m-1-6) edge node[right, near end] {\(g_2'\)} (m-3-6)
    (m-1-8) edge node[right] {\(g_2\)} (m-3-8)
            edge[den] node[above] {\(a_1\)} (m-1-6)
    (m-3-4) edge node[above, near end] {\(h_1\)} (m-3-6)
            edge[den=0.25] node[above, near start] {\(v_1\)} (m-3-2)
    (m-3-8) edge[den=0.25] node[above, near start] {\(u_1\)} (m-3-6)
    (m-5-2) edge node[right, near end] {\(g_1\)} (m-7-2)
            edge[equality] (m-4-1)
            edge[tden=0.25] node[right, near start] {\(p_1\)} (m-3-2)
    (m-5-4) edge node[above, near end] {\(h_2\)} (m-5-6)
            edge node[right, near end] {\(g_1'\)} (m-7-4)
            edge[den=0.25] node[above, near start] {\(v_2\)} (m-5-2)
            edge[sden] node[left] {\(j_1\)} (m-4-3)
            edge[den=0.25] node[right, near start] {\(d\)} (m-3-4)
    (m-5-6) edge node[right, near end] {\(g_1''\)} (m-7-6)
            edge[sden] node[left] {\(j_2\)} (m-4-5)
            edge[den=0.25] node[right, near start] {\(e\)} (m-3-6)
    (m-5-8) edge[equality] (m-7-8)
            edge[den=0.25] node[above, near start] {\(u_2\)} (m-5-6)
            edge[sden] node[left] {\(i_2\)} (m-4-7)
            edge[sden] node[right] {\(i_2\)} (m-3-8)
    (m-7-2) edge[equality] (m-6-1)
    (m-7-4) edge node[above] {\(f_2\)} (m-7-6)
            edge[den] node[above] {\(b_2\)} (m-7-2)
            edge[sden] node[left] {\(\tilde i_1\)} (m-6-3)
    (m-7-6) edge[sden] node[left] {\(\tilde i_2\)} (m-6-5)
    (m-7-8) edge[den] node[above] {\(a_2\)} (m-7-6)
            edge[sden] node[left] {\(i_2\)} (m-6-7)
    (m-2-1) edge[equality] (m-4-1)
            edge[tden] node[left] {\(p_1\)} (m-1-2)
    (m-2-3) edge[cross line] node[above, near end] {\(\tilde f_1\)} (m-2-5)
            edge[cross line] node[left, near end] {\(\tilde g_1'\)} (m-4-3)
            edge[cross line, den=0.25] node[above, near start] {\(\tilde b_1\)} (m-2-1)
            edge[tden] node[left] {\(\tilde p_1\)} (m-1-4)
    (m-2-5) edge[cross line] node[left, near end] {\(\tilde g_2'\)} (m-4-5)
            edge[tden] node[left] {\(\tilde p_2\)} (m-1-6)
    (m-2-7) edge[cross line] node[left, near end] {\(g_2\)} (m-4-7)
            edge[cross line, den=0.25] node[above, near start] {\(\tilde a_1\)} (m-2-5)
            edge[equality] (m-1-8)
    (m-4-1) edge node[left, near end] {\(g_1\)} (m-6-1)
            edge[tden] node[left] {\(p_1\)} (m-3-2)
    (m-4-3) edge[cross line] node[above, near end] {\(f\)} (m-4-5)
            edge[cross line] node[left, near end] {\(\tilde g_1''\)} (m-6-3)
            edge[cross line, den=0.25] node[above, near start] {\(b\)} (m-4-1)
            edge[tden] node[left] {\(q_1\)} (m-3-4)
    (m-4-5) edge[cross line] node[left, near end] {\(\tilde g_2''\)} (m-6-5)
            edge[tden] node[left] {\(q_2\)} (m-3-6)
    (m-4-7) edge[cross line, den=0.25] node[above, near start] {\(a\)} (m-4-5)
            edge[cross line, equality] (m-6-7)
            edge[equality] (m-3-8)
    (m-6-3) edge[cross line] node[above, near end] {\(\tilde f_2\)} (m-6-5)
            edge[cross line, den=0.25] node[above, near start] {\(\tilde b_2\)} (m-6-1)
    (m-6-7) edge[cross line, den=0.25] node[above, near start] {\(\tilde a_2\)} (m-6-5);
\end{tikzpicture}\]
Setting \(\tilde g_1 := \tilde g_1' \tilde g_1''\) and \(\tilde g_2 := \tilde g_2' \tilde g_2''\) yields \(\tilde b_1 g_1 = \tilde g_1 \tilde b_2\), \(\tilde f_1 \tilde g_2 = \tilde g_1 \tilde f_2\), \(\tilde a_1 \tilde g_2 = g_2 \tilde a_2\). Moreover, \(\tilde a_1\), \(\tilde a_2\), \(\tilde b_1\), \(\tilde b_2\) are denominators in \(\mathcal{C}\) by semi-saturatedness.
\end{proof}

\begin{proposition}[{3-arrow calculus, cf.~\cite[sec.~36.3]{dwyer_hirschhorn_kan_smith:2004:homotopy_limit_functors_on_model_categories_and_homotopical_categories}}] \label{prop:3-arrow_calculus}
We suppose given a uni-fractionable category \(\mathcal{C}\).
\begin{enumerate}
\item \label{prop:3-arrow_calculus:fraction_equality} Given parallel \(3\)-arrows \((b_1, f_1, a_1)\) and \((b_2, f_2, a_2)\) in \(\mathcal{C}\), we have
\[\doublefrac{b_1}{f_1}{a_1} = \doublefrac{b_2}{f_2}{a_2}\]
in \(\FractionCategory \mathcal{C}\) if and only if there exist \(3\)-arrows \((\tilde b_1, \tilde f_1, \tilde a_1)\), \((\tilde b_2, \tilde f_2, \tilde a_2)\) and normal \(3\)-arrows \((p_1, d_1, i_1)\), \((p_2, d_2, i_2)\) with denominators \(d_1\), \(d_2\), fitting into the following commutative diagram in \(\mathcal{C}\).
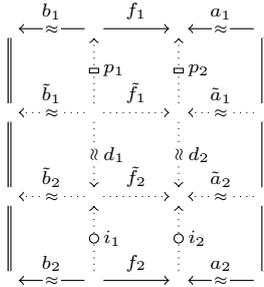
\[\begin{tikzpicture}[baseline=(m-4-1.base)]
  \matrix (m) [diagram without objects]{
    & & & \\
    & & & \\
    & & & \\
    & & & \\};
  \path[->, font=\scriptsize]
    (m-1-2) edge node[above] {\(f_1\)} (m-1-3)
            edge[den] node[above] {\(b_1\)} (m-1-1)
    (m-1-4) edge[den] node[above] {\(a_1\)} (m-1-3)
    (m-2-1) edge[equality] (m-3-1)
            edge[equality] (m-1-1)
    (m-2-2) edge[exists] node[above] {\(\tilde f_1\)} (m-2-3)
            edge[exists, den] node[right] {\(d_1\)} (m-3-2)
            edge[exists, den] node[above] {\(\tilde b_1\)} (m-2-1)
            edge[exists, tden] node[right] {\(p_1\)} (m-1-2)
    (m-2-3) edge[exists, den] node[right] {\(d_2\)} (m-3-3)
            edge[exists, tden] node[right] {\(p_2\)} (m-1-3)
    (m-2-4) edge[equality] (m-3-4)
            edge[exists, den] node[above] {\(\tilde a_1\)} (m-2-3)
            edge[equality] (m-1-4)
    (m-3-2) edge[exists] node[above] {\(\tilde f_2\)} (m-3-3)
            edge[exists, den] node[above] {\(\tilde b_2\)} (m-3-1)
    (m-3-4) edge[exists, den] node[above] {\(\tilde a_2\)} (m-3-3)
    (m-4-1) edge[equality] (m-3-1)
    (m-4-2) edge node[above] {\(f_2\)} (m-4-3)
            edge[den] node[above] {\(b_2\)} (m-4-1)
            edge[exists, sden] node[right] {\(i_1\)} (m-3-2)
    (m-4-3) edge[exists, sden] node[right] {\(i_2\)} (m-3-3)
    (m-4-4) edge[den] node[above] {\(a_2\)} (m-4-3)
            edge[equality] (m-3-4);
\end{tikzpicture}\]
If \((b_1, f_1, a_1)\) and \((b_2, f_2, a_2)\) are normal \(3\)-arrows, then \((\tilde b_1, \tilde f_1, \tilde a_1)\) and \((\tilde b_2, \tilde f_2, \tilde a_2)\) can be chosen to be normal, too.
\item \label{prop:3-arrow_calculus:commutative_diagrams} Given \(3\)-arrows \((b_1, f_1, a_1)\), \((b_2, f_2, a_2)\) and normal \(3\)-arrows \((p_1, g_1, i_1)\), \((p_2, g_2, i_2)\) in \(\mathcal{C}\), we have
\[(\doublefrac{b_1}{f_1}{a_1}) (\doublefrac{p_2}{g_2}{i_2}) = (\doublefrac{p_1}{g_1}{i_1}) (\doublefrac{b_2}{f_2}{a_2})\]
in \(\FractionCategory \mathcal{C}\) if and only if there exist \(3\)-arrows \((\tilde b_1, \tilde f_1, \tilde a_1)\), \((\tilde b_2, \tilde f_2, \tilde a_2)\) and normal \(3\)-arrows \((\tilde p_1, \tilde g_1, \tilde i_1)\), \((\tilde p_2, \tilde g_2, \tilde i_2)\), fitting into the following commutative diagram in \(\mathcal{C}\).
\[\begin{tikzpicture}[baseline=(m-4-1.base)]
  \matrix (m) [diagram without objects]{
    & & & \\
    & & & \\
    & & & \\
    & & & \\};
  \path[->, font=\scriptsize]
    (m-1-2) edge node[above] {\(f_1\)} (m-1-3)
            edge[den] node[above] {\(b_1\)} (m-1-1)
    (m-1-4) edge[den] node[above] {\(a_1\)} (m-1-3)
    (m-2-1) edge node[right] {\(g_1\)} (m-3-1)
            edge[tden] node[right] {\(p_1\)} (m-1-1)
    (m-2-2) edge[exists] node[above] {\(\tilde f_1\)} (m-2-3)
            edge[exists] node[right] {\(\tilde g_1\)} (m-3-2)
            edge[exists, den] node[above] {\(\tilde b_1\)} (m-2-1)
            edge[exists, tden] node[right] {\(\tilde p_1\)} (m-1-2)
    (m-2-3) edge[exists] node[right] {\(\tilde g_2\)} (m-3-3)
            edge[exists, tden] node[right] {\(\tilde p_2\)} (m-1-3)
    (m-2-4) edge node[right] {\(g_2\)} (m-3-4)
            edge[exists, den] node[above] {\(\tilde a_1\)} (m-2-3)
            edge[tden] node[right] {\(p_2\)} (m-1-4)
    (m-3-2) edge[exists] node[above] {\(\tilde f_2\)} (m-3-3)
            edge[exists, den] node[above] {\(\tilde b_2\)} (m-3-1)
    (m-3-4) edge[exists, den] node[above] {\(\tilde a_2\)} (m-3-3)
    (m-4-1) edge[sden] node[right] {\(i_1\)} (m-3-1)
    (m-4-2) edge node[above] {\(f_2\)} (m-4-3)
            edge[den] node[above] {\(b_2\)} (m-4-1)
            edge[exists, sden] node[right] {\(\tilde i_1\)} (m-3-2)
    (m-4-3) edge[exists, sden] node[right] {\(\tilde i_2\)} (m-3-3)
    (m-4-4) edge[den] node[above] {\(a_2\)} (m-4-3)
            edge[sden] node[right] {\(i_2\)} (m-3-4);
\end{tikzpicture}\]
\end{enumerate}
\end{proposition}
\begin{proof} \
\begin{enumerate}
\item If we have a commutative diagram as stated, then we have
\[(b_1, f_1, a_1) \fractionequal (\tilde b_1, \tilde f_1, \tilde a_1) \fractionequal (\tilde b_2, \tilde f_2, \tilde a_2) \fractionequal (b_2, f_2, a_2)\]
and thus \(\doublefrac{b_1}{f_1}{a_1} = \doublefrac{b_2}{f_2}{a_2}\) in \(\FractionCategory \mathcal{C}\).

So we suppose conversely that \(\doublefrac{b_1}{f_1}{a_1} = \doublefrac{b_2}{f_2}{a_2}\) in \(\FractionCategory \mathcal{C}\), that is, we suppose that \((b_1, f_1, a_1) \fractionequal (b_2, f_2, a_2)\) in \(\threearrowgraph{\mathcal{C}}\). By remark~\ref{rem:other_generating_sets_for_the_fraction_equality_relation_on_the_3-arrow_graph}\ref{rem:other_generating_sets_for_the_fraction_equality_relation_on_the_3-arrow_graph:equal_directions}, there exist \(n \in \Naturals_0\), \((v_l, h_l, u_l) \in \Arr \threearrowgraph{\mathcal{C}}\) for \(l \in [0, 2 n + 1]\), \(c_l, c_l' \in \Mor \mathcal{C}\) for \(l \in [0, n]\), \(w_l, w_l' \in \Mor \mathcal{C}\) for \(l \in [0, n - 1]\), with \((v_0, h_0, u_0) = (b_1, f_1, a_1)\) and \((v_{2 n + 1}, h_{2 n + 1}, u_{2 n + 1}) = (b_2, f_2, a_2)\) as well as \(v_{2 l} = c_l v_{2 l + 1}\), \(h_{2 l} c_l' = c_l h_{2 l + 1}\), \(u_{2 l} c_l' = u_{2 l + 1}\) for \(l \in [0, n]\) and \(v_{2 l + 2} = w_l v_{2 l + 1}\), \(w_l h_{2 l + 1} = h_{2 l + 2} w_l'\), \(u_{2 l + 2} w_l' = u_{2 l + 1}\) for \(l \in [0, n - 1]\).
\[\begin{tikzpicture}[baseline=(m-3-1.base)]
  \matrix (m) [diagram without objects]{
    & & & \\
    & & & \\
    & & & \\};
  \path[->, font=\scriptsize]
    (m-1-1) edge[equality] (m-2-1)
    (m-1-2) edge node[above] {\(h_{2 l}\)} (m-1-3)
            edge node[right] {\(c_l\)} (m-2-2)
            edge[den] node[above] {\(v_{2 l}\)} (m-1-1)
    (m-1-3) edge node[right] {\(c_l'\)} (m-2-3)
    (m-1-4) edge[equality] (m-2-4)
            edge[den] node[above] {\(u_{2 l}\)} (m-1-3)
    (m-2-2) edge node[above] {\(h_{2 l + 1}\)} (m-2-3)
            edge[den] node[above] {\(v_{2 l + 1}\)} (m-2-1)
    (m-2-4) edge[den] node[above] {\(u_{2 l + 1}\)} (m-2-3)
    (m-3-1) edge[equality] (m-2-1)
    (m-3-2) edge node[above] {\(h_{2 l + 2}\)} (m-3-3)
            edge[den] node[above] {\(v_{2 l + 2}\)} (m-3-1)
            edge node[right] {\(w_l\)} (m-2-2)
    (m-3-3) edge node[right] {\(w_l'\)} (m-2-3)
    (m-3-4) edge[den] node[above] {\(u_{2 l + 2}\)} (m-3-3)
            edge[equality] (m-2-4);
\end{tikzpicture}\]
By semi-saturatedness, \(c_l\) and \(c_l'\) are denominators for all \(l \in [0, n]\) and \(w_l\), \(w_l'\) are denominators for all \(l \in [0, n - 1]\). Using the flipping lemma~\ref{lem:flipping_lemma} and induction on \(n \in \Naturals_0\) yields the first assertion.

Now let us suppose that \((b_1, f_1, a_1)\) and \((b_2, f_2, a_2)\) are normal \(3\)-arrows. By multiplicativity, \(\tilde b_1 = p_1 b_1\) is a T-denominator and \(\tilde a_2 = a_2 i_2\) is an S-denominator in \(\mathcal{C}\). We choose S-denominators \(j_1\), \(j_2\) and T-denominators \(q_1\), \(q_2\) with \(\tilde a_1 = j_1 q_1\) and \(\tilde b_2 = j_2 q_2\). Moreover, we choose a T-denominator \(q_1'\) and a morphism \(\tilde f_1'\) in \(\mathcal{C}\) with \(\tilde f_1' q_1 = q_1' \tilde f_1\), and we choose an S-denominator \(j_2'\) and a morphism \(\tilde f_2'\) in \(\mathcal{C}\) with \(j_2 \tilde f_2' = \tilde f_2 j_2'\).
\[\begin{tikzpicture}[baseline=(m-2-1.base)]
  \matrix (m) [diagram without objects=0.9em]{
    & & & & \\
    & & & & \\};
  \path[->, font=\scriptsize]
    (m-1-2) edge node[above] {\(\tilde f_1'\)} (m-1-4)
            edge[tden] node[left=1pt] {\(q_1'\)} (m-2-1)
    (m-1-4) edge[tden] node[left] {\(q_1\)} (m-2-3)
    (m-2-1) edge node[below] {\(\tilde f_1\)} (m-2-3)
    (m-2-5) edge[den] node[below] {\(\tilde a_1\)} (m-2-3)
            edge[sden] node[right] {\(j_1\)} (m-1-4);
\end{tikzpicture}
\qquad
\begin{tikzpicture}[baseline=(m-2-1.base)]
  \matrix (m) [diagram without objects=0.9em]{
    & & & & \\
    & & & & \\};
  \path[->, font=\scriptsize]
    (m-1-2) edge node[above] {\(\tilde f_2'\)} (m-1-4)
            edge[tden] node[left] {\(q_2\)} (m-2-1)
    (m-2-3) edge node[below] {\(\tilde f_2\)} (m-2-5)
            edge[den] node[below] {\(\tilde b_2\)} (m-2-1)
            edge[sden] node[right] {\(j_2\)} (m-1-2)
    (m-2-5) edge[sden] node[right=1pt] {\(j_2'\)} (m-1-4);
\end{tikzpicture}\]
We obtain the following commutative diagram.
\[\begin{tikzpicture}[baseline=(m-8-1.base)]
  \matrix (m) [diagram without objects]{
    & & & \\
    & & & \\
    & & & \\
    & & & \\
    & & & \\
    & & & \\
    & & & \\
    & & & \\};
  \path[->, font=\scriptsize]
    (m-1-2) edge node[above] {\(f_1\)} (m-1-3)
            edge[tden] node[above] {\(b_1\)} (m-1-1)
    (m-1-4) edge[sden] node[above] {\(a_1\)} (m-1-3)
    (m-2-1) edge[equality] (m-1-1)
    (m-2-2) edge node[above] {\(\tilde f_1\)} (m-2-3)
            edge[tden] node[above] {\(\tilde b_1\)} (m-2-1)
            edge[tden] node[right] {\(p_1\)} (m-1-2)
    (m-2-3) edge[tden] node[right] {\(p_2\)} (m-1-3)
    (m-2-4) edge[den] node[above] {\(\tilde a_1\)} (m-2-3)
            edge[equality] (m-1-4)
    (m-3-1) edge[equality] (m-4-1)
            edge[equality] (m-2-1)
    (m-3-2) edge node[above] {\(\tilde f_1'\)} (m-3-3)
            edge[tden] node[right] {\(q_1'\)} (m-4-2)
            edge[tden] node[above] {\(q_1' \tilde b_1\)} (m-3-1)
            edge[tden] node[right] {\(q_1'\)} (m-2-2)
    (m-3-3) edge[tden] node[right] {\(q_1\)} (m-4-3)
            edge[tden] node[right] {\(q_1\)} (m-2-3)
    (m-3-4) edge[equality] (m-4-4)
            edge[sden] node[above] {\(j_1\)} (m-3-3)
            edge[equality] (m-2-4)
    (m-4-1) edge[equality] (m-5-1)
    (m-4-2) edge node[above] {\(\tilde f_1\)} (m-4-3)
            edge[den] node[right] {\(d_1\)} (m-5-2)
            edge[tden] node[above] {\(\tilde b_1\)} (m-4-1)
    (m-4-3) edge[den] node[right] {\(d_2\)} (m-5-3)
    (m-4-4) edge[equality] (m-5-4)
            edge[den] node[above] {\(\tilde a_1\)} (m-4-3)
    (m-5-1) edge[equality] (m-6-1)
    (m-5-2) edge node[above] {\(\tilde f_2\)} (m-5-3)
            edge[sden] node[right] {\(j_2\)} (m-6-2)
            edge[den] node[above] {\(\tilde b_2\)} (m-5-1)
    (m-5-3) edge[sden] node[right] {\(j_2'\)} (m-6-3)
    (m-5-4) edge[equality] (m-6-4)
            edge[sden] node[above] {\(\tilde a_2\)} (m-5-3)
    (m-6-2) edge node[above] {\(\tilde f_2'\)} (m-6-3)
            edge[tden] node[above] {\(q_2\)} (m-6-1)
    (m-6-4) edge[sden] node[above] {\(\tilde a_2 j_2'\)} (m-6-3)
    (m-7-1) edge[equality] (m-6-1)
    (m-7-2) edge node[above] {\(\tilde f_2\)} (m-7-3)
            edge[den] node[above] {\(\tilde b_2\)} (m-7-1)
            edge[sden] node[right] {\(j_2\)} (m-6-2)
    (m-7-3) edge[sden] node[right] {\(j_2'\)} (m-6-3)
    (m-7-4) edge[sden] node[above] {\(\tilde a_2\)} (m-7-3)
            edge[equality] (m-6-4)
    (m-8-1) edge[equality] (m-7-1)
    (m-8-2) edge node[above] {\(f_2\)} (m-8-3)
            edge[tden] node[above] {\(b_2\)} (m-8-1)
            edge[sden] node[right] {\(i_1\)} (m-7-2)
    (m-8-3) edge[sden] node[right] {\(i_2\)} (m-7-3)
    (m-8-4) edge[sden] node[above] {\(a_2\)} (m-8-3)
            edge[equality] (m-7-4);
\end{tikzpicture}\]
By multiplicativity, \(q_1' \tilde b_1 = q_1' p_1 b_1\), \(q_1' p_1\), \(q_1 p_2\) are T-denominators and \(\tilde a_2 j_2' = a_2 i_2 j_2'\), \(i_1 j_2\), \(i_2 j_2'\) are S{\nbd}denominators in \(\mathcal{C}\). Altogether, the diagram
\[\begin{tikzpicture}[baseline=(m-4-1.base)]
  \matrix (m) [diagram without objects]{
    & & & \\
    & & & \\
    & & & \\
    & & & \\};
  \path[->, font=\scriptsize]
    (m-1-2) edge node[above] {\(f_1\)} (m-1-3)
            edge[tden] node[above] {\(b_1\)} (m-1-1)
    (m-1-4) edge[sden] node[above] {\(a_1\)} (m-1-3)
    (m-2-1) edge[equality] (m-3-1)
            edge[equality] (m-1-1)
    (m-2-2) edge node[above] {\(\tilde f_1'\)} (m-2-3)
            edge[den] node[right, near start] {\(q_1' d_1 j_2\)} (m-3-2)
            edge[tden] node[above] {\(q_1' \tilde b_1\)} (m-2-1)
            edge[tden] node[right, near end] {\(q_1' p_1\)} (m-1-2)
    (m-2-3) edge[den] node[right, near start] {\(q_1 d_2 j_2'\)} (m-3-3)
            edge[tden] node[right, near end] {\(q_1 p_2\)} (m-1-3)
    (m-2-4) edge[equality] (m-3-4)
            edge[sden] node[above] {\(j_1\)} (m-2-3)
            edge[equality] (m-1-4)
    (m-3-2) edge node[above] {\(\tilde f_2'\)} (m-3-3)
            edge[tden] node[above] {\(q_2\)} (m-3-1)
    (m-3-4) edge[sden] node[above] {\(\tilde a_2 j_2'\)} (m-3-3)
    (m-4-1) edge[equality] (m-3-1)
    (m-4-2) edge node[above] {\(f_2\)} (m-4-3)
            edge[tden] node[above] {\(b_2\)} (m-4-1)
            edge[sden] node[right, near end] {\(i_1 j_2\)} (m-3-2)
    (m-4-3) edge[sden] node[right, near end] {\(i_2 j_2'\)} (m-3-3)
    (m-4-4) edge[sden] node[above] {\(a_2\)} (m-4-3)
            edge[equality] (m-3-4);
\end{tikzpicture}\]
commutes, and \((q_1' \tilde b_1, \tilde f_1', j_1)\), \((q_2, \tilde f_2', \tilde a_2 j_2')\), \((q_1' p_1, q_1' d_1 j_2, i_1 j_2)\), \((q_1 p_2, q_1 d_2 j_2', i_2 j_2')\) are normal \(3\)-arrows.
\item If we have a commutative diagram as stated, then proposition~\ref{prop:flexibility_of_composites_and_inverses_in_the_fraction_category}\ref{prop:flexibility_of_composites_and_inverses_in_the_fraction_category:composition} implies that
\[(\doublefrac{b_1}{f_1}{a_1}) (\doublefrac{p_2}{g_2}{i_2}) = \doublefrac{\tilde p_1 b_1}{\tilde f_1 \tilde g_2}{i_2 \tilde a_2} = \doublefrac{\tilde b_1 p_1}{\tilde g_1 \tilde f_2}{a_2 \tilde i_2} = (\doublefrac{p_1}{g_1}{i_1}) (\doublefrac{b_2}{f_2}{a_2}).\]

So we suppose conversely that \((\doublefrac{b_1}{f_1}{a_1}) (\doublefrac{p_2}{g_2}{i_2}) = (\doublefrac{p_1}{g_1}{i_1}) (\doublefrac{b_2}{f_2}{a_2})\). We construct the composites \((\doublefrac{b_1}{f_1}{a_1}) (\doublefrac{p_2}{g_2}{i_2}) = \doublefrac{q' b_1}{f_1' g_2'}{i_2 j'}\) and \((\doublefrac{p_1}{g_1}{i_1}) (\doublefrac{b_2}{f_2}{a_2}) = \doublefrac{\tilde q' p_1}{g_1' f_2'}{a_2 \tilde j'}\) as in proposition~\ref{prop:welldefinedness_of_the_fraction_category}.
\[\begin{tikzpicture}[baseline=(m-4-1.base)]
  \matrix (m) [diagram without objects]{
    & & & \\
    & & & \\
    & & & \\
    & & & \\};
  \path[->, font=\scriptsize]
    (m-1-2) edge node[above] {\(f_1\)} (m-1-3)
            edge[den] node[above] {\(b_1\)} (m-1-1)
    (m-1-4) edge[den] node[above] {\(a_1\)} (m-1-3)
    (m-2-2) edge node[above] {\(f_1'\)} (m-2-3)
            edge[tden] node[right] {\(q'\)} (m-1-2)
    (m-2-3) edge node[right] {\(g_2'\)} (m-3-3)
            edge[tden] node[right] {\(q\)} (m-1-3)
    (m-2-4) edge node[right] {\(g_2\)} (m-3-4)
            edge[sden] node[above] {\(j\)} (m-2-3)
            edge[tden] node[right] {\(p_2\)} (m-1-4)
    (m-3-4) edge[sden] node[above] {\(j'\)} (m-3-3)
    (m-4-4) edge[sden] node[right] {\(i_2\)} (m-3-4);
\end{tikzpicture}
\qquad
\begin{tikzpicture}[baseline=(m-4-1.base)]
  \matrix (m) [diagram without objects]{
    & & & \\
    & & & \\
    & & & \\
    & & & \\};
  \path[->, font=\scriptsize]
    (m-2-1) edge node[right] {\(g_1\)} (m-3-1)
            edge[tden] node[right] {\(p_1\)} (m-1-1)
    (m-2-2) edge node[right] {\(g_1'\)} (m-3-2)
            edge[tden] node[above] {\(\tilde q'\)} (m-2-1)
    (m-3-2) edge node[above] {\(f_2'\)} (m-3-3)
            edge[tden] node[above] {\(\tilde q\)} (m-3-1)
    (m-4-1) edge[sden] node[right] {\(i_1\)} (m-3-1)
    (m-4-2) edge node[above] {\(f_2\)} (m-4-3)
            edge[den] node[above] {\(b_2\)} (m-4-1)
            edge[sden] node[right] {\(\tilde j\)} (m-3-2)
    (m-4-3) edge[sden] node[right] {\(\tilde j'\)} (m-3-3)
    (m-4-4) edge[den] node[above] {\(a_2\)} (m-4-3);
\end{tikzpicture}\]
Hence the following diagrams commute.
\[\begin{tikzpicture}[baseline=(m-4-1.base)]
  \matrix (m) [diagram without objects]{
    & & & \\
    & & & \\
    & & & \\
    & & & \\};
  \path[->, font=\scriptsize]
    (m-1-2) edge node[above] {\(f_1\)} (m-1-3)
            edge[den] node[above] {\(b_1\)} (m-1-1)
    (m-1-4) edge[den] node[above] {\(a_1\)} (m-1-3)
    (m-2-1) edge[equality] (m-3-1)
            edge[equality] (m-1-1)
    (m-2-2) edge node[above] {\(f_1'\)} (m-2-3)
            edge[equality] (m-3-2)
            edge[den] node[above] {\(q' b_1\)} (m-2-1)
            edge[tden] node[right] {\(q'\)} (m-1-2)
    (m-2-3) edge node[right] {\(g_2'\)} (m-3-3)
            edge[tden] node[right] {\(q\)} (m-1-3)
    (m-2-4) edge node[right] {\(g_2\)} (m-3-4)
            edge[sden] node[above] {\(j\)} (m-2-3)
            edge[tden] node[right] {\(p_2\)} (m-1-4)
    (m-3-2) edge node[above] {\(f_1' g_2'\)} (m-3-3)
            edge[den] node[above] {\(q' b_1\)} (m-3-1)
    (m-3-4) edge[sden] node[above] {\(j'\)} (m-3-3)
    (m-4-1) edge[equality] (m-3-1)
    (m-4-2) edge node[above] {\(f_1' g_2'\)} (m-4-3)
            edge[den] node[above] {\(q' b_1\)} (m-4-1)
            edge[equality] (m-3-2)
    (m-4-3) edge[equality] (m-3-3)
    (m-4-4) edge[sden] node[above] {\(i_2 j'\)} (m-4-3)
            edge[sden] node[right] {\(i_2\)} (m-3-4);
\end{tikzpicture}
\qquad
\begin{tikzpicture}[baseline=(m-4-1.base)]
  \matrix (m) [diagram without objects]{
    & & & \\
    & & & \\
    & & & \\
    & & & \\};
  \path[->, font=\scriptsize]
    (m-1-2) edge node[above] {\(g_1' f_2'\)} (m-1-3)
            edge[tden] node[above] {\(\tilde q' p_1\)} (m-1-1)
    (m-1-4) edge[den] node[above] {\(a_2 \tilde j'\)} (m-1-3)
    (m-2-1) edge node[right] {\(g_1\)} (m-3-1)
            edge[tden] node[right] {\(p_1\)} (m-1-1)
    (m-2-2) edge node[above] {\(g_1' f_2'\)} (m-2-3)
            edge node[right] {\(g_1'\)} (m-3-2)
            edge[tden] node[above] {\(\tilde q'\)} (m-2-1)
            edge[equality] (m-1-2)
    (m-2-3) edge[equality] (m-3-3)
            edge[equality] (m-1-3)
    (m-2-4) edge[equality] (m-3-4)
            edge[den] node[above] {\(a_2 \tilde j'\)} (m-2-3)
            edge[equality] (m-1-4)
    (m-3-2) edge node[above] {\(f_2'\)} (m-3-3)
            edge[tden] node[above] {\(\tilde q\)} (m-3-1)
    (m-3-4) edge[den] node[above] {\(a_2 \tilde j'\)} (m-3-3)
    (m-4-1) edge[sden] node[right] {\(i_1\)} (m-3-1)
    (m-4-2) edge node[above] {\(f_2\)} (m-4-3)
            edge[den] node[above] {\(b_2\)} (m-4-1)
            edge[sden] node[right] {\(\tilde j\)} (m-3-2)
    (m-4-3) edge[sden] node[right] {\(\tilde j'\)} (m-3-3)
    (m-4-4) edge[den] node[above] {\(a_2\)} (m-4-3)
            edge[equality] (m-3-4);
\end{tikzpicture}\]
By~\ref{prop:3-arrow_calculus:fraction_equality}, since
\[\doublefrac{q' b_1}{f_1' g_2'}{i_2 j'} = (\doublefrac{b_1}{f_1}{a_1}) (\doublefrac{p_2}{g_2}{i_2}) = (\doublefrac{p_1}{g_1}{i_1}) (\doublefrac{b_2}{f_2}{a_2}) = \doublefrac{\tilde q' p_1}{g_1' f_2'}{a_2 \tilde j'},\]
there exist \(3\)-arrows \((v_1, h_1, u_1)\), \((v_2, h_2, u_2)\) and normal \(3\)-arrows \((r_1, d_1, k_1)\), \((r_2, d_2, k_2)\) in \(\mathcal{C}\) with denominators \(d_1\), \(d_2\), fitting into the following commutative diagram.
\[\begin{tikzpicture}[baseline=(m-4-1.base)]
  \matrix (m) [diagram without objects]{
    & & & \\
    & & & \\
    & & & \\
    & & & \\};
  \path[->, font=\scriptsize]
    (m-1-2) edge node[above] {\(f_1' g_2'\)} (m-1-3)
            edge[den] node[above] {\(q' b_1\)} (m-1-1)
    (m-1-4) edge[sden] node[above] {\(i_2 j'\)} (m-1-3)
    (m-2-1) edge[equality] (m-3-1)
            edge[equality] (m-1-1)
    (m-2-2) edge[exists] node[above] {\(h_1\)} (m-2-3)
            edge[exists, den] node[right] {\(d_1\)} (m-3-2)
            edge[exists, den] node[above] {\(v_1\)} (m-2-1)
            edge[exists, tden] node[right] {\(r_1\)} (m-1-2)
    (m-2-3) edge[exists, den] node[right] {\(d_2\)} (m-3-3)
            edge[exists, tden] node[right] {\(r_2\)} (m-1-3)
    (m-2-4) edge[equality] (m-3-4)
            edge[exists, den] node[above] {\(u_1\)} (m-2-3)
            edge[equality] (m-1-4)
    (m-3-2) edge[exists] node[above] {\(h_2\)} (m-3-3)
            edge[exists, den] node[above] {\(v_2\)} (m-3-1)
    (m-3-4) edge[exists, den] node[above] {\(u_2\)} (m-3-3)
    (m-4-1) edge[equality] (m-3-1)
    (m-4-2) edge node[above] {\(g_1' f_2'\)} (m-4-3)
            edge[tden] node[above] {\(\tilde q' p_1\)} (m-4-1)
            edge[exists, sden] node[right] {\(k_1\)} (m-3-2)
    (m-4-3) edge[exists, sden] node[right] {\(k_2\)} (m-3-3)
    (m-4-4) edge[den] node[above] {\(a_2 \tilde j'\)} (m-4-3)
            edge[equality] (m-3-4);
\end{tikzpicture}\]
Altogether, the following diagram commutes.
\[\begin{tikzpicture}[baseline=(m-8-1.base)]
  \matrix (m) [diagram without objects]{
    & & & \\
    & & & \\
    & & & \\
    & & & \\
    & & & \\
    & & & \\
    & & & \\
    & & & \\};
  \path[->, font=\scriptsize]
    (m-1-2) edge node[above] {\(f_1\)} (m-1-3)
            edge[den] node[above] {\(b_1\)} (m-1-1)
    (m-1-4) edge[den] node[above] {\(a_1\)} (m-1-3)
    (m-2-1) edge[equality] (m-3-1)
            edge[equality] (m-1-1)
    (m-2-2) edge node[above] {\(f_1'\)} (m-2-3)
            edge[equality] (m-3-2)
            edge[den] node[above] {\(q' b_1\)} (m-2-1)
            edge[tden] node[right] {\(q'\)} (m-1-2)
    (m-2-3) edge node[right] {\(g_2'\)} (m-3-3)
            edge[tden] node[right] {\(q\)} (m-1-3)
    (m-2-4) edge node[right] {\(g_2\)} (m-3-4)
            edge[sden] node[above] {\(j\)} (m-2-3)
            edge[tden] node[right] {\(p_2\)} (m-1-4)
    (m-3-2) edge node[above] {\(f_1' g_2'\)} (m-3-3)
            edge[den] node[above] {\(q' b_1\)} (m-3-1)
    (m-3-4) edge[sden] node[above] {\(j'\)} (m-3-3)
    (m-4-1) edge[equality] (m-5-1)
            edge[equality] (m-3-1)
    (m-4-2) edge node[above] {\(h_1\)} (m-4-3)
            edge[den] node[right] {\(d_1\)} (m-5-2)
            edge[den] node[above] {\(v_1\)} (m-4-1)
            edge[tden] node[right] {\(r_1\)} (m-3-2)
    (m-4-3) edge[den] node[right] {\(d_2\)} (m-5-3)
            edge[tden] node[right] {\(r_2\)} (m-3-3)
    (m-4-4) edge[equality] (m-5-4)
            edge[den] node[above] {\(u_1\)} (m-4-3)
            edge[sden] node[right] {\(i_2\)} (m-3-4)
    (m-5-2) edge node[above] {\(h_2\)} (m-5-3)
            edge[den] node[above] {\(v_2\)} (m-5-1)
    (m-5-4) edge[den] node[above] {\(u_2\)} (m-5-3)
    (m-6-1) edge node[right] {\(g_1\)} (m-7-1)
            edge[tden] node[right] {\(p_1\)} (m-5-1)
    (m-6-2) edge node[above] {\(g_1' f_2'\)} (m-6-3)
            edge node[right] {\(g_1'\)} (m-7-2)
            edge[tden] node[above] {\(\tilde q'\)} (m-6-1)
            edge[sden] node[right] {\(k_1\)} (m-5-2)
    (m-6-3) edge[equality] (m-7-3)
            edge[sden] node[right] {\(k_2\)} (m-5-3)
    (m-6-4) edge[equality] (m-7-4)
            edge[den] node[above] {\(a_2 \tilde j'\)} (m-6-3)
            edge[equality] (m-5-4)
    (m-7-2) edge node[above] {\(f_2'\)} (m-7-3)
            edge[tden] node[above] {\(\tilde q\)} (m-7-1)
    (m-7-4) edge[den] node[above] {\(a_2 \tilde j'\)} (m-7-3)
    (m-8-1) edge[sden] node[right] {\(i_1\)} (m-7-1)
    (m-8-2) edge node[above] {\(f_2\)} (m-8-3)
            edge[den] node[above] {\(b_2\)} (m-8-1)
            edge[sden] node[right] {\(\tilde j\)} (m-7-2)
    (m-8-3) edge[sden] node[right] {\(\tilde j'\)} (m-7-3)
    (m-8-4) edge[den] node[above] {\(a_2\)} (m-8-3)
            edge[equality] (m-7-4);
\end{tikzpicture}\]
Applying the flipping lemma~\ref{lem:flipping_lemma} twice and composing yields the assertion:
\[\begin{tikzpicture}[baseline=(m-8-1.base)]
  \matrix (m) [diagram without objects=1.25em]{
    & & & \\
    & & & \\
    & & & \\
    & & & \\
    & & & \\
    & & & \\
    & & & \\
    & & & \\};
  \path[->, font=\scriptsize]
    (m-1-2) edge node[above] {\(f_1\)} (m-1-3)
            edge[den] node[above] {\(b_1\)} (m-1-1)
    (m-1-4) edge[den] node[above] {\(a_1\)} (m-1-3)
    (m-2-1) edge[equality, thick] (m-3-1)
            edge[equality] (m-1-1)
    (m-2-2) edge[thick] (m-2-3)
            edge[equality, thick] (m-3-2)
            edge[den, thick] (m-2-1)
            edge[tden] (m-1-2)
    (m-2-3) edge[thick] (m-3-3)
            edge[tden] (m-1-3)
    (m-2-4) edge[thick] node[right] {\(g_2\)} (m-3-4)
            edge[sden, thick] (m-2-3)
            edge[tden] node[right] {\(p_2\)} (m-1-4)
    (m-3-2) edge[thick] (m-3-3)
            edge[den, thick] (m-3-1)
    (m-3-4) edge[sden, thick] (m-3-3)
    (m-4-1) edge[equality, thick] (m-5-1)
            edge[equality, thick] (m-3-1)
    (m-4-2) edge[thick] (m-4-3)
            edge[den, thick] (m-5-2)
            edge[den, thick] (m-4-1)
            edge[tden, thick] (m-3-2)
    (m-4-3) edge[den, thick] (m-5-3)
            edge[tden, thick] (m-3-3)
    (m-4-4) edge[equality, thick] (m-5-4)
            edge[den, thick] (m-4-3)
            edge[sden, thick] node[right] {\(i_2\)} (m-3-4)
    (m-5-2) edge[thick] (m-5-3)
            edge[den, thick] (m-5-1)
    (m-5-4) edge[den, thick] (m-5-3)
    (m-6-1) edge node[left] {\(g_1\)} (m-7-1)
            edge[tden] node[left] {\(p_1\)} (m-5-1)
    (m-6-2) edge (m-6-3)
            edge (m-7-2)
            edge[tden] (m-6-1)
            edge[sden] (m-5-2)
    (m-6-3) edge[equality] (m-7-3)
            edge[sden] (m-5-3)
    (m-6-4) edge[equality] (m-7-4)
            edge[den] (m-6-3)
            edge[equality] (m-5-4)
    (m-7-2) edge (m-7-3)
            edge[tden] (m-7-1)
    (m-7-4) edge[den] (m-7-3)
    (m-8-1) edge[sden] node[left] {\(i_1\)} (m-7-1)
    (m-8-2) edge node[below] {\(f_2\)} (m-8-3)
            edge[den] node[below] {\(b_2\)} (m-8-1)
            edge[sden] (m-7-2)
    (m-8-3) edge[sden] (m-7-3)
    (m-8-4) edge[den] node[below] {\(a_2\)} (m-8-3)
            edge[equality] (m-7-4);
\end{tikzpicture}
\begin{tikzpicture}[baseline=(m-8-1.base)]
  \matrix (m) [diagram without objects=1.25em]{
    & & & \\
    & & & \\
    & & & \\
    & & & \\
    & & & \\
    & & & \\
    & & & \\
    & & & \\};
  \path[->, font=\scriptsize]
    (m-1-2) edge[thick] node[above] {\(f_1\)} (m-1-3)
            edge[den, thick] node[above] {\(b_1\)} (m-1-1)
    (m-1-4) edge[den, thick] node[above] {\(a_1\)} (m-1-3)
    (m-2-1) edge[equality, thick] (m-1-1)
    (m-2-2) edge[thick] (m-2-3)
            edge[den, thick] (m-2-1)
            edge[tden, thick] (m-1-2)
    (m-2-3) edge[tden, thick] (m-1-3)
    (m-2-4) edge[sden, thick] (m-2-3)
            edge[tden, thick] node[right] {\(p_2\)} (m-1-4)
    (m-3-1) edge[equality] (m-4-1)
            edge[equality, thick] (m-2-1)
    (m-3-2) edge[thick] (m-3-3)
            edge (m-4-2)
            edge[den, thick] (m-3-1)
            edge[tden, thick] (m-2-2)
    (m-3-3) edge (m-4-3)
            edge[tden, thick] (m-2-3)
    (m-3-4) edge[den, thick] (m-3-3)
            edge node[right] {\(g_2\)} (m-4-4)
            edge[equality, thick] (m-2-4)
    (m-4-2) edge[thick] (m-4-3)
            edge[den, thick] (m-4-1)
    (m-4-4) edge[den, thick] (m-4-3)
    (m-5-1) edge[equality, thick] (m-4-1)
    (m-5-2) edge[thick] (m-5-3)
            edge[den, thick] (m-5-1)
            edge[sden, thick] (m-4-2)
    (m-5-3) edge[sden, thick] (m-4-3)
    (m-5-4) edge[den, thick] (m-5-3)
            edge[sden, thick] node[right] {\(i_2\)} (m-4-4)
    (m-6-1) edge node[left] {\(g_1\)} (m-7-1)
            edge[tden, thick] node[left] {\(p_1\)} (m-5-1)
    (m-6-2) edge[thick] (m-6-3)
            edge (m-7-2)
            edge[tden, thick] (m-6-1)
            edge[sden, thick] (m-5-2)
    (m-6-3) edge[equality] (m-7-3)
            edge[sden, thick] (m-5-3)
    (m-6-4) edge[equality] (m-7-4)
            edge[den, thick] (m-6-3)
            edge[equality, thick] (m-5-4)
    (m-7-2) edge (m-7-3)
            edge[tden] (m-7-1)
    (m-7-4) edge[den] (m-7-3)
    (m-8-1) edge[sden] node[left] {\(i_1\)} (m-7-1)
    (m-8-2) edge node[below] {\(f_2\)} (m-8-3)
            edge[den] node[below] {\(b_2\)} (m-8-1)
            edge[sden] (m-7-2)
    (m-8-3) edge[sden] (m-7-3)
    (m-8-4) edge[den] node[below] {\(a_2\)} (m-8-3)
            edge[equality] (m-7-4);
\end{tikzpicture}
\begin{tikzpicture}[baseline=(m-8-1.base)]
  \matrix (m) [diagram without objects=1.25em]{
    & & & \\
    & & & \\
    & & & \\
    & & & \\
    & & & \\
    & & & \\
    & & & \\
    & & & \\};
  \path[->, font=\scriptsize]
    (m-2-2) edge node[above] {\(f_1\)} (m-2-3)
            edge[den] node[above] {\(b_1\)} (m-2-1)
    (m-2-4) edge[den] node[above] {\(a_1\)} (m-2-3)
    (m-3-1) edge[equality, thick] (m-4-1)
            edge[equality] (m-2-1)
    (m-3-2) edge[thick] (m-3-3)
            edge[thick] (m-4-2)
            edge[den, thick] (m-3-1)
            edge[tden] (m-2-2)
    (m-3-3) edge[thick] (m-4-3)
            edge[tden] (m-2-3)
    (m-3-4) edge[den, thick] (m-3-3)
            edge[thick] node[right] {\(g_2\)} (m-4-4)
            edge[tden] node[right] {\(p_2\)} (m-2-4)
    (m-4-2) edge[thick] (m-4-3)
            edge[den, thick] (m-4-1)
    (m-4-4) edge[den, thick] (m-4-3)
    (m-5-1) edge[thick] node[left] {\(g_1\)} (m-6-1)
            edge[tden, thick] node[left] {\(p_1\)} (m-4-1)
    (m-5-2) edge[thick] (m-5-3)
            edge[thick] (m-6-2)
            edge[tden, thick] (m-5-1)
            edge[sden, thick] (m-4-2)
    (m-5-3) edge[equality, thick] (m-6-3)
            edge[sden, thick] (m-4-3)
    (m-5-4) edge[den, thick] (m-5-3)
            edge[equality, thick] (m-6-4)
            edge[sden, thick] node[right] {\(i_2\)} (m-4-4)
    (m-6-2) edge[thick] (m-6-3)
            edge[tden, thick] (m-6-1)
    (m-6-4) edge[den, thick] (m-6-3)
    (m-7-1) edge[sden] node[left] {\(i_1\)} (m-6-1)
    (m-7-2) edge node[below] {\(f_2\)} (m-7-3)
            edge[den] node[below] {\(b_2\)} (m-7-1)
            edge[sden] (m-6-2)
    (m-7-3) edge[sden] (m-6-3)
    (m-7-4) edge[den] node[below] {\(a_2\)} (m-7-3)
            edge[equality] (m-6-4);
\end{tikzpicture}
\begin{tikzpicture}[baseline=(m-8-1.base)]
  \matrix (m) [diagram without objects=1.25em]{
    & & & \\
    & & & \\
    & & & \\
    & & & \\
    & & & \\
    & & & \\
    & & & \\
    & & & \\};
  \path[->, font=\scriptsize]
    (m-2-2) edge[thick] node[above] {\(f_1\)} (m-2-3)
            edge[den, thick] node[above] {\(b_1\)} (m-2-1)
    (m-2-4) edge[den, thick] node[above] {\(a_1\)} (m-2-3)
    (m-3-1) edge[equality, thick] (m-2-1)
    (m-3-2) edge[thick] (m-3-3)
            edge[den, thick] (m-3-1)
            edge[tden, thick] (m-2-2)
    (m-3-3) edge[tden, thick] (m-2-3)
    (m-3-4) edge[den, thick] (m-3-3)
            edge[tden, thick] node[right] {\(p_2\)} (m-2-4)
    (m-4-1) edge node[left] {\(g_1\)} (m-5-1)
            edge[tden, thick] node[left] {\(p_1\)} (m-3-1)
    (m-4-2) edge[thick] (m-4-3)
            edge (m-5-2)
            edge[den, thick] (m-4-1)
            edge[tden, thick] (m-3-2)
    (m-4-3) edge (m-5-3)
            edge[tden, thick] (m-3-3)
    (m-4-4) edge node[right] {\(g_2\)} (m-5-4)
            edge[den, thick] (m-4-3)
            edge[equality, thick] (m-3-4)
    (m-5-2) edge[thick] (m-5-3)
            edge[den, thick] (m-5-1)
    (m-5-4) edge[den, thick] (m-5-3)
    (m-6-1) edge[equality, thick] (m-5-1)
    (m-6-2) edge[thick] (m-6-3)
            edge[tden, thick] (m-6-1)
            edge[sden, thick] (m-5-2)
    (m-6-3) edge[sden, thick] (m-5-3)
    (m-6-4) edge[den, thick] (m-6-3)
            edge[sden, thick] node[right] {\(i_2\)} (m-5-4)
    (m-7-1) edge[sden, thick] node[left] {\(i_1\)} (m-6-1)
    (m-7-2) edge node[below] {\(f_2\)} (m-7-3)
            edge[den, thick] node[below] {\(b_2\)} (m-7-1)
            edge[sden, thick] (m-6-2)
    (m-7-3) edge[sden, thick] (m-6-3)
    (m-7-4) edge[den, thick] node[below] {\(a_2\)} (m-7-3)
            edge[equality, thick] (m-6-4);
\end{tikzpicture}
\begin{tikzpicture}[baseline=(m-8-1.base)]
  \matrix (m) [diagram without objects=1.25em]{
    & & & \\
    & & & \\
    & & & \\
    & & & \\
    & & & \\
    & & & \\
    & & & \\
    & & & \\};
  \path[->, font=\scriptsize]
    (m-3-2) edge node[above] {\(f_1\)} (m-3-3)
            edge[den] node[above] {\(b_1\)} (m-3-1)
    (m-3-4) edge[den] node[above] {\(a_1\)} (m-3-3)
    (m-4-1) edge node[left] {\(g_1\)} (m-5-1)
            edge[tden] node[left] {\(p_1\)} (m-3-1)
    (m-4-2) edge (m-4-3)
            edge (m-5-2)
            edge[den] (m-4-1)
            edge[tden] (m-3-2)
    (m-4-3) edge (m-5-3)
            edge[tden] (m-3-3)
    (m-4-4) edge node[right] {\(g_2\)} (m-5-4)
            edge[den] (m-4-3)
            edge[tden] node[right] {\(p_2\)} (m-3-4)
    (m-5-2) edge (m-5-3)
            edge[den] (m-5-1)
    (m-5-4) edge[den] (m-5-3)
    (m-6-1) edge[sden] node[left] {\(i_1\)} (m-5-1)
    (m-6-2) edge node[below] {\(f_2\)} (m-6-3)
            edge[den] node[below] {\(b_2\)} (m-6-1)
            edge[sden] (m-5-2)
    (m-6-3) edge[sden] (m-5-3)
    (m-6-4) edge[den] node[below] {\(a_2\)} (m-6-3)
            edge[sden] node[right] {\(i_2\)} (m-5-4);
\end{tikzpicture} \qedhere\]
\end{enumerate}
\end{proof}

Altogether, we have proven the following main theorem of this article.

\begin{theorem} \label{th:description_of_the_fraction_category}
The fraction category \(\FractionCategory \mathcal{C}\) of a uni-fractionable category \(\mathcal{C}\) (see definition~\ref{def:uni-fractionable_categories_and_their_morphisms}\ref{def:uni-fractionable_categories_and_their_morphisms:uni-fractionable_category}) fulfills the following properties.
\begin{enumerate}
\item \label{th:description_of_the_fraction_category:category} The object set of \(\FractionCategory \mathcal{C}\) is the object set of \(\mathcal{C}\). The morphism set of \(\FractionCategory \mathcal{C}\) consists of double fractions, that is, equivalence classes of \(3\)-arrows with respect to fraction equality, where a \(3\)-arrow \((b, f, a)\) is a diagram
\[\begin{tikzpicture}[baseline=(m-1-1.base)]
  \matrix (m) [diagram without objects]{
    & & & \\};
  \path[->, font=\scriptsize]
    (m-1-2) edge node[above] {\(f\)} (m-1-3)
            edge[den] node[above] {\(b\)} (m-1-1)
    (m-1-4) edge[den] node[above] {\(a\)} (m-1-3);
\end{tikzpicture}\]
in \(\mathcal{C}\) with denominators \(a\) and \(b\). For every \(3\)-arrow \((b, f, a)\) in \(\mathcal{C}\), source and target of the double fraction \(\doublefrac{b}{f}{a}\) are given by \(\Source{\doublefrac{b}{f}{a}} = \Target b\) and \(\Target{\doublefrac{b}{f}{a}} = \Source a\). Given \(3\)-arrows \((b_1, f_1, a_1)\) and \((b_2, f_2, a_2)\) in \(\mathcal{C}\) with \(\Target{\doublefrac{b_1}{f_1}{a_1}} = \Source{\doublefrac{b_2}{f_2}{a_2}}\), the composite of the double fractions can be constructed as follows: One chooses denominators \(d\), \(d'\), \(e\), \(e'\) and morphisms \(g_1\), \(g_2\) in \(\mathcal{C}\) with \(b_2 a_1 = d e\), \(g_1 e = e' f_1\), \(d g_2 = f_2 d'\). Then \((\doublefrac{b_1}{f_1}{a_1}) (\doublefrac{b_2}{f_2}{a_2}) = \doublefrac{e' b_1}{g_1 g_2}{a_2 d'}\).
\[\begin{tikzpicture}[baseline=(m-3-1.base)]
  \matrix (m) [diagram without objects=0.9em]{
    & & & & & & & & & & \\
    & & & & & & & & & & \\
    & & & & & & & & & & \\};
  \path[->, font=\scriptsize]
    (m-1-3) edge node[above] {\(g_1\)} (m-1-6)
            edge[den] node[left] {\(e'\)} (m-2-2)
    (m-1-6) edge node[above] {\(g_2\)} (m-1-9)
            edge[den] node[left] {\(e\)} (m-2-5)
    (m-2-2) edge node[above] {\(f_1\)} (m-2-5)
            edge[den] node[left] {\(b_1\)} (m-3-1)
    (m-2-7) edge node[above] {\(f_2\)} (m-2-10)
            edge[den] node[right] {\(b_2\)} (m-3-6)
            edge[den] node[right] {\(d\)} (m-1-6)
    (m-2-10) edge[den] node[right] {\(d'\)} (m-1-9)
    (m-3-6) edge[den] node[left] {\(a_1\)} (m-2-5)
    (m-3-11) edge[den] node[right] {\(a_2\)} (m-2-10);
\end{tikzpicture}\]
The identity of an object \(X\) in \(\FractionCategory \mathcal{C}\) is given by \(1_X = \doublefrac{1_X}{1_X}{1_X}\).
\item \label{th:description_of_the_fraction_category:localisation} The fraction category \(\FractionCategory \mathcal{C}\) is a localisation of \(\mathcal{C}\), where the localisation functor \(\LocalisationFunctor\colon \mathcal{C} \map \FractionCategory \mathcal{C}\) is given on the objects by \(\LocalisationFunctor(X) = X\) for \(X \in \Ob \mathcal{C}\) and on the morphisms by \(\LocalisationFunctor(f) = \doublefrac{1}{f}{1}\) for \(f \in \Mor \mathcal{C}\). The inverse of \(\LocalisationFunctor(d)\) for \(d \in \Denominators \mathcal{C}\) is given by \((\LocalisationFunctor(d))^{- 1} = \doublefrac{d}{1}{1} = \doublefrac{1}{1}{d}\).

Given a functor \(F\colon \mathcal{C} \map \mathcal{D}\) such that \(F d\) is invertible for all \(d \in \Denominators \mathcal{C}\), the unique functor \(\hat F\colon \FractionCategory \mathcal{C} \map \mathcal{D}\) with \(F = \hat F \comp \LocalisationFunctor\) is given by \(\hat F(\doublefrac{b}{f}{a}) = (F b)^{- 1} (F f) (F a)^{- 1}\).

Given functors \(F, G\colon \mathcal{C} \map \mathcal{D}\) such that \(F d\) and \(G d\) are invertible for all \(d \in \Denominators \mathcal{C}\), and given a transformation \(\alpha\colon F \map G\), the unique transformation \(\hat \alpha\colon \hat F \map \hat G\) with \(\alpha_X = \hat \alpha_{\LocalisationFunctor(X)}\) for \(X \in \Ob \mathcal{C}\) is given by \(\hat \alpha_X = \alpha_X\) for \(X \in \Ob \FractionCategory \mathcal{C} = \Ob \mathcal{C}\).
\item \label{th:description_of_the_fraction_category:3-arrow_calculus} Given \(3\)-arrows \((b_1, f_1, a_1)\), \((b_2, f_2, a_2)\) and normal \(3\)-arrows \((p_1, g_1, i_1)\), \((p_2, g_2, i_2)\) in \(\mathcal{C}\), we have
\[(\doublefrac{b_1}{f_1}{a_1}) (\doublefrac{p_2}{g_2}{i_2}) = (\doublefrac{p_1}{g_1}{i_1}) (\doublefrac{b_2}{f_2}{a_2})\]
if and only if there exist \(3\)-arrows \((\tilde b_1, \tilde f_1, \tilde a_1)\), \((\tilde b_2, \tilde f_2, \tilde a_2)\) and normal \(3\)-arrows \((\tilde p_1, \tilde g_1, \tilde i_1)\), \((\tilde p_2, \tilde g_2, \tilde i_2)\) in \(\mathcal{C}\), fitting into the following commutative diagram in \(\mathcal{C}\).
\[\begin{tikzpicture}[baseline=(m-4-1.base)]
  \matrix (m) [diagram without objects]{
    & & & \\
    & & & \\
    & & & \\
    & & & \\};
  \path[->, font=\scriptsize]
    (m-1-2) edge node[above] {\(f_1\)} (m-1-3)
            edge[den] node[above] {\(b_1\)} (m-1-1)
    (m-1-4) edge[den] node[above] {\(a_1\)} (m-1-3)
    (m-2-1) edge node[right] {\(g_1\)} (m-3-1)
            edge[tden] node[right] {\(p_1\)} (m-1-1)
    (m-2-2) edge[exists] node[above] {\(\tilde f_1\)} (m-2-3)
            edge[exists] node[right] {\(\tilde g_1\)} (m-3-2)
            edge[exists, den] node[above] {\(\tilde b_1\)} (m-2-1)
            edge[exists, tden] node[right] {\(\tilde p_1\)} (m-1-2)
    (m-2-3) edge[exists] node[right] {\(\tilde g_2\)} (m-3-3)
            edge[exists, tden] node[right] {\(\tilde p_2\)} (m-1-3)
    (m-2-4) edge node[right] {\(g_2\)} (m-3-4)
            edge[exists, den] node[above] {\(\tilde a_1\)} (m-2-3)
            edge[tden] node[right] {\(p_2\)} (m-1-4)
    (m-3-2) edge[exists] node[above] {\(\tilde f_2\)} (m-3-3)
            edge[exists, den] node[above] {\(\tilde b_2\)} (m-3-1)
    (m-3-4) edge[exists, den] node[above] {\(\tilde a_2\)} (m-3-3)
    (m-4-1) edge[sden] node[right] {\(i_1\)} (m-3-1)
    (m-4-2) edge node[above] {\(f_2\)} (m-4-3)
            edge[den] node[above] {\(b_2\)} (m-4-1)
            edge[exists, sden] node[right] {\(\tilde i_1\)} (m-3-2)
    (m-4-3) edge[exists, sden] node[right] {\(\tilde i_2\)} (m-3-3)
    (m-4-4) edge[den] node[above] {\(a_2\)} (m-4-3)
            edge[sden] node[right] {\(i_2\)} (m-3-4);
\end{tikzpicture}\]
\end{enumerate}
\end{theorem}
\begin{proof}
This follows from propositions~\ref{prop:welldefinedness_of_the_fraction_category} and~\ref{prop:flexibility_of_composites_and_inverses_in_the_fraction_category}\ref{prop:flexibility_of_composites_and_inverses_in_the_fraction_category:composition}, proposition~\ref{prop:universal_property_of_the_fraction_category} and proposition~\ref{prop:3-arrow_calculus}\ref{prop:3-arrow_calculus:commutative_diagrams}.
\end{proof}

As a consequence of \(3\)-arrow calculus, we get the following criterion. For a related 2-arrow version of this result, cf.~\cite[ch.~1, \S 2, th.~4-2]{verdier:1963:categories_derivees} and~\cite[III.2.10]{gelfand_manin:2003:methods_of_homological_algebra}.

\begin{proposition} \label{prop:full_uni-fractionable_subcategory_induces_equivalence_of_the_fraction_categories}
We suppose given a uni-fractionable category \(\mathcal{C}\) and a category with denominators \(\mathcal{U}\) such that \(\mathcal{U}\) is a full subcategory of \(\mathcal{C}\) and \(\Denominators \mathcal{U} = (\Denominators \mathcal{C}) \intersection (\Mor \mathcal{U})\). Moreover, we suppose that \(\mathcal{U}\) fulfills one of the following two dual conditions.
\begin{enumerate}
\item \label{prop:full_uni-fractionable_subcategory_induces_equivalence_of_the_fraction_categories:s-resolution} For every object \(X\) in \(\mathcal{C}\), there exist an object \(\tilde X\) in \(\mathcal{U}\) and a denominator \(d\colon \tilde X \map X\) in \(\mathcal{C}\). Moreover, for every S-denominator \(i\colon U \map \tilde U\) with \(U\) in \(\mathcal{U}\), it follows that \(\tilde U\) is in \(\mathcal{U}\).
\item \label{prop:full_uni-fractionable_subcategory_induces_equivalence_of_the_fraction_categories:t-resolution} For every object \(X\) in \(\mathcal{C}\), there exist an object \(\tilde X\) in \(\mathcal{U}\) and a denominator \(d\colon X \map \tilde X\) in \(\mathcal{C}\). Moreover, for every T-denominator \(p\colon \tilde U \map U\) with \(U\) in \(\mathcal{U}\), it follows that \(\tilde U\) is in \(\mathcal{U}\).
\end{enumerate}
Then the inclusion functor \(\inc\colon \mathcal{U} \map \mathcal{C}\) induces an equivalence \({\FractionCategory \inc}\colon \FractionCategory \mathcal{U} \map \FractionCategory \mathcal{C}\).
\end{proposition}
\begin{proof}
We suppose that \(\mathcal{U}\) fulfills~\ref{prop:full_uni-fractionable_subcategory_induces_equivalence_of_the_fraction_categories:s-resolution}, the other case follows by duality. To show that \(\FractionCategory \inc\) is an equivalence of categories, we will verify that \(\FractionCategory \inc\) is full, faithful and dense. Since for every \(X \in \Ob \mathcal{C}\) there exist \(\tilde X \in \Ob \mathcal{U}\) and a denominator \(d\colon \tilde X \map X\) in \(\mathcal{C}\), we have \(X \isomorphic \tilde X = (\FractionCategory \inc) \tilde X\) in \(\FractionCategory \mathcal{C}\). Hence \(\FractionCategory \inc\) is dense. To prove that \(\FractionCategory \inc\) is full and faithful, we have to show that the map
\[{_{\FractionCategory \mathcal{U}}}(U, V) \map {_{\FractionCategory \mathcal{C}}}(U, V), \varphi \mapsto (\FractionCategory \inc) \varphi\]
is bijective for \(U, V \in \Ob \mathcal{U}\).

To show surjectivity, we suppose given a morphism \(\psi \in {_{\FractionCategory \mathcal{C}}}(U, V)\) and a normal \(3\)-arrow \((p, f, i)\colon \threearrow{U}{X}{Y}{V}\) in \(\mathcal{C}\) with \(\psi = \doublefrac{p}{f}{i}\). Since \(i\) is an S-denominator and \(V\) is an object in \(\mathcal{U}\), it follows that \(Y\) is an object in \(\mathcal{U}\). Moreover, there exists an object \(\tilde X\) in \(\mathcal{U}\) and a denominator \(d\colon \tilde X \map X\).
\[\begin{tikzpicture}[baseline=(m-2-1.base)]
  \matrix (m) [diagram]{
    U & \tilde X & Y & V \\
    U & X & Y & V \\};
  \path[->, font=\scriptsize]
    (m-1-1) edge[equality] (m-2-1)
    (m-1-2) edge node[above] {\(d f\)} (m-1-3)
            edge[den] node[right] {\(d\)} (m-2-2)
            edge[den] node[above] {\(d p\)} (m-1-1)
    (m-1-3) edge[equality] (m-2-3)
    (m-1-4) edge[equality] (m-2-4)
            edge[den] node[above] {\(i\)} (m-1-3)
    (m-2-2) edge node[above] {\(f\)} (m-2-3)
            edge[tden] node[above] {\(p\)} (m-2-1)
    (m-2-4) edge[sden] node[above] {\(i\)} (m-2-3);
\end{tikzpicture}\]
It follows that \((p, f, i) \fractionequal (d p, d f, i)\), and as \((d p, d f, i)\) is a \(3\)-arrow in \(\mathcal{U}\), we have
\[\psi = \doublefrac{p}{f}{i} = \doublefrac{d p}{d f}{i} = (\FractionCategory \inc)(\doublefrac{d p}{d f}{i}).\]
Thus \({_{\FractionCategory \mathcal{U}}}(U, V) \map {_{\FractionCategory \mathcal{C}}}(U, V), \varphi \mapsto (\FractionCategory \inc) \varphi\) is surjective.

To show injectivity, we suppose given \(\varphi_1, \varphi_2 \in {_{\FractionCategory \mathcal{U}}}(U, V)\) with \((\FractionCategory \inc) \varphi_1 = (\FractionCategory \inc) \varphi_2\). We choose normal \(3\)-arrows \((p_1, f_1, i_1)\colon \threearrow{U}{U_1}{V_1}{V}\) and \((p_2, f_2, i_2)\colon \threearrow{U}{U_2}{V_2}{V}\) in \(\mathcal{U}\) with \(\varphi_1 = \doublefrac{p_1}{f_1}{i_1}\) and \(\varphi_2 = \doublefrac{p_2}{f_2}{i_2}\). By proposition~\ref{prop:3-arrow_calculus}\ref{prop:3-arrow_calculus:fraction_equality}, there exist normal \(3\)-arrows \((\tilde p_1, \tilde f_1, \tilde i_1)\colon \threearrow{U}{X_1}{Y_1}{V}\), \((\tilde p_2, \tilde f_2, \tilde i_2)\colon \threearrow{U}{X_2}{Y_2}{V}\), \((q_1, d_1, j_1)\colon \threearrow{U_1}{X_1}{X_2}{U_2}\), \((q_2, d_2, j_2)\colon \threearrow{V_1}{Y_1}{Y_2}{V_2}\) in \(\mathcal{C}\) with denominators \(d_1\), \(d_2\), fitting into a commutative diagram as follows.
\[\begin{tikzpicture}[baseline=(m-4-1.base)]
  \matrix (m) [diagram]{
    U & U_1 & V_1 & V \\
    U & X_1 & Y_1 & V \\
    U & X_2 & Y_2 & V \\
    U & U_2 & V_2 & V \\};
  \path[->, font=\scriptsize]
    (m-1-2) edge node[above] {\(f_1\)} (m-1-3)
            edge[tden] node[above] {\(p_1\)} (m-1-1)
    (m-1-4) edge[sden] node[above] {\(i_1\)} (m-1-3)
    (m-2-1) edge[equality] (m-3-1)
            edge[equality] (m-1-1)
    (m-2-2) edge[exists] node[above] {\(\tilde f_1\)} (m-2-3)
            edge[exists, den] node[right] {\(d_1\)} (m-3-2)
            edge[exists, tden] node[above] {\(\tilde p_1\)} (m-2-1)
            edge[exists, tden] node[right] {\(q_1\)} (m-1-2)
    (m-2-3) edge[exists, den] node[right] {\(d_2\)} (m-3-3)
            edge[exists, tden] node[right] {\(q_2\)} (m-1-3)
    (m-2-4) edge[equality] (m-3-4)
            edge[exists, sden] node[above] {\(\tilde i_1\)} (m-2-3)
            edge[equality] (m-1-4)
    (m-3-2) edge[exists] node[above] {\(\tilde f_2\)} (m-3-3)
            edge[exists, tden] node[above] {\(\tilde p_2\)} (m-3-1)
    (m-3-4) edge[exists, sden] node[above] {\(\tilde i_2\)} (m-3-3)
    (m-4-1) edge[equality] (m-3-1)
    (m-4-2) edge node[above] {\(f_2\)} (m-4-3)
            edge[tden] node[above] {\(p_2\)} (m-4-1)
            edge[exists, sden] node[right] {\(j_1\)} (m-3-2)
    (m-4-3) edge[exists, sden] node[right] {\(j_2\)} (m-3-3)
    (m-4-4) edge[sden] node[above] {\(i_2\)} (m-4-3)
            edge[equality] (m-3-4);
\end{tikzpicture}\]
Since \(\tilde i_1\) resp.\ \(j_1\) resp.\ \(j_2\) is an S-denominator and \(V\) resp.\ \(U_2\) resp.\ \(V_2\) is an object in \(\mathcal{U}\), it follows that \(Y_1\) resp.\ \(X_2\) resp.\ \(Y_2\) is an object in \(\mathcal{U}\). Moreover, there exists an object \(\tilde X_1\) in \(\mathcal{U}\) and a denominator \(d\colon \tilde X_1 \map X_1\) in \(\mathcal{C}\). Thus we obtain the following commutative diagram in which all objects -- and hence all morphisms -- are in \(\mathcal{U}\), and where \(d \tilde p_1\) is a denominator by multiplicativity.
\[\begin{tikzpicture}[baseline=(m-4-1.base)]
  \matrix (m) [diagram]{
    U & U_1 & V_1 & V \\
    U & \tilde X_1 & Y_1 & V \\
    U & X_2 & Y_2 & V \\
    U & U_2 & V_2 & V \\};
  \path[->, font=\scriptsize]
    (m-1-2) edge node[above] {\(f_1\)} (m-1-3)
            edge[tden] node[above] {\(p_1\)} (m-1-1)
    (m-1-4) edge[sden] node[above] {\(i_1\)} (m-1-3)
    (m-2-1) edge[equality] (m-3-1)
            edge[equality] (m-1-1)
    (m-2-2) edge node[above] {\(d \tilde f_1\)} (m-2-3)
            edge[den] node[right] {\(d d_1\)} (m-3-2)
            edge[den] node[above] {\(d \tilde p_1\)} (m-2-1)
            edge[den] node[right] {\(d q_1\)} (m-1-2)
    (m-2-3) edge[den] node[right] {\(d_2\)} (m-3-3)
            edge[tden] node[right] {\(q_2\)} (m-1-3)
    (m-2-4) edge[equality] (m-3-4)
            edge[sden] node[above] {\(\tilde i_1\)} (m-2-3)
            edge[equality] (m-1-4)
    (m-3-2) edge node[above] {\(\tilde f_2\)} (m-3-3)
            edge[tden] node[above] {\(\tilde p_2\)} (m-3-1)
    (m-3-4) edge[sden] node[above] {\(\tilde i_2\)} (m-3-3)
    (m-4-1) edge[equality] (m-3-1)
    (m-4-2) edge node[above] {\(f_2\)} (m-4-3)
            edge[tden] node[above] {\(p_2\)} (m-4-1)
            edge[sden] node[right] {\(j_1\)} (m-3-2)
    (m-4-3) edge[sden] node[right] {\(j_2\)} (m-3-3)
    (m-4-4) edge[sden] node[above] {\(i_2\)} (m-4-3)
            edge[equality] (m-3-4);
\end{tikzpicture}\]
But this implies
\[\varphi_1 = \doublefrac{p_1}{f_1}{i_1} = \doublefrac{p_2}{f_2}{i_2} = \varphi_2\]
in \(\FractionCategory \mathcal{U}\). Therefore the map \({_{\FractionCategory \mathcal{U}}}(U, V) \map {_{\FractionCategory \mathcal{C}}}(U, V), \varphi \mapsto (\FractionCategory \inc) \varphi\) is injective.
\end{proof}

\section{(Co)products and additive uni-fractionable categories} \label{sec:co_products_and_additive_uni-fractionable_categories}

Some of our examples of uni-fractionable categories in section~\ref{sec:applications} have finite coproducts or products or are even additive categories, so it is a natural question to ask whether these features are preserved when passing to the fraction category.

\begin{proposition} \label{prop:uni-fractionable_categories_admitting_finite_co_products}
We suppose given a uni-fractionable category \(\mathcal{C}\).
\begin{enumerate}
\item \label{prop:uni-fractionable_categories_admitting_finite_co_products:coproducts} We suppose that \(\mathcal{C}\) admits finite coproducts.
\begin{enumerate}
\item \label{prop:uni-fractionable_categories_admitting_finite_co_products:coproducts:denominators_closedness_implies_preservation_by_localisation} If \(\Denominators \mathcal{C}\) is closed under finite coproducts, then the fraction category \(\FractionCategory \mathcal{C}\) admits finite coproducts and the localisation functor \(\LocalisationFunctor\colon \mathcal{C} \map \FractionCategory \mathcal{C}\) preserves finite coproducts. In this case, we have \(\ini[\LocalisationFunctor(\initialobject)]_{\LocalisationFunctor(X)} = \LocalisationFunctor(\ini[\initialobject]_X)\colon \LocalisationFunctor(\initialobject) \map \LocalisationFunctor(X)\) for \(X \in \Ob \mathcal{C}\), and we have
\[\begin{smallpmatrix} \doublefrac{b_1}{f_1}{a} \\ \doublefrac{b_2}{f_2}{a} \end{smallpmatrix}^{\LocalisationFunctor(X_1 \smallcoprod X_2)} = \doublefrac{(b_1 \smallcoprod b_2)}{\begin{smallpmatrix} f_1 \\ f_2 \end{smallpmatrix}^{\tilde X_1 \smallcoprod \tilde X_2}}{a}\colon \LocalisationFunctor(X_1 \smallcoprod X_2) \map \LocalisationFunctor(Y)\]
for \(3\)-arrows \((b_1, f_1, a)\colon \threearrow{X_1}{\tilde X_1}{\tilde Y}{Y}\) and \((b_2, f_2, a)\colon \threearrow{X_2}{\tilde X_2}{\tilde Y}{Y}\) in \(\mathcal{C}\).
\item \label{prop:uni-fractionable_categories_admitting_finite_co_products:coproducts:in_saturated_case_the_implies_preservation_by_localisation_implies_denominator_closedness} If \(\mathcal{C}\) is saturated and the localisation functor \(\LocalisationFunctor\colon \mathcal{C} \map \FractionCategory \mathcal{C}\) preserves finite coproducts, then \(\Denominators \mathcal{C}\) is closed under finite coproducts.
\end{enumerate}
\item \label{prop:uni-fractionable_categories_admitting_finite_co_products:products} We suppose that \(\mathcal{C}\) admits finite products.
\begin{enumerate}
\item \label{prop:uni-fractionable_categories_admitting_finite_co_products:products:denominators_closedness_implies_preservation_by_localisation} If \(\Denominators \mathcal{C}\) is closed under finite products, then the fraction category \(\FractionCategory \mathcal{C}\) admits finite products and the localisation functor \(\LocalisationFunctor\colon \mathcal{C} \map \FractionCategory \mathcal{C}\) preserves finite products. In this case, we have \(\ter[\LocalisationFunctor(\terminalobject)]_{\LocalisationFunctor(X)} = \LocalisationFunctor(\ter[\terminalobject]_X)\colon \LocalisationFunctor(X) \map \LocalisationFunctor(\terminalobject)\) for \(X \in \Ob \mathcal{C}\), and we have
\[\begin{smallpmatrix} \doublefrac{b}{f_1}{a_1} & \doublefrac{b}{f_2}{a_2} \end{smallpmatrix}^{\LocalisationFunctor(Y_1 \smallprod Y_2)} = \doublefrac{b}{\begin{smallpmatrix} f_1 & f_2 \end{smallpmatrix}^{\tilde Y_1 \smallprod \tilde Y_2}}{(a_1 \smallprod a_2)}\colon \LocalisationFunctor(X) \map \LocalisationFunctor(Y_1 \smallprod Y_2)\]
for \(3\)-arrows \((b, f_1, a_1)\colon \threearrow{X}{\tilde X}{\tilde Y_1}{Y_1}\) and \((b, f_2, a_2)\colon \threearrow{X}{\tilde X}{\tilde Y_2}{Y_2}\) in \(\mathcal{C}\).
\item \label{prop:uni-fractionable_categories_admitting_finite_co_products:products:in_saturated_case_the_implies_preservation_by_localisation_implies_denominator_closedness} If \(\mathcal{C}\) is saturated and the localisation functor \(\LocalisationFunctor\colon \mathcal{C} \map \FractionCategory \mathcal{C}\) preserves finite products, then \(\Denominators \mathcal{C}\) is closed under finite products.
\end{enumerate}
\item \label{prop:uni-fractionable_categories_admitting_finite_co_products:sums} We suppose that \(\mathcal{C}\) admits finite sums.
\begin{enumerate}
\item \label{prop:uni-fractionable_categories_admitting_finite_co_products:sums:denominators_closedness_implies_preservation_by_localisation} If \(\Denominators \mathcal{C}\) is closed under finite sums, then the fraction category \(\FractionCategory \mathcal{C}\) admits finite sums and the localisation functor \(\LocalisationFunctor\colon \mathcal{C} \map \FractionCategory \mathcal{C}\) preserves finite sums. In this case, we have \(0 = \LocalisationFunctor(0)\colon \LocalisationFunctor(X) \map \LocalisationFunctor(Y)\) for \(X, Y \in \Ob \mathcal{C}\). Moreover, we have
\[\begin{smallpmatrix} \doublefrac{b_1}{f_1}{a} \\ \doublefrac{b_2}{f_2}{a} \end{smallpmatrix}^{\LocalisationFunctor(X_1 \directsum X_2)} = \doublefrac{(b_1 \directsum b_2)}{\begin{smallpmatrix} f_1 \\ f_2 \end{smallpmatrix}^{\tilde X_1 \directsum \tilde X_2}}{a}\colon \LocalisationFunctor(X_1 \directsum X_2) \map \LocalisationFunctor(Y)\]
for \(3\)-arrows \((b_1, f_1, a)\colon \threearrow{X_1}{\tilde X_1}{\tilde Y}{Y}\) and \((b_2, f_2, a)\colon \threearrow{X_2}{\tilde X_2}{\tilde Y}{Y}\) in \(\mathcal{C}\), and we have
\[\begin{smallpmatrix} \doublefrac{b}{f_1}{a_1} & \doublefrac{b}{f_2}{a_2} \end{smallpmatrix}^{\LocalisationFunctor(Y_1 \directsum Y_2)} = \doublefrac{b}{\begin{smallpmatrix} f_1 & f_2 \end{smallpmatrix}^{\tilde Y_1 \directsum \tilde Y_2}}{(a_1 \directsum a_2)}\colon \LocalisationFunctor(X) \map \LocalisationFunctor(Y_1 \directsum Y_2)\]
for \(3\)-arrows \((b, f_1, a_1)\colon \threearrow{X}{\tilde X}{\tilde Y_1}{Y_1}\) and \((b, f_2, a_2)\colon \threearrow{X}{\tilde X}{\tilde Y_2}{Y_2}\) in \(\mathcal{C}\).
\item \label{prop:uni-fractionable_categories_admitting_finite_co_products:sums:in_saturated_case_the_implies_preservation_by_localisation_implies_denominator_closedness} If \(\mathcal{C}\) is saturated and the localisation functor \(\LocalisationFunctor\colon \mathcal{C} \map \FractionCategory \mathcal{C}\) preserves finite sums, then \(\Denominators \mathcal{C}\) is closed under finite sums.
\end{enumerate}
\end{enumerate}
\end{proposition}
\begin{proof} \
\begin{enumerate}
\item
\begin{enumerate}
\item We suppose that \(\Denominators \mathcal{C}\) is closed under finite coproducts. Moreover, we suppose given \(X \in \Ob \mathcal{C}\). Then \(\LocalisationFunctor(\ini[\initialobject]_X)\) is a morphism from \(\LocalisationFunctor(\initialobject)\) to \(\LocalisationFunctor(X)\). So let us suppose given an arbitrary morphism \(\varphi\colon \LocalisationFunctor(\initialobject) \map X\) in \(\FractionCategory \mathcal{C}\), and we let \((b, f, a)\colon \threearrow{\initialobject}{I}{\tilde X}{X}\) be a \(3\)-arrow in \(\mathcal{C}\) with \(\varphi = \doublefrac{b}{f}{a}\). By the universal property of \(\initialobject\), we have \(\ini[\initialobject]_{I} b = 1_{\initialobject}\) and \(\ini[\initialobject]_{I} f = \ini[\initialobject]_{\tilde X} = \ini[\initialobject]_{X} a\), and therefore
\[\varphi = \doublefrac{b}{f}{a} = \doublefrac{\ini[\initialobject]_{I} b}{\ini[\initialobject]_{I} f}{a} = \doublefrac{1_{\initialobject}}{\ini[\initialobject]_X}{1} = \LocalisationFunctor(\ini[\initialobject]_X).\]
Hence \(\LocalisationFunctor(\initialobject)\) is an initial object in \(\FractionCategory \mathcal{C}\) with \(\ini[\LocalisationFunctor(\initialobject)]_{\LocalisationFunctor(X)} = \LocalisationFunctor(\ini[\initialobject]_X)\) for all \(X \in \Ob \mathcal{C}\).
\[\begin{tikzpicture}[baseline=(m-2-1.base)]
  \matrix (m) [diagram]{
    \initialobject & \initialobject & X & X \\
    \initialobject & I & \tilde X & X \\};
  \path[->, font=\scriptsize]
    (m-1-1) edge[equality] (m-2-1)
    (m-1-2) edge node[above] {\(\ini[\initialobject]_X\)} (m-1-3)
            edge[den] node[right] {\(\ini[\initialobject]_{I}\)} (m-2-2)
            edge[equality] (m-1-1)
    (m-1-3) edge[den] node[right] {\(a\)} (m-2-3)
    (m-1-4) edge[equality] (m-2-4)
            edge[equality] (m-1-3)
    (m-2-2) edge node[above] {\(f\)} (m-2-3)
            edge[den] node[above] {\(b\)} (m-2-1)
    (m-2-4) edge[den] node[above] {\(a\)} (m-2-3);
\end{tikzpicture}\]

Next, we suppose given morphisms \(\varphi_1\colon X_1 \map Y\) and \(\varphi_2\colon X_2 \map Y\) in \(\FractionCategory \mathcal{C}\). By proposition~\ref{prop:morphisms_in_the_fraction_category_can_have_the_same_denominators}, there exist \(3\)-arrows \((b_1, f_1, a)\colon \threearrow{X_1}{\tilde X_1}{\tilde Y}{Y}\) and \((b_2, f_2, a)\colon \threearrow{X_2}{\tilde X_2}{\tilde Y}{Y}\) in \(\mathcal{C}\) with \(\varphi_1 = \doublefrac{b_1}{f_1}{a}\) and \(\varphi_2 = \doublefrac{b_2}{f_2}{a}\). As \(b_1 \smallcoprod b_2\) is a denominator in \(\mathcal{C}\) by assumption, we have the \(3\)-arrow \((b_1 \smallcoprod b_2, \begin{smallpmatrix} f_1 \\ f_2 \end{smallpmatrix}, a)\) in \(\mathcal{C}\). Moreover, since \(\emb_k^{\tilde X_1 \smallcoprod \tilde X_2} (b_1 \smallcoprod b_2) = b_k \emb_k^{X_1 \smallcoprod X_2}\), we have
\[\LocalisationFunctor(\emb_k^{X_1 \smallcoprod X_2}) (\doublefrac{(b_1 \smallcoprod b_2)}{\begin{smallpmatrix} f_1 \\ f_2 \end{smallpmatrix}^{\tilde X_1 \smallcoprod \tilde X_2}}{a}) = \doublefrac{b_k}{\emb_k^{\tilde X_1 \smallcoprod \tilde X_2} \begin{smallpmatrix} f_1 \\ f_2 \end{smallpmatrix}^{\tilde X_1 \smallcoprod \tilde X_2}}{a} = \doublefrac{b_k}{f_k}{a} = \varphi_k\]
for \(k \in \{1, 2\}\).
\[\begin{tikzpicture}[baseline=(m-3-1.base)]
  \matrix (m) [diagram mixed=1.25em]{
    & & \tilde X_k & & & \tilde X_1 \smallcoprod \tilde X_2 & & & \tilde Y & & \\
    & X_k & & & X_1 \smallcoprod X_2 & & \tilde X_1 \smallcoprod \tilde X_2 & & & \tilde Y & \\
    X_k & & & & & X_1 \smallcoprod X_2 & & & & & Y \\};
  \path[->, font=\scriptsize]
    (m-1-3) edge node[above] {\(\emb_k^{\tilde X_1 \smallcoprod \tilde X_2}\)} (m-1-6)
            edge[den] node[left] {\(b_k\)} (m-2-2)
    (m-1-6) edge node[above] {\(\begin{smallpmatrix} f_1 \\ f_2 \end{smallpmatrix}^{\tilde X_1 \smallcoprod \tilde X_2}\)} (m-1-9)
            edge[den] node[left=3pt] {\(b_1 \smallcoprod b_2\)} (m-2-5)
    (m-2-2) edge node[above] {\(\emb_k^{X_1 \smallcoprod X_2}\)} (m-2-5)
            edge[equality] (m-3-1)
    (m-2-7) edge node[above, pos=0.25] {\(\begin{smallpmatrix} f_1 \\ f_2 \end{smallpmatrix}^{\tilde X_1 \smallcoprod \tilde X_2}\)} (m-2-10)
            edge[den] node[right=4pt] {\(b_1 \smallcoprod b_2\)} (m-3-6)
            edge[equality] (m-1-6)
    (m-2-10) edge[equality] (m-1-9)
    (m-3-6) edge[equality] (m-2-5)
    (m-3-11) edge[den] node[right] {\(a\)} (m-2-10);
\end{tikzpicture}\]
Conversely, we suppose given morphisms \(\varphi, \varphi'\colon \LocalisationFunctor(X_1 \smallcoprod X_2) \map \LocalisationFunctor(Y)\) in \(\FractionCategory \mathcal{C}\) such that \linebreak 
\(\LocalisationFunctor(\emb_k^{X_1 \smallcoprod X_2}) \varphi = \LocalisationFunctor(\emb_k^{X_1 \smallcoprod X_2}) \varphi' = \varphi_k\) for \(k \in \{1, 2\}\). By proposition~\ref{prop:morphisms_in_the_fraction_category_can_have_the_same_denominators}, there exist normal \(3\)-arrows \((p, f, i), (p, f', i)\colon \threearrow{X_1 \smallcoprod X_2}{\tilde X}{\tilde Y}{Y}\) in \(\mathcal{C}\) with \(\varphi = \doublefrac{p}{f}{i}\) and \(\varphi' = \doublefrac{p}{f'}{i}\). For \(k \in \{1, 2\}\), we choose a T-denominator \(p_k\colon \tilde X_k \map X_k\) and a morphism \(e_k\colon \tilde X_k \map \tilde X\) in \(\mathcal{C}\) with \(p_k \emb_k^{X_1 \smallcoprod X_2} = e_k p\). Then we have
\[\varphi_k = \LocalisationFunctor(\emb_k^{X_1 \smallcoprod X_2}) \varphi = \LocalisationFunctor(\emb_k^{X_1 \smallcoprod X_2}) (\doublefrac{p}{f}{i}) = \doublefrac{p_k}{e_k f}{i}\]
for \(k \in \{1, 2\}\).
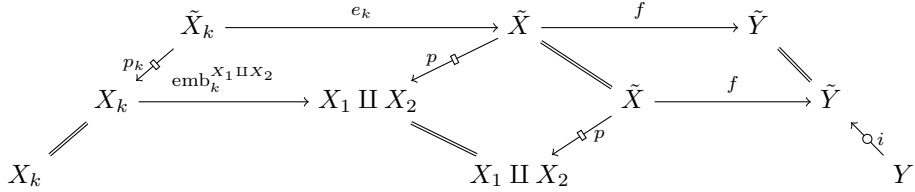
\[\begin{tikzpicture}[baseline=(m-3-1.base)]
  \matrix (m) [diagram mixed=1.25em]{
    & & \tilde X_k & & & \tilde X & & & \tilde Y & & \\
    & X_k & & & X_1 \smallcoprod X_2 & & \tilde X & & & \tilde Y & \\
    X_k & & & & & X_1 \smallcoprod X_2 & & & & & Y \\};
  \path[->, font=\scriptsize]
    (m-1-3) edge node[above] {\(e_k\)} (m-1-6)
            edge[tden] node[left] {\(p_k\)} (m-2-2)
    (m-1-6) edge node[above] {\(f\)} (m-1-9)
            edge[tden] node[left=3pt] {\(p\)} (m-2-5)
    (m-2-2) edge node[above] {\(\emb_k^{X_1 \smallcoprod X_2}\)} (m-2-5)
            edge[equality] (m-3-1)
    (m-2-7) edge node[above] {\(f\)} (m-2-10)
            edge[tden] node[right=1pt] {\(p\)} (m-3-6)
            edge[equality] (m-1-6)
    (m-2-10) edge[equality] (m-1-9)
    (m-3-6) edge[equality] (m-2-5)
    (m-3-11) edge[sden] node[right] {\(i\)} (m-2-10);
\end{tikzpicture}\]
Analogously, we also have \(\varphi_k = \doublefrac{p_k}{e_k f'}{i}\) and therefore \(\LocalisationFunctor(e_k f) = \LocalisationFunctor(e_k f')\) for \(k \in \{1, 2\}\). By proposition~\ref{prop:3-arrow_calculus}\ref{prop:3-arrow_calculus:fraction_equality}, there exist normal \(3\)-arrows \((\tilde p_k, \tilde f_k, \tilde i_k)\), \((\tilde p_k', \tilde f_k', \tilde i_k')\), \((q_k, d_k, j_k)\), \((\tilde q_k, \tilde d_k, \tilde j_k)\) in \(\mathcal{C}\) with denominators \(d_k, \tilde d_k\) for \(k \in \{1, 2\}\), fitting into the following commutative diagrams in \(\mathcal{C}\).
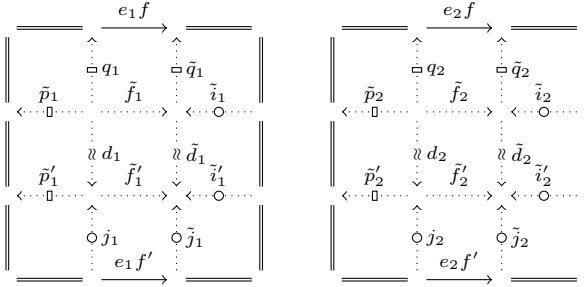
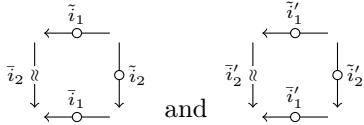
\[\begin{tikzpicture}[baseline=(m-4-1.base)]
  \matrix (m) [diagram without objects]{
    & & & \\
    & & & \\
    & & & \\
    & & & \\};
  \path[->, font=\scriptsize]
    (m-1-2) edge node[above] {\(e_1 f\)} (m-1-3)
            edge[equality] (m-1-1)
    (m-1-4) edge[equality] (m-1-3)
    (m-2-1) edge[equality] (m-3-1)
            edge[equality] (m-1-1)
    (m-2-2) edge[exists] node[above] {\(\tilde f_1\)} (m-2-3)
            edge[exists, den] node[right] {\(d_1\)} (m-3-2)
            edge[exists, tden] node[above] {\(\tilde p_1\)} (m-2-1)
            edge[exists, tden] node[right] {\(q_1\)} (m-1-2)
    (m-2-3) edge[exists, den] node[right] {\(\tilde d_1\)} (m-3-3)
            edge[exists, tden] node[right] {\(\tilde q_1\)} (m-1-3)
    (m-2-4) edge[equality] (m-3-4)
            edge[exists, sden] node[above] {\(\tilde i_1\)} (m-2-3)
            edge[equality] (m-1-4)
    (m-3-2) edge[exists] node[above] {\(\tilde f_1'\)} (m-3-3)
            edge[exists, tden] node[above] {\(\tilde p_1'\)} (m-3-1)
    (m-3-4) edge[exists, sden] node[above] {\(\tilde i_1'\)} (m-3-3)
    (m-4-1) edge[equality] (m-3-1)
    (m-4-2) edge node[above] {\(e_1 f'\)} (m-4-3)
            edge[equality] (m-4-1)
            edge[exists, sden] node[right] {\(j_1\)} (m-3-2)
    (m-4-3) edge[exists, sden] node[right] {\(\tilde j_1\)} (m-3-3)
    (m-4-4) edge[equality] (m-4-3)
            edge[equality] (m-3-4);
\end{tikzpicture}
\qquad
\begin{tikzpicture}[baseline=(m-4-1.base)]
  \matrix (m) [diagram without objects]{
    & & & \\
    & & & \\
    & & & \\
    & & & \\};
  \path[->, font=\scriptsize]
    (m-1-2) edge node[above] {\(e_2 f\)} (m-1-3)
            edge[equality] (m-1-1)
    (m-1-4) edge[equality] (m-1-3)
    (m-2-1) edge[equality] (m-3-1)
            edge[equality] (m-1-1)
    (m-2-2) edge[exists] node[above] {\(\tilde f_2\)} (m-2-3)
            edge[exists, den] node[right] {\(d_2\)} (m-3-2)
            edge[exists, tden] node[above] {\(\tilde p_2\)} (m-2-1)
            edge[exists, tden] node[right] {\(q_2\)} (m-1-2)
    (m-2-3) edge[exists, den] node[right] {\(\tilde d_2\)} (m-3-3)
            edge[exists, tden] node[right] {\(\tilde q_2\)} (m-1-3)
    (m-2-4) edge[equality] (m-3-4)
            edge[exists, sden] node[above] {\(\tilde i_2\)} (m-2-3)
            edge[equality] (m-1-4)
    (m-3-2) edge[exists] node[above] {\(\tilde f_2'\)} (m-3-3)
            edge[exists, tden] node[above] {\(\tilde p_2'\)} (m-3-1)
    (m-3-4) edge[exists, sden] node[above] {\(\tilde i_2'\)} (m-3-3)
    (m-4-1) edge[equality] (m-3-1)
    (m-4-2) edge node[above] {\(e_2 f'\)} (m-4-3)
            edge[equality] (m-4-1)
            edge[exists, sden] node[right] {\(j_2\)} (m-3-2)
    (m-4-3) edge[exists, sden] node[right] {\(\tilde j_2\)} (m-3-3)
    (m-4-4) edge[equality] (m-4-3)
            edge[equality] (m-3-4);
\end{tikzpicture}\]
We let
\[\begin{tikzpicture}[baseline=(m-2-1.base)]
  \matrix (m) [diagram without objects]{
    & \\
    & \\};
  \path[->, font=\scriptsize]
    (m-1-1) edge[den] node[left] {\(\bar i_2\)} (m-2-1)
    (m-1-2) edge[sden] node[right] {\(\tilde i_2\)} (m-2-2)
            edge[sden] node[above] {\(\tilde i_1\)} (m-1-1)
    (m-2-2) edge[sden] node[above] {\(\bar i_1\)} (m-2-1);
\end{tikzpicture}
\text{ and }
\begin{tikzpicture}[baseline=(m-2-1.base)]
  \matrix (m) [diagram without objects]{
    & \\
    & \\};
  \path[->, font=\scriptsize]
    (m-1-1) edge[den] node[left] {\(\bar i_2'\)} (m-2-1)
    (m-1-2) edge[sden] node[right] {\(\tilde i_2'\)} (m-2-2)
            edge[sden] node[above] {\(\tilde i_1'\)} (m-1-1)
    (m-2-2) edge[sden] node[above] {\(\bar i_1'\)} (m-2-1);
\end{tikzpicture}\]
be weak pushout rectangles in \(\mathcal{C}\) such that \(\bar i_1\) and \(\bar i_1'\) are S-denominators, so that we obtain morphisms \(q\), \(d\), \(j\) such that the following diagram commutes.
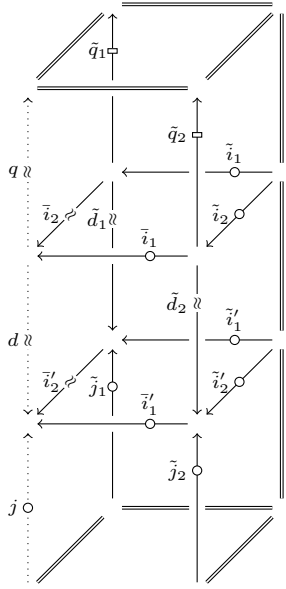
\[\begin{tikzpicture}[baseline=(m-8-1.base)]
  \matrix (m) [diagram without objects]{
    & & & \\
    & & & \\
    & & & \\
    & & & \\
    & & & \\
    & & & \\
    & & & \\
    & & & \\};
  \path[->, font=\scriptsize]
    (m-1-2) edge[equality] (m-2-1)
    (m-1-4) edge[equality] (m-2-3)
            edge[equality] (m-1-2)
    (m-3-2) edge[den=0.25]  node[left=-2pt, near start] {\(\tilde d_1\)} (m-5-2)
            edge[den] node[left] {\(\bar i_2\)} (m-4-1)
            edge[tden=0.75] node[left=-2pt, near end] {\(\tilde q_1\)} (m-1-2)
    (m-3-4) edge[equality] (m-5-4)
            edge[sden] node[left] {\(\tilde i_2\)} (m-4-3)
            edge[equality] (m-1-4)
            edge[sden=0.25] node[above, near start] {\(\tilde i_1\)} (m-3-2)
    (m-5-2) edge[den] node[left] {\(\bar i_2'\)} (m-6-1)
    (m-5-4) edge[sden] node[left] {\(\tilde i_2'\)} (m-6-3)
            edge[sden=0.25] node[above, near start] {\(\tilde i_1'\)} (m-5-2)
    (m-7-2) edge[equality] (m-8-1)
            edge[sden=0.75] node[left=-2pt, near end] {\(\tilde j_1\)} (m-5-2)
    (m-7-4) edge[equality] (m-5-4)
            edge[equality] (m-8-3)
            edge[equality] (m-7-2)
    (m-2-3) edge[cross line, equality] (m-2-1)
    (m-4-1) edge[exists, den] node[left] {\(d\)} (m-6-1)
            edge[exists, den] node[left] {\(q\)} (m-2-1)
    (m-4-3) edge[cross line, den=0.25] node[left, near start] {\(\tilde d_2\)} (m-6-3)
            edge[cross line, sden=0.25] node[above, near start] {\(\bar i_1\)} (m-4-1)
            edge[cross line, tden=0.75] node[left, near end] {\(\tilde q_2\)} (m-2-3)
    (m-6-3) edge[cross line, sden=0.25] node[above, near start] {\(\bar i_1'\)} (m-6-1)
    (m-8-1) edge[exists, sden] node[left] {\(j\)} (m-6-1)
    (m-8-3) edge[equality] (m-8-1)
            edge[cross line, sden=0.75] node[left, near end] {\(\tilde j_2\)} (m-6-3);
\end{tikzpicture}\]
Thus we obtain the following commutative diagrams.
\[\begin{tikzpicture}[baseline=(m-4-1.base)]
  \matrix (m) [diagram without objects]{
    & & & \\
    & & & \\
    & & & \\
    & & & \\};
  \path[->, font=\scriptsize]
    (m-1-2) edge node[above] {\(e_1 f\)} (m-1-3)
            edge[equality] (m-1-1)
    (m-1-4) edge[equality] (m-1-3)
    (m-2-1) edge[equality] (m-3-1)
            edge[equality] (m-1-1)
    (m-2-2) edge node[above] {\(\tilde f_1 \bar i_2\)} (m-2-3)
            edge[den] node[right] {\(d_1\)} (m-3-2)
            edge[tden] node[above] {\(\tilde p_1\)} (m-2-1)
            edge[tden] node[right] {\(q_1\)} (m-1-2)
    (m-2-3) edge[den] node[right] {\(d\)} (m-3-3)
            edge[den] node[right] {\(q\)} (m-1-3)
    (m-2-4) edge[equality] (m-3-4)
            edge[sden] node[above] {\(\tilde i_1 \bar i_2\)} (m-2-3)
            edge[equality] (m-1-4)
    (m-3-2) edge node[above] {\(\tilde f_1' \bar i_2'\)} (m-3-3)
            edge[tden] node[above] {\(\tilde p_1'\)} (m-3-1)
    (m-3-4) edge[sden] node[above] {\(\tilde i_1' \bar i_2'\)} (m-3-3)
    (m-4-1) edge[equality] (m-3-1)
    (m-4-2) edge node[above] {\(e_1 f'\)} (m-4-3)
            edge[equality] (m-4-1)
            edge[sden] node[right] {\(j_1\)} (m-3-2)
    (m-4-3) edge[sden] node[right] {\(j\)} (m-3-3)
    (m-4-4) edge[equality] (m-4-3)
            edge[equality] (m-3-4);
\end{tikzpicture}
\qquad
\begin{tikzpicture}[baseline=(m-4-1.base)]
  \matrix (m) [diagram without objects]{
    & & & \\
    & & & \\
    & & & \\
    & & & \\};
  \path[->, font=\scriptsize]
    (m-1-2) edge node[above] {\(e_2 f\)} (m-1-3)
            edge[equality] (m-1-1)
    (m-1-4) edge[equality] (m-1-3)
    (m-2-1) edge[equality] (m-3-1)
            edge[equality] (m-1-1)
    (m-2-2) edge node[above] {\(\tilde f_2 \bar i_1\)} (m-2-3)
            edge[den] node[right] {\(d_2\)} (m-3-2)
            edge[tden] node[above] {\(\tilde p_2\)} (m-2-1)
            edge[tden] node[right] {\(q_2\)} (m-1-2)
    (m-2-3) edge[den] node[right] {\(d\)} (m-3-3)
            edge[den] node[right] {\(q\)} (m-1-3)
    (m-2-4) edge[equality] (m-3-4)
            edge[sden] node[above] {\(\tilde i_2 \bar i_1\)} (m-2-3)
            edge[equality] (m-1-4)
    (m-3-2) edge node[above] {\(\tilde f_2' \bar i_1'\)} (m-3-3)
            edge[tden] node[above] {\(\tilde p_2'\)} (m-3-1)
    (m-3-4) edge[sden] node[above] {\(\tilde i_2' \bar i_1'\)} (m-3-3)
    (m-4-1) edge[equality] (m-3-1)
    (m-4-2) edge node[above] {\(e_2 f'\)} (m-4-3)
            edge[equality] (m-4-1)
            edge[sden] node[right] {\(j_2\)} (m-3-2)
    (m-4-3) edge[sden] node[right] {\(j\)} (m-3-3)
    (m-4-4) edge[equality] (m-4-3)
            edge[equality] (m-3-4);
\end{tikzpicture}\]
Using coproducts, these diagrams provide in turn the following commutative diagram.
\[\begin{tikzpicture}[baseline=(m-4-1.base)]
  \matrix (m) [diagram without objects=3.75em]{
    & & & \\
    & & & \\
    & & & \\
    & & & \\};
  \path[->, font=\scriptsize]
    (m-1-2) edge node[above] {\(\begin{smallpmatrix} e_1 f \\ e_2 f \end{smallpmatrix}\)} (m-1-3)
            edge[equality] (m-1-1)
    (m-1-4) edge[equality] (m-1-3)
    (m-2-1) edge[equality] (m-3-1)
            edge[equality] (m-1-1)
    (m-2-2) edge node[above] {\(\begin{smallpmatrix} \tilde f_1 \bar i_2 \\ \tilde f_2 \bar i_1 \end{smallpmatrix}\)} (m-2-3)
            edge[den] node[right, near start] {\(d_1 \smallcoprod d_2\)} (m-3-2)
            edge[den] node[above] {\(\tilde p_1 \smallcoprod \tilde p_2\)} (m-2-1)
            edge[den] node[right, near end] {\(q_1 \smallcoprod q_2\)} (m-1-2)
    (m-2-3) edge[den] node[right] {\(d\)} (m-3-3)
            edge[den] node[right] {\(q\)} (m-1-3)
    (m-2-4) edge[equality] (m-3-4)
            edge[sden] node[above] {\(\tilde i_1 \bar i_2\)} (m-2-3)
            edge[equality] (m-1-4)
    (m-3-2) edge node[above] {\(\begin{smallpmatrix} \tilde f_1' \bar i_2' \\ \tilde f_2' \bar i_1' \end{smallpmatrix}\)} (m-3-3)
            edge[den] node[above] {\(\tilde p_1' \smallcoprod \tilde p_2'\)} (m-3-1)
    (m-3-4) edge[sden] node[above] {\(\tilde i_1' \bar i_2'\)} (m-3-3)
    (m-4-1) edge[equality] (m-3-1)
    (m-4-2) edge node[above] {\(\begin{smallpmatrix} e_1 f' \\ e_2 f' \end{smallpmatrix}\)} (m-4-3)
            edge[equality] (m-4-1)
            edge[den] node[right, near end] {\(j_1 \smallcoprod j_2\)} (m-3-2)
    (m-4-3) edge[sden] node[right] {\(j\)} (m-3-3)
    (m-4-4) edge[equality] (m-4-3)
            edge[equality] (m-3-4);
\end{tikzpicture}\]
We finally have
\[\LocalisationFunctor(\begin{smallpmatrix} e_1 \\ e_2 \end{smallpmatrix}) \LocalisationFunctor(f) = \LocalisationFunctor(\begin{smallpmatrix} e_1 \\ e_2 \end{smallpmatrix} f) = \LocalisationFunctor(\begin{smallpmatrix} e_1 f \\ e_2 f \end{smallpmatrix}) = \LocalisationFunctor(\begin{smallpmatrix} e_1 f' \\ e_2 f' \end{smallpmatrix}) = \LocalisationFunctor(\begin{smallpmatrix} e_1 \\ e_2 \end{smallpmatrix} f') = \LocalisationFunctor(\begin{smallpmatrix} e_1 \\ e_2 \end{smallpmatrix}) \LocalisationFunctor(f').\]
On the other hand,
\[\begin{smallpmatrix} e_1 \\ e_2 \end{smallpmatrix} p = \begin{smallpmatrix} e_1 p \\ e_2 p \end{smallpmatrix} = \begin{smallpmatrix} p_1 \emb_1^{X_1 \smallcoprod X_2} \\ p_2 \emb_2^{X_1 \smallcoprod X_2} \end{smallpmatrix} = p_1 \smallcoprod p_2\]
implies that \(\begin{smallpmatrix} e_1 \\ e_2 \end{smallpmatrix}\) is a denominator in \(\mathcal{C}\) by semi-saturatedness, so we have \(\LocalisationFunctor(f) = \LocalisationFunctor(f')\) and therefore
\[\varphi = \doublefrac{p}{f}{i} = \doublefrac{p}{f'}{i} = \varphi'.\]
Altogether, \(\LocalisationFunctor(X_1 \smallcoprod X_2)\) is a coproduct of \(\LocalisationFunctor(X_1)\) and \(\LocalisationFunctor(X_2)\) with embeddings \(\emb_k^{\LocalisationFunctor(X_1 \smallcoprod X_2)} = \LocalisationFunctor(\emb_k^{X_1 \smallcoprod X_2})\) for \(k \in \{1, 2\}\).
\item We suppose that \(\mathcal{C}\) is saturated and that \(\LocalisationFunctor\) preserves finite coproducts. Moreover, we suppose given denominators \(d_1\colon X_1 \map Y_1\) and \(d_2\colon X_2 \map Y_2\) in \(\mathcal{C}\). Then we have
\begin{align*}
\LocalisationFunctor(d_k) \emb_k^{\LocalisationFunctor(Y_1 \smallcoprod Y_2)} & = \LocalisationFunctor(d_k) \LocalisationFunctor(\emb_k^{Y_1 \smallcoprod Y_2}) = \LocalisationFunctor(d_k \emb_k^{Y_1 \smallcoprod Y_2}) = \LocalisationFunctor(\emb_k^{X_1 \smallcoprod X_2} (d_1 \smallcoprod d_2)) \\
& = \LocalisationFunctor(\emb_k^{X_1 \smallcoprod X_2}) \LocalisationFunctor(d_1 \smallcoprod d_2) = \emb_k^{\LocalisationFunctor(X_1 \smallcoprod X_2)} \LocalisationFunctor(d_1 \smallcoprod d_2).
\end{align*}
Since \(d_1\) and \(d_2\) are denominators, \(\LocalisationFunctor(d_1)\) and \(\LocalisationFunctor(d_2)\) are isomorphisms. But then \(\LocalisationFunctor(d_1 \smallcoprod d_2)\) is also an isomorphism and hence \(d_1 \smallcoprod d_2\) is a denominator since \(\mathcal{C}\) is saturated.
\end{enumerate}
\item This is dual to~\ref{prop:uni-fractionable_categories_admitting_finite_co_products:coproducts}.
\item This follows from~\ref{prop:uni-fractionable_categories_admitting_finite_co_products:coproducts} and~\ref{prop:uni-fractionable_categories_admitting_finite_co_products:products}. \qedhere
\end{enumerate}
\end{proof}

The preceding criterion motivates the next definition.

\begin{definition} \label{def:additive_uni-fractionable_category}
An \newnotion{additive uni-fractionable category} is a uni-fractionable category \(\mathcal{A}\) such that the underlying category of \(\mathcal{A}\) is equipped with the structure of an additive category and such that \(\Denominators \mathcal{A}\) is closed under finite sums.
\end{definition}

\begin{remark} \label{rem:denominators_closed_under_finite_co_products}
We suppose given a uni-fractionable category \(\mathcal{C}\).
\begin{enumerate}
\item \label{rem:denominators_closed_under_finite_co_products:coproducts} We suppose that \(\mathcal{C}\) admits finite coproducts. Then \(\Denominators \mathcal{C}\) is closed under finite coproducts if and only if \(i \smallcoprod j\) is a denominator for all S-denominators \(i\), \(j\) in \(\mathcal{C}\) and \(p \smallcoprod q\) is a denominator for all T-denominators \(p\), \(q\) in \(\mathcal{C}\).
\item \label{rem:denominators_closed_under_finite_co_products:products} We suppose that \(\mathcal{C}\) admits finite products. Then \(\Denominators \mathcal{C}\) is closed under finite products if and only if \(i \smallprod j\) is a denominator for all S-denominators \(i\), \(j\) in \(\mathcal{C}\) and \(p \smallprod q\) is a denominator for all T-denominators \(p\), \(q\) in \(\mathcal{C}\).
\end{enumerate}
\end{remark}
\begin{proof} \
\begin{enumerate}
\item If \(\Denominators \mathcal{C}\) is closed under finite coproducts, then in particular \(i \smallcoprod j\) is a denominator for all S-denominators \(i\), \(j\), and \(p \smallcoprod q\) is a denominator for all T-denominators \(p\), \(q\) in \(\mathcal{C}\). So let us conversely suppose that \(i \smallcoprod j\) is a denominator for all S-denominators \(i\), \(j\) and that \(p \smallcoprod q\) is a denominator for all T-denominators \(p\), \(q\), and let us suppose given denominators \(d\), \(e\) in \(\mathcal{C}\). Then there exist S-denominators \(i\), \(j\) and T-denominators \(p\), \(q\) with \(d = i p\) and \(e = j q\), and hence
\[d \smallcoprod e = (i p) \smallcoprod (j q) = (i \smallcoprod j) (p \smallcoprod q)\]
is a denominator in \(\mathcal{C}\) by multiplicativity. Thus \(\Denominators \mathcal{C}\) is closed under finite coproducts. \qedhere
\end{enumerate}
\end{proof}

Recall that every hom-set in a category that admits finite sums carries a unique structure of a commutative monoid such that addition of morphisms becomes compatible with composition~\cite[ch.~VIII, sec.~2, ex.~4(a)]{maclane:1998:categories_for_the_working_mathematician}. In modern terms: Such a category is enriched over the category of commutative monoids in a unique way. Moreover, every hom-set becomes an abelian group, that is, the category under consideration is additive, if and only if every identity has a negative element with respect to the addition on its hom-set~\cite[ch.~VIII, sec.~2, ex.~4(b)]{maclane:1998:categories_for_the_working_mathematician}. The latter condition is equivalent to the condition that the morphism \(\begin{smallpmatrix} 1 & 0 \\ 1 & 1 \end{smallpmatrix}\) is always an isomorphism. A functor between additive categories is additive if and only if it preserves finite sums, that is, if and only if the image of every (chosen) finite sum is a finite sum of the images, such that the embeddings resp.\ projections are the images of the embeddings resp.\ projections. Cf.\ also~\cite[sec.~18--19]{maclane:1950:duality_for_groups},~\cite[sec.~3.1--3.2]{kuenzer:2010:homologische_algebra}.

\begin{proposition} \label{prop:additive_uni-fractionable_categories_induce_additive_fraction_categories_and_additive_localisation_functors}
Given an additive uni-fractionable category \(\mathcal{A}\), the additive structure of \(\mathcal{A}\) induces an additive structure on the fraction category \(\FractionCategory \mathcal{A}\) such that the localisation functor \(\LocalisationFunctor\colon \mathcal{A} \map \FractionCategory \mathcal{A}\) becomes an additive functor. For parallel \(3\)-arrows \((b, f, a)\) and \((b, g, a)\) in \(\mathcal{A}\) (cf.\ proposition~\ref{prop:morphisms_in_the_fraction_category_can_have_the_same_denominators}), we have
\[\doublefrac{b}{f}{a} + \doublefrac{b}{g}{a} = \doublefrac{b}{(f + g)}{a}.\]
\end{proposition}
\begin{proof}
By proposition~\ref{prop:uni-fractionable_categories_admitting_finite_co_products}\ref{prop:uni-fractionable_categories_admitting_finite_co_products:sums}\ref{prop:uni-fractionable_categories_admitting_finite_co_products:sums:denominators_closedness_implies_preservation_by_localisation}, \(\LocalisationFunctor(0)\) is a zero object in \(\FractionCategory \mathcal{A}\), and for objects \(X_1\), \(X_2\) in \(\mathcal{A}\), the object \(\LocalisationFunctor(X_1 \directsum X_2)\) is a sum of \(\LocalisationFunctor(X_1)\) and \(\LocalisationFunctor(X_2)\) in \(\FractionCategory \mathcal{A}\) with \(\emb_k^{\LocalisationFunctor(X_1 \directsum X_2)} = \LocalisationFunctor(\emb_k^{X_1 \directsum X_2})\) and \(\pr_k^{\LocalisationFunctor(X_1 \directsum X_2)} = \LocalisationFunctor(\pr_k^{X_1 \directsum X_2})\) for \(k \in \{1, 2\}\). Thus \(\FractionCategory \mathcal{A}\) admits finite sums. For the purpose of this proof, let us choose \(\LocalisationFunctor(X_1) \directsum \LocalisationFunctor(X_2) := \LocalisationFunctor(X_1 \directsum X_2)\) for \(X_1, X_2 \in \Ob \mathcal{A}\), so that we can use matrix notation for induced morphisms between those objects. Then we have \(\begin{smallpmatrix} 1 & 0 \\ 1 & 1 \end{smallpmatrix} = \LocalisationFunctor(\begin{smallpmatrix} 1 & 0 \\ 1 & 1 \end{smallpmatrix})\colon \LocalisationFunctor(X) \directsum \LocalisationFunctor(X) \map \LocalisationFunctor(X) \directsum \LocalisationFunctor(X)\) for every object \(X\) in \(\mathcal{A}\), and so \(\begin{smallpmatrix} 1 & 0 \\ 1 & 1 \end{smallpmatrix}\) is an isomorphism. Altogether, \(\FractionCategory \mathcal{A}\) is an additive category and \(\LocalisationFunctor\colon \mathcal{A} \map \FractionCategory \mathcal{A}\) is an additive functor. In particular, we obtain
\begin{align*}
\doublefrac{b}{f}{a} + \doublefrac{b}{g}{a} & = (\LocalisationFunctor(b))^{- 1} \LocalisationFunctor(f) (\LocalisationFunctor(a))^{- 1} + (\LocalisationFunctor(b))^{- 1} \LocalisationFunctor(g) (\LocalisationFunctor(a))^{- 1} \\
& = (\LocalisationFunctor(b))^{- 1} (\LocalisationFunctor(f) + \LocalisationFunctor(g)) (\LocalisationFunctor(a))^{- 1} = (\LocalisationFunctor(b))^{- 1} \LocalisationFunctor(f + g) (\LocalisationFunctor(a))^{- 1} \\
& = \doublefrac{b}{(f + g)}{a}
\end{align*}
for parallel \(3\)-arrows \((b, f, a)\) and \((b, g, a)\) in \(\mathcal{A}\).
\end{proof}

\section{Applications} \label{sec:applications}

In this final section, we consider some examples and applications.

\subsection*{Quillen model categories} \label{ssec:quillen_model_categories}

Given a Quillen model category \(\mathcal{M}\)~\cite[ch.~I, \S 1, def.~1]{quillen:1967:homotopical_algebra}, we denote by \(\Cof(\mathcal{M})\) the full subcategory of cofibrant objects, by \(\Fib(\mathcal{M})\) the full subcategory of fibrant objects and by \(\Bif(\mathcal{M})\) the full subcategory of bifibrant (that is, cofibrant and fibrant) objects.

\begin{example} \label{ex:model_categories_are_uni-fractionable_categories}
Given a Quillen model category \(\mathcal{M}\), the categories \(\mathcal{M}\), \(\Cof(\mathcal{M})\), \(\Fib(\mathcal{M})\), \(\Bif(\mathcal{M})\) carry the structure of uni-fractionable categories, where
\begin{align*}
\Denominators \mathcal{C} & = \{w \in \Mor \mathcal{C} \mid \text{\(w\) is a weak equivalence}\}, \\
\SDenominators \mathcal{C} & = \{i \in \Mor \mathcal{C} \mid \text{\(i\) is an acyclic cofibration}\}, \\
\TDenominators \mathcal{C} & = \{p \in \Mor \mathcal{C} \mid \text{\(p\) is an acyclic fibration}\}
\end{align*}
for \(\mathcal{C} \in \{\mathcal{M}, \Cof(\mathcal{M}), \Fib(\mathcal{M}), \Bif(\mathcal{M})\}\). In particular, the homotopy category \(\HomotopyCategory \mathcal{M}\) is isomorphic to \(\FractionCategory \mathcal{M}\). If \(\mathcal{M}\) is a closed Quillen model category, then \(\Denominators \mathcal{C}\) is saturated for \(\mathcal{C} \in \{\mathcal{M}, \Cof(\mathcal{M}), \Fib(\mathcal{M}), \Bif(\mathcal{M})\}\). The localisation functor \(\LocalisationFunctor\colon \mathcal{C} \map \FractionCategory \mathcal{C}\) preserves finite coproducts for \(\mathcal{C} \in \{\Cof(\mathcal{M}), \Bif(\mathcal{M})\}\) and finite products for \(\mathcal{C} \in \{\Fib(\mathcal{M}), \Bif(\mathcal{M})\}\). (\footnote{In general, the localisation functor \(\LocalisationFunctor\colon \mathcal{M} \map \FractionCategory \mathcal{M}\) does not preserve finite coproducts or finite products since the set of denominators in a closed Quillen model category need not be closed under finite (co)products. A counterexample is provided by \((\Integers / 4 \downarrow \fgMod(\Integers / 4))\), cf.~\cite[rem.~3.11]{dwyer_spalinski:1995:homotopy_theories_and_model_categories}, as considered in~\cite[ex.]{dwyer_radulescu-banu_thomas:2010:faithfulness_of_a_functor_of_quillen}: The coproduct of \(2\colon (\Integers / 4, 1) \map (\Integers / 4, 2)\) with itself is given by \(\begin{smallpmatrix} 2 & 0 \end{smallpmatrix}\colon (\Integers / 4, 1) \map (\Integers / 4 \directsum \Integers / 2, \begin{smallpmatrix} 2 & 0 \end{smallpmatrix})\); the former is a weak equivalence, but the latter is not since \(\Integers / 4\) is a bijective object and \(\Integers / 4 \directsum \Integers / 2\) is not a bijective object in \(\fgMod(\Integers / 4)\).})
\end{example}
\begin{proof} \
\begin{enumerate}
\item \label{ex:model_categories_are_uni-fractionable_categories:proof:whole_model_category} We consider \(\mathcal{M}\) and verify the axioms of a uni-fractionable category.
\begin{itemize}
\item[(Cat)] By definition of a Quillen model category, weak equivalences, cofibrations and fibrations are closed under composition and contain all isomorphisms. Hence in particular weak equivalences, acyclic cofibrations and acyclic fibrations are closed under composition and contain all identities.
\item[(2\,of\,3)] This holds by definition of a Quillen model category.
\item[(WU)] We suppose given an acyclic cofibration \(i\colon X \map X'\) and a morphism \(f\colon X \map Y\) in \(\mathcal{M}\), and we let
\[\begin{tikzpicture}[baseline=(m-2-1.base)]
  \matrix (m) [diagram]{
    X' & Y' \\
    X & Y \\};
  \path[->, font=\scriptsize]
    (m-1-1) edge node[above] {\(f'\)} (m-1-2)
    (m-2-1) edge node[above] {\(f\)} (m-2-2)
            edge node[left] {\(i\)} (m-1-1)
    (m-2-2) edge node[right] {\(i'\)} (m-1-2);
\end{tikzpicture}\]
be a pushout rectangle in \(\mathcal{C}\). Then \(i'\) is an acyclic cofibration.

The other assertion follows by duality.
\item[(Fac)] Since every morphism decomposes into a composite of a cofibration followed by an acyclic fibration, the assertion follows by semi-saturatedness.
\end{itemize}
Altogether, \(\mathcal{M}\) becomes a uni-fractionable category with
\begin{align*}
\Denominators \mathcal{M} & = \{w \in \Mor \mathcal{M} \mid \text{\(w\) is a weak equivalence}\}, \\
\SDenominators \mathcal{M} & = \{i \in \Mor \mathcal{M} \mid \text{\(i\) is an acyclic cofibration}\}, \\
\TDenominators \mathcal{M} & = \{p \in \Mor \mathcal{M} \mid \text{\(p\) is an acyclic fibration}\}.
\end{align*}
The assertion on the saturatedness of \(\mathcal{M}\) is proven in~\cite[ch.~I, \S 5, prop.~1]{quillen:1967:homotopical_algebra}.
\item \label{ex:model_categories_are_uni-fractionable_categories:proof:cofibrant_objects} We consider \(\Cof(\mathcal{M})\) and have to verify the axioms of a uni-fractionable category. Since (Cat) and (2\,of\,3) hold for \(\mathcal{M}\) by~\ref{ex:model_categories_are_uni-fractionable_categories:proof:whole_model_category}, they hold in particular for \(\Cof(\mathcal{M})\).
\begin{itemize}
\item[(WU)] We suppose given an acyclic cofibration \(i\colon X \map X'\) and a morphism \(f\colon X \map Y\) in \(\Cof(\mathcal{M})\), and we let
\[\begin{tikzpicture}[baseline=(m-2-1.base)]
  \matrix (m) [diagram without objects]{
    X' & Y' \\
    X & Y \\};
  \path[->, font=\scriptsize]
    (m-1-1) edge node[above] {\(f'\)} (m-1-2)
    (m-2-1) edge node[above] {\(f\)} (m-2-2)
            edge node[left] {\(i\)} (m-1-1)
    (m-2-2) edge node[right] {\(i'\)} (m-1-2);
\end{tikzpicture}\]
be a pushout rectangle in \(\mathcal{C}\). Then \(i'\) is an acyclic cofibration, and since \(Y\) is cofibrant and \(i'\) is in particular a cofibration, it follows that \(Y'\) is also cofibrant.

Now we suppose given an acyclic fibration \(p\colon Y' \map Y\) and a morphism \(f\colon X \map Y\) in \(\Cof(\mathcal{M})\), and we let
\[\begin{tikzpicture}[baseline=(m-2-1.base)]
  \matrix (m) [diagram]{
    X' & Y' \\
    X & Y \\};
  \path[->, font=\scriptsize]
    (m-1-1) edge node[above] {\(f'\)} (m-1-2)
            edge node[left] {\(p'\)} (m-2-1)
    (m-1-2) edge node[right] {\(p\)} (m-2-2)
    (m-2-1) edge node[above] {\(f\)} (m-2-2);
\end{tikzpicture}\]
be a pullback rectangle in \(\mathcal{C}\). Then \(p'\) is an acyclic fibration. We consider a strong cofibrant approximation of \(X'\), that is, we let \(\tilde X'\) be a cofibrant object together with an acyclic fibration \(q\colon \tilde X' \map X'\). The composite \(q p'\) is an acyclic fibration by multiplicativity. We will show that
\[\begin{tikzpicture}[baseline=(m-2-1.base)]
  \matrix (m) [diagram without objects]{
    \tilde X' & Y' \\
    X & Y \\};
  \path[->, font=\scriptsize]
    (m-1-1) edge node[above] {\(q f'\)} (m-1-2)
            edge node[left] {\(q p'\)} (m-2-1)
    (m-1-2) edge node[right] {\(p\)} (m-2-2)
    (m-2-1) edge node[above] {\(f\)} (m-2-2);
\end{tikzpicture}\]
is a weak pullback of \(f\) along \(p\). To this end, we suppose given an object \(T \in \Ob \Cof(\mathcal{M})\) and morphisms \(s\colon T \map X\), \(t\colon T \map Y'\) with \(s f = t p\). By the universal property of \(X'\), there exists a (unique) morphism \(u\colon T \map X'\) such that \(u p' = s\) and \(u f' = t\).
\[\begin{tikzpicture}[baseline=(m-3-1.base)]
  \matrix (m) [diagram]{
    T & & \\
    & X' & Y' \\
    & X & Y \\};
  \path[->, font=\scriptsize]
    (m-1-1) edge[out=0, in=130] node[above] {\(t\)} (m-2-3)
            edge[exists] node[right] {\(u\)} (m-2-2)
            edge[out=-90, in=140] node[left] {\(s\)} (m-3-2)
    (m-2-2) edge node[above] {\(f'\)} (m-2-3)
            edge node[left] {\(p'\)} (m-3-2)
    (m-2-3) edge node[right] {\(p\)} (m-3-3)
    (m-3-2) edge node[above] {\(f\)} (m-3-3);
\end{tikzpicture}\]
Moreover, since \(T\) is cofibrant and \(q\) is an acyclic fibration, there exists a lift \(\hat u\colon T \map \tilde X'\) such that \(u = \hat u q\).
\[\begin{tikzpicture}[baseline=(m-2-1.base)]
  \matrix (m) [diagram]{
    & \tilde X' \\
    T & X' \\};
  \path[->, font=\scriptsize]
    (m-1-2) edge node[right] {\(q\)} (m-2-2)
    (m-2-1) edge node[above] {\(u\)} (m-2-2)
            edge[exists] node[left] {\(\hat u\)} (m-1-2);
\end{tikzpicture}\]
Now we have \(\hat u q p' = u p' = s\) and \(\hat u q f' = u f' = t\).
\[\begin{tikzpicture}[baseline=(m-3-1.base)]
  \matrix (m) [diagram]{
    T & & \\
    & \tilde X' & Y' \\
    & X & Y \\};
  \path[->, font=\scriptsize]
    (m-1-1) edge[out=0, in=130] node[above] {\(t\)} (m-2-3)
            edge node[right] {\(\hat u\)} (m-2-2)
            edge[out=-90, in=140] node[left] {\(s\)} (m-3-2)
    (m-2-2) edge node[above] {\(q f'\)} (m-2-3)
            edge node[left] {\(q p'\)} (m-3-2)
    (m-2-3) edge node[right] {\(p\)} (m-3-3)
    (m-3-2) edge node[above] {\(f\)} (m-3-3);
\end{tikzpicture}\]
\item[(Fac)] We let \(w\colon X \map Y\) be a weak equivalence in \(\Cof(\mathcal{M})\). Then there exists an acyclic cofibration \(i\colon X \map Z\) and an acyclic fibration \(p\colon Z \map Y\) in \(\mathcal{M}\) with \(w = i p\).
\[\begin{tikzpicture}[baseline=(m-2-1.base)]
  \matrix (m) [diagram=0.9em]{
    & Z & \\
    X & & Y \\};
  \path[->, font=\scriptsize]
    (m-1-2) edge[exists] node[right=2pt] {\(p\)} (m-2-3)
    (m-2-1) edge node[above] {\(w\)} (m-2-3)
            edge[exists] node[left] {\(i\)} (m-1-2);
\end{tikzpicture}\]
But since \(X\) is cofibrant and \(i\) is a cofibration, \(Z\) is cofibrant, too.
\end{itemize}
Altogether, \(\Cof(\mathcal{M})\) becomes a uni-fractionable category with
\begin{align*}
\Denominators \Cof(\mathcal{M}) & = \{w \in \Mor \Cof(\mathcal{M}) \mid \text{\(w\) is a weak equivalence}\}, \\
\SDenominators \Cof(\mathcal{M}) & = \{i \in \Mor \Cof(\mathcal{M}) \mid \text{\(i\) is an acyclic cofibration}\}, \\
\TDenominators \Cof(\mathcal{M}) & = \{p \in \Mor \Cof(\mathcal{M}) \mid \text{\(p\) is an acyclic fibration}\}.
\end{align*}
The assertion on the saturatedness of \(\Cof(\mathcal{M})\) follows from~\ref{ex:model_categories_are_uni-fractionable_categories:proof:whole_model_category} since if \(\LocalisationFunctor[\FractionCategory \Cof(\mathcal{M})](f)\) is an isomorphism, then also \(\LocalisationFunctor[\FractionCategory \mathcal{M}](f)\) is an isomorphism. The fact that the localisation functor \(\LocalisationFunctor\colon \Cof(\mathcal{M}) \map \FractionCategory \Cof(\mathcal{M})\) preserves finite coproducts follows from the gluing lemma~\cite[lem.~7.4]{gunnarsson:1978:abstract_homotopy_theory_and_related_topics}, cf.\ also~\cite[ch.~II, lem.~8.8]{goerss_jardine:1999:simplicial_homotopy_theory}, and proposition~\ref{prop:uni-fractionable_categories_admitting_finite_co_products}\ref{prop:uni-fractionable_categories_admitting_finite_co_products:coproducts}\ref{prop:uni-fractionable_categories_admitting_finite_co_products:coproducts:denominators_closedness_implies_preservation_by_localisation}.
\item \label{ex:model_categories_are_uni-fractionable_categories:proof:fibrant_objects} This is dual to~\ref{ex:model_categories_are_uni-fractionable_categories:proof:cofibrant_objects}.
\item \label{ex:model_categories_are_uni-fractionable_categories:proof:bifibrant_objects} This is a combination of~\ref{ex:model_categories_are_uni-fractionable_categories:proof:cofibrant_objects} and~\ref{ex:model_categories_are_uni-fractionable_categories:proof:fibrant_objects}. \qedhere
\end{enumerate}
\end{proof}

As an application of our abstract machinery, we obtain the following part of Quillen's homotopy category theorem~\cite[ch.~I, \S 1, th.~1]{quillen:1967:homotopical_algebra}. Given a Quillen model category \(\mathcal{M}\), we (re-)define the \newnotion{homotopy category} of \(\mathcal{C} \in \{\mathcal{M}, \Cof(\mathcal{M}), \Fib(\mathcal{M}), \Bif(\mathcal{M})\}\) by \(\HomotopyCategory \mathcal{C} := \FractionCategory \mathcal{C}\), using the uni-fractionable category structures from the preceding example.

\begin{example} \label{ex:part_of_quillens_homotopy_category_theorem}
We suppose given a Quillen model category \(\mathcal{M}\). The commutative diagram of inclusion functors
\[\begin{tikzpicture}[baseline=(m-3-2.base)]
  \matrix (m) [diagram]{
    & \Cof(\mathcal{M}) & \\
    \Bif(\mathcal{M}) & & \mathcal{M} \\
    & \Fib(\mathcal{M}) & \\};
  \path[->, font=\scriptsize]
    (m-1-2) edge node[right=1pt] {\(\inc\)} (m-2-3)
    (m-2-1) edge node[left] {\(\inc\)} (m-3-2)
            edge node[left=1pt] {\(\inc\)} (m-1-2)
    (m-3-2) edge node[right] {\(\inc\)} (m-2-3);
\end{tikzpicture}\]
induces a commutative diagram of equivalences
\[\begin{tikzpicture}[baseline=(m-3-2.base)]
  \matrix (m) [diagram]{
    & \HomotopyCategory{\Cof(\mathcal{M})} & \\
    \HomotopyCategory{\Bif(\mathcal{M})} & & \HomotopyCategory \mathcal{M} \\
    & \HomotopyCategory{\Fib(\mathcal{M})} & \\};
  \path[->, font=\scriptsize]
    (m-1-2) edge node[sloped, above] {\(\categoricalequivalent\)} (m-2-3)
    (m-2-1) edge node[sloped, above] {\(\categoricalequivalent\)} (m-1-2)
            edge node[sloped, above] {\(\categoricalequivalent\)} (m-3-2)
    (m-3-2) edge node[sloped, above] {\(\categoricalequivalent\)} (m-2-3);
\end{tikzpicture}.\]

In particular, \(\HomotopyCategory{\Bif(\mathcal{M})} \categoricalequivalent \HomotopyCategory \mathcal{M}\).
\end{example}
\begin{proof}
This follows from proposition~\ref{prop:full_uni-fractionable_subcategory_induces_equivalence_of_the_fraction_categories}, using criterion~\ref{prop:full_uni-fractionable_subcategory_induces_equivalence_of_the_fraction_categories:s-resolution} for \(\inc\colon \Cof(\mathcal{M}) \map \mathcal{M}\) and \(\inc\colon \Bif(\mathcal{M}) \map \Fib(\mathcal{M})\), and using criterion~\ref{prop:full_uni-fractionable_subcategory_induces_equivalence_of_the_fraction_categories:t-resolution} for \(\inc\colon \Fib(\mathcal{M}) \map \mathcal{M}\) and \(\inc\colon \Bif(\mathcal{M}) \map \Cof(\mathcal{M})\).
\end{proof}

The proof of Quillen's homotopy category theorem, which states in particular that the homotopy category \(\HomotopyCategory \mathcal{M}\) is equivalent to the quotient category \(\Bif(\mathcal{M}) / {\homotopic}\), where \({\homotopic}\) denotes the homotopy congruence, can now be completed as in~\cite[cor.~1.2.9]{hovey:1999:model_categories} by showing that \(\Bif(\mathcal{M}) / {\homotopic}\) fulfills the universal property of a localisation, which is essentially a corollary of Whitehead's theorem~\cite[prop.~1.2.8]{hovey:1999:model_categories}.

\subsection*{Derivable categories} \label{ssec:derivable_categories}

Recall that a \newnotion{derivable category} in the sense of \eigenname{Cisinski}~\cite[sec.~2.25]{cisinski:2010:categories_derivables} consists of the same data as a Quillen model category, that is, a category \(\mathcal{C}\) together with three distinguished subsets of morphisms, called \newnotion{cofibrations}, \newnotion{fibrations} and \newnotion{weak equivalences}, subject to the following axioms, where (co)fibrant objects and acyclic (co)fibrations are defined as in the Quillen model category case: The set of weak equivalences is supposed to be semi-saturated. The set of cofibrations is supposed to be closed under (binary) composition. There exists an initial object in \(\mathcal{C}\), which is supposed to be cofibrant. The set of cofibrant objects is supposed to be closed under isomorphisms. The set of cofibrations between cofibrant objects and the subset of acyclic cofibrations therein are supposed to be stable under pushouts along morphisms between cofibrant objects. Every morphism with cofibrant source object factors into a cofibration followed by a weak equivalence. And dually for the fibrations and fibrant objects.

For homotopical algebra in derivable categories, cf.\ also the manuscript of \eigenname{R{\u{a}}dulescu-Banu}~\cite{radulescu-banu:2006:cofibrations_in_homotopy_theory}, who uses the terminology \newnotion{Anderson-Brown-Cisinski premodel category}.

Derivable categories are a natural generalisation of \newnotion{categories of fibrant objects} in the sense of \eigenname{K.~Brown}~\cite[sec.~1]{brown:1974:abstract_homotopy_theory_and_generalized_sheaf_cohomology}. More precisely: Given a derivable category, then its full subcategory of fibrant objects is a category of fibrant objects in this sense, and its full subcategory of cofibrant objects fulfills the dual properties. Conversely, given a category \(\mathcal{C}\) together with distinguished subsets of cofibrations, fibrations and weak equivalences such that there exists an terminal object in \(\mathcal{C}\) and such that the full subcategory of fibrant objects is a category of fibrant objects in the sense of \eigenname{K.~Brown}, and dually, then \(\mathcal{C}\) fulfills all axioms of a derivable category except possibly for the stronger factorisation axioms of \eigenname{Cisinski}. These stronger factorisation axioms are sufficient to obtain the desirable equivalences between the homotopy categories of the full subcategories of (co)fibrant objects in \(\mathcal{C}\) and the homotopy category of \(\mathcal{C}\), see~\cite[prop.~1.8]{cisinski:2010:categories_derivables}.

In the proof of example~\ref{ex:model_categories_are_uni-fractionable_categories}, we have not used the existence of general finite limits and colimits~\cite[ch.~I, \S 1, def.~1, ax.~M0]{quillen:1967:homotopical_algebra}. Moreover, to show that a Quillen model category carries the structure of a uni-fractionable category, we also did not use the lifting axioms~\cite[ch.~I, \S 1, def.~1, ax.~M1]{quillen:1967:homotopical_algebra}. Thus we obtain the following more general example.

\begin{example} \label{ex:derivable_categories_with_additional_properties_are_uni-fractionable_categories}
We let \(\mathcal{C}\) be a derivable category such that the following properties hold.
\begin{itemize}
\item Every identity in \(\mathcal{C}\) is a cofibration and a fibration. (\footnote{This is no restriction: Given a derivable category \(\mathcal{C}\) with set of cofibrations \(C\), set of fibrations \(F\) and set of weak equivalences \(W\), the underlying category of \(\mathcal{C}\) also becomes a derivable category with set of cofibrations \(C \union \{1_X \mid X \in \Ob \mathcal{C}\}\), set of fibrations \(F \union \{1_X \mid X \in \Ob \mathcal{C}\}\) and set of weak equivalences \(W\).})
\item Given an acyclic cofibration \(i\colon X \map X'\) and a morphism \(f\colon X \map Y\) in \(\mathcal{C}\), there exists a pushout rectangle
\[\begin{tikzpicture}[baseline=(m-2-1.base)]
  \matrix (m) [diagram]{
    X' & Y' \\
    X & Y \\};
  \path[->, font=\scriptsize]
    (m-1-1) edge node[above] {\(f'\)} (m-1-2)
    (m-2-1) edge node[above] {\(f\)} (m-2-2)
            edge node[left] {\(i\)} (m-1-1)
    (m-2-2) edge node[right] {\(i'\)} (m-1-2);
\end{tikzpicture}\]
in \(\mathcal{C}\) such that \(i'\) is an acyclic cofibration. Dually, given an acyclic fibration \(p\colon Y' \map Y\) and a morphism \(f\colon X \map Y\) in \(\mathcal{C}\), there exists a pullback rectangle
\[\begin{tikzpicture}[baseline=(m-2-1.base)]
  \matrix (m) [diagram]{
    X' & Y' \\
    X & Y \\};
  \path[->, font=\scriptsize]
    (m-1-1) edge node[above] {\(f'\)} (m-1-2)
            edge node[left] {\(p'\)} (m-2-1)
    (m-1-2) edge node[right] {\(p\)} (m-2-2)
    (m-2-1) edge node[above] {\(f\)} (m-2-2);
\end{tikzpicture}\]
in \(\mathcal{C}\) such that \(p'\) is an acyclic fibration.
\item For every weak equivalence \(w\colon X \map Y\) in \(\mathcal{C}\) there exists an acyclic cofibration \(i\colon X \map Z\) and an acyclic fibration \(p\colon Z \map Y\) with \(w = i p\).
\[\begin{tikzpicture}[baseline=(m-2-1.base)]
  \matrix (m) [diagram=0.9em]{
    & Z & \\
    X & & Y \\};
  \path[->, font=\scriptsize]
    (m-1-2) edge[exists] node[right=2pt] {\(p\)} (m-2-3)
    (m-2-1) edge node[above] {\(w\)} (m-2-3)
            edge[exists] node[left] {\(i\)} (m-1-2);
\end{tikzpicture}\]
\end{itemize}
Then \(\mathcal{C}\) carries the structure of a uni-fractionable category, where
\begin{align*}
\Denominators \mathcal{C} & = \{w \in \Mor \mathcal{C} \mid \text{\(w\) is a weak equivalence}\}, \\
\SDenominators \mathcal{C} & = \{i \in \Mor \mathcal{C} \mid \text{\(i\) is an acyclic cofibration}\}, \\
\TDenominators \mathcal{C} & = \{p \in \Mor \mathcal{C} \mid \text{\(p\) is an acyclic fibration}\}.
\end{align*}
\end{example}
\begin{proof}
This is the same proof as for a Quillen model category, see part~\ref{ex:model_categories_are_uni-fractionable_categories:proof:whole_model_category} of the proof of example~\ref{ex:model_categories_are_uni-fractionable_categories}.
\end{proof}

\subsection*{Complexes and exact categories} \label{ssec:complexes_and_exact_categories}

Our next example yields a construction for the derived category of an arbitrary abelian category \(\mathcal{A}\). We denote by \(\CohomologyGroup\colon \Com(\mathcal{A}) \map \mathcal{A}^{\DiscreteCategory \Integers}\) the cohomology functor, where \(\DiscreteCategory \Integers\) denotes the discrete category associated to the set \(\Integers\) of integers.

\begin{example} \label{ex:category_of_complexes_in_an_abelian_category_with_quasi-isomorphisms_as_denominators}
The category \(\Com(\mathcal{A})\) of complexes in an abelian category \(\mathcal{A}\) carries the structure of a saturated additive uni-fractionable category, where
\begin{align*}
\Denominators \Com(\mathcal{A}) & = \{f \in \Mor \Com(\mathcal{A}) \mid \text{\(\CohomologyGroup(f)\) is an isomorphism}\}, \\
\SDenominators \Com(\mathcal{A}) & = \{i \in \Denominators \Com(\mathcal{A}) \mid \text{\(i\) is a monomorphism}\}, \\
\TDenominators \Com(\mathcal{A}) & = \{p \in \Denominators \Com(\mathcal{A}) \mid \text{\(p\) is an epimorphism}\}.
\end{align*}
In particular, the derived category \(\DerivedCategory(\mathcal{A})\) is isomorphic to \(\FractionCategory \Com(\mathcal{A})\).
\end{example}
\begin{proof}
First of all, the set \(\{f \in \Mor \Com(\mathcal{A}) \mid \text{\(\CohomologyGroup(f)\) is an isomorphism}\}\) is closed under finite sums since the cohomology functor \(\CohomologyGroup\colon \Com(\mathcal{A}) \map \mathcal{A}^{\DiscreteCategory \Integers}\) is additive. We verify the axioms of a uni-fractionable category.
\begin{itemize}
\item[(2\,of\,3)] Given morphisms \(f, g \in \Mor \Com(\mathcal{A})\) with \(\Target f = \Source g\), we have \(\CohomologyGroup(f) \CohomologyGroup(g) = \CohomologyGroup(f g)\). Hence if two out of the morphisms \(\CohomologyGroup(f)\), \(\CohomologyGroup(g)\), \(\CohomologyGroup(f g)\) are isomorphisms, then so is the third.
\item[(Cat)] For every \(X \in \Ob \Com(\mathcal{A})\), we have \(\CohomologyGroup(1_X) = 1_{\CohomologyGroup(X)}\), so in particular \(\CohomologyGroup(1_X)\) is an isomorphism. Thus the set \(\{f \in \Mor \Com(\mathcal{A}) \mid \text{\(\CohomologyGroup(f)\) is an isomorphism}\}\) is multiplicative. But then also its subsets of monomorphisms resp.\ epimorphisms are multiplicative since monomorphisms compose to monomorphisms resp.\ epimorphisms compose to epimorphisms.
\item[(WU)] We suppose given a monomorphism \(i\colon X \map X'\) and a morphism \(f\colon X \map Y\) in \(\Com(\mathcal{A})\), and we let
\[\begin{tikzpicture}[baseline=(m-2-1.base)]
  \matrix (m) [diagram]{
    X' & Y' \\
    X & Y \\};
  \path[->, font=\scriptsize]
    (m-1-1) edge node[above] {\(f'\)} (m-1-2)
    (m-2-1) edge node[above] {\(f\)} (m-2-2)
            edge node[left] {\(i\)} (m-1-1)
    (m-2-2) edge node[right] {\(i'\)} (m-1-2);
\end{tikzpicture}\]
be a pushout rectangle in \(\Com(\mathcal{A})\). Then we obtain an induced isomorphism \(\Cokernel i \map \Cokernel i'\). A consideration of the long exact cohomology sequence induced by
\[X \morphism[i] X' \morphism[\quo] \Cokernel i\]
shows that \(\CohomologyGroup(i)\) is an isomorphism if and only if \(\CohomologyGroup(\Cokernel i) \isomorphic 0\), and analogously for \(i'\). So if \(\CohomologyGroup(i)\) is an isomorphism, then also \(\CohomologyGroup(i')\) is an isomorphism.

The other assertion follows by duality.
\item[(Fac)] This follows from the fact that every morphism \(f\colon X \map Y\) in \(\Com(\mathcal{A})\) can be factorised into the morphism \(\begin{smallpmatrix} f & \ins \end{smallpmatrix}\colon X \map Y \directsum \Cone X\), where \(\ins\colon X \map \Cone X\) is the insertion of \(X\) into the cone of \(X\), followed by the (split) epimorphism \(\begin{smallpmatrix} 1 \\ 0 \end{smallpmatrix}\colon Y \directsum \Cone X \map Y\), cf.~\cite[III.3.2--3]{gelfand_manin:2003:methods_of_homological_algebra}.
\[\begin{tikzpicture}[baseline=(m-3-1.base)]
  \matrix (m) [diagram=0.9em]{
    & Y \directsum \Cone X & \\
    & & \\
    X & & Y \\};
  \path[->, font=\scriptsize]
    (m-1-2) edge node[right] {\(\begin{smallpmatrix} 1 \\ 0 \end{smallpmatrix}\)} (m-3-3)
    (m-3-1) edge node[above] {\(f\)} (m-3-3)
            edge node[left] {\(\begin{smallpmatrix} f & \ins \end{smallpmatrix}\)} (m-1-2);
\end{tikzpicture}\]
The morphism \(\begin{smallpmatrix} f & \ins \end{smallpmatrix}\) is a monomorphism as \(\ins\) is a monomorphism. Moreover, \(\CohomologyGroup(\begin{smallpmatrix} 1 \\ 0 \end{smallpmatrix})\colon \CohomologyGroup(Y \directsum \Cone X) \map \CohomologyGroup(Y)\) is an isomorphism since \(\CohomologyGroup\) is an additive functor and \(\CohomologyGroup(\Cone X) \isomorphic 0\). Hence if \(\CohomologyGroup(f)\) is an isomorphism, then also \(\CohomologyGroup(\begin{smallpmatrix} f & \ins \end{smallpmatrix})\) is an isomorphism. (\footnote{Alternatively, one can give a factorisation using the cylinder of \(f\) instead of \(Y \directsum \Cone X\), cf.\ for example~\cite[III.3.2--3]{gelfand_manin:2003:methods_of_homological_algebra}.}) \qedhere
\end{itemize}
Altogether, \(\Com(\mathcal{A})\) becomes an additive uni-fractionable category with
\begin{align*}
\Denominators \Com(\mathcal{A}) & = \{f \in \Mor \Com(\mathcal{A}) \mid \text{\(\CohomologyGroup(f)\) is an isomorphism}\}, \\
\SDenominators \Com(\mathcal{A}) & = \{i \in \Denominators \Com(\mathcal{A}) \mid \text{\(i\) is a monomorphism}\}, \\
\TDenominators \Com(\mathcal{A}) & = \{p \in \Denominators \Com(\mathcal{A}) \mid \text{\(p\) is an epimorphism}\}.
\end{align*}

To show that \(\Com(\mathcal{A})\) is saturated, we suppose given morphisms \(f, g, h \in \Mor \Com(\mathcal{A})\) with \(\Target f = \Source g\) and \(\Target g = \Source h\) and such that \(f g\) and \(g h\) are denominators in \(\Com(\mathcal{A})\), that is, \(\CohomologyGroup(f g)\) and \(\CohomologyGroup(g h)\) are isomorphisms in \(\mathcal{A}^{\DiscreteCategory \Integers}\). Then also \(\CohomologyGroup(g)\) is an isomorphism with \(\CohomologyGroup(g)^{- 1} = \CohomologyGroup(f g)^{- 1} \CohomologyGroup(f) = \CohomologyGroup(h) \CohomologyGroup(g h)^{- 1}\) and hence \(g\) is a denominator in \(\Com(\mathcal{A})\). But since \(\Com(\mathcal{A})\) is semi-saturated, this already implies that \(\Com(\mathcal{A})\) is weakly saturated and therefore saturated by proposition~\ref{prop:weakly_saturated_uni-fractionable_category_is_a_saturated_uni-fractionable_category}.
\end{proof}

The preceding example of complexes can be generalised to exact categories in the sense of \eigenname{Quillen}~\cite[\S 2, pp.~99--100]{quillen:1973:higher_algebraic_k-theory_I} as follows in example~\ref{ex:exact_categories_with_enough_formal_cones_with_respect_to_subcategories_closed_under_pure_short_exact_sequences}. A denominator in example~\ref{ex:category_of_complexes_in_an_abelian_category_with_quasi-isomorphisms_as_denominators} is a morphism of complexes such that the induced morphisms on the cohomology level are isomorphisms. This is what one usually calls a \emph{quasi-isomorphism}, and can be characterised as follows: A morphism of complexes with entries in an abelian category is a quasi-isomorphism if and only if its cone is acyclic~\cite[cor.~1.5.4]{weibel:1997:an_introduction_to_homological_algebra}. Since we have no cohomology functor for exact categories, we have to clarify first what we want to understand by a (formal) cone and a quasi-isomorphism in an exact category. We will develop these notions and some facts in appendix~\ref{sec:formal_cones_in_exact_categories}, cf.\ in particular definitions~\ref{def:formal_cone_with_respect_to_a_subcategory_of_an_exact_category} resp.~\ref{def:exact_category_having_enough_formal_cones} resp.~\ref{def:quasi-isomorphisms_in_an_exact_category_having_enough_formal_cones} resp.~\ref{def:subcategories_closed_under_pure_short_exact_sequences} for the definitions of formal cones resp.\ having enough formal cones resp.\ quasi-isomorphisms resp.\ closedness under pure short sequences.

\begin{example} \label{ex:exact_categories_with_enough_formal_cones_with_respect_to_subcategories_closed_under_pure_short_exact_sequences} \
We suppose given an exact category \(\mathcal{E}\) and a non-empty full subcategory \(\mathcal{U}\) of \(\mathcal{E}\) that is closed under pure short exact sequences and such that \(\mathcal{E}\) has enough formal \(\mathcal{U}\)-cones. Then \(\mathcal{E}\) carries the structure of a uni-fractionable category with
\begin{align*}
\Denominators \mathcal{E} & = \{f \in \Mor \mathcal{E} \mid \text{\(f\) is a \(\mathcal{U}\)-quasi-isomorphism}\}, \\
\SDenominators \mathcal{E} & = \{i \in \Denominators \mathcal{E} \mid \text{\(i\) is a pure monomorphism}\} \\
& = \{i \in \Mor \mathcal{E} \mid \text{\(i\) is a pure monomorphism with \(\Cokernel i \in \Ob \mathcal{U}\)}\}, \\
\TDenominators \mathcal{E} & = \{p \in \Denominators \mathcal{E} \mid \text{\(p\) is a pure epimorphism}\} \\
& = \{p \in \Mor \mathcal{E} \mid \text{\(p\) is a pure epimorphism with \(\Kernel p \in \Ob \mathcal{U}\)}\}.
\end{align*}
If moreover \(\mathcal{U}\) is closed under summands, then \(\mathcal{E}\) is saturated.
\end{example}
\begin{proof}
The set of quasi-isomorphisms in \(\mathcal{E}\) is closed under finite sums by remark~\ref{rem:for_subcategories_closed_under_pure_short_exact_sequences_the_set_of_quasi-isomorphisms_is_closed_under_sums} and semi-saturated by proposition~\ref{prop:set_of_quasi-isomorphisms_in_an_exact_category_is_semi-saturated}, hence in particular multiplicative. The set of pure monomorphisms that are quasi-isomorphisms is also multiplicative since the set of pure monomorphisms in an exact category is multiplicative~\cite[def.~2.1]{buehler:2010:exact_categories}. Dually, the set of pure epimorphisms that are quasi-isomorphisms is multiplicative. Axiom (WU) is fulfilled by corollary~\ref{cor:quasi-isomorphisms_that_are_pure_mono_or_pure_epi_are_stable_under_bicartesian_squares} since pure monomorphisms in an exact category are stable under pushouts and pure epimorphisms are stable under pullbacks~\cite[def.~2.1]{buehler:2010:exact_categories}. Finally, every quasi-isomorphism factors into a pure monomorphism that is a quasi-isomorphism and a pure epimorphism that is a quasi-isomorphism by corollary~\ref{cor:equivalent_descriptions_for_quasi-isomorphisms_in_exact_categories_having_enough_formal_cones}. Altogether, \(\mathcal{E}\) becomes an additive uni-fractionable category with
\begin{align*}
\Denominators \mathcal{E} & = \{f \in \Mor \mathcal{E} \mid \text{\(f\) is a \(\mathcal{U}\)-quasi-isomorphism}\}, \\
\SDenominators \mathcal{E} & = \{i \in \Denominators \mathcal{E} \mid \text{\(i\) is a pure monomorphism}\}, \\
\TDenominators \mathcal{E} & = \{p \in \Denominators \mathcal{E} \mid \text{\(p\) is a pure epimorphism}\}.
\end{align*}
Moreover, \(\SDenominators \mathcal{E} = \{i \in \Mor \mathcal{E} \mid \text{\(i\) is a pure monomorphism with \(\Cokernel i \in \Ob \mathcal{U}\)}\}\) and \(\TDenominators \mathcal{E} = \{p \in \Mor \mathcal{E} \mid \text{\(p\) is a pure epimorphism with \(\Kernel p \in \Ob \mathcal{U}\)}\}\) by proposition~\ref{prop:criterion_for_pure_monos_resp_pure_epis_being_quasi-isomorphism}. The assertion on the saturatedness of \(\mathcal{E}\) follows from proposition~\ref{prop:for_closed_under_taking_summands_subcategories_the_set_of_quasi-isomorphisms_in_an_exact_category_is_weakly_saturated} and proposition~\ref{prop:weakly_saturated_uni-fractionable_category_is_a_saturated_uni-fractionable_category}.
\end{proof}

Recall that an additive category is said to be \newnotion{idempotent splitting} if every morphism \(e\) with \(e^2 = e\) is split. (\eigenname{B{\"u}hler} uses the notion ``idempotent complete'', see~\cite[def.~6.1]{buehler:2010:exact_categories}.) Moreover, recall that the category of complexes with entries in an exact category becomes an exact category with degreewise pure short exact sequences, see~\cite[lem.~9.1]{buehler:2010:exact_categories}.

As an application of example~\ref{ex:exact_categories_with_enough_formal_cones_with_respect_to_subcategories_closed_under_pure_short_exact_sequences}, we obtain in particular a \(3\)-arrow calculus for the derived category of an idempotent splitting exact category without passing to the homotopy category in advance, cf.~\cite{neeman:1990:the_derived_category_of_an_exact_category},~\cite[sec.~10]{buehler:2010:exact_categories}.

\begin{example} \label{ex:category_of_complexes_in_an_idempotent_splitting_exact_category_with_quasi-isomorphisms_as_denominators}
The category \(\Com(\mathcal{E})\) of complexes in an idempotent splitting exact category \(\mathcal{E}\) carries the structure of a saturated additive uni-fractionable category, where
\begin{align*}
\Denominators \Com(\mathcal{E}) & = \{f \in \Mor \Com(\mathcal{E}) \mid \text{\(f\) is a quasi-isomorphism}\}, \\
\SDenominators \Com(\mathcal{E}) & = \{i \in \Denominators \Com(\mathcal{E}) \mid \text{\(i\) is a pure monomorphism}\} \\
& = \{i \in \Mor \Com(\mathcal{E}) \mid \text{\(i\) is a pure monomorphism with \(\Cokernel i\) purely acyclic}\}, \\
\TDenominators \Com(\mathcal{E}) & = \{p \in \Denominators \Com(\mathcal{E}) \mid \text{\(p\) is a pure epimorphism}\} \\
& = \{p \in \Mor \Com(\mathcal{E}) \mid \text{\(p\) is a pure epimorphism with \(\Kernel p\) purely acyclic}\}.
\end{align*}
In particular, the derived category \(\DerivedCategory(\mathcal{E})\) is isomorphic to \(\FractionCategory \Com(\mathcal{E})\).
\end{example}
\begin{proof}
By~\cite[rem.~10.17]{buehler:2010:exact_categories} and corollary~\ref{cor:equivalent_descriptions_for_quasi-isomorphisms_in_exact_categories_having_enough_formal_cones}, a quasi-isomorphism, in the sense of~\cite[def.~10.16]{buehler:2010:exact_categories}, is precisely a quasi-isomorphism with respect to the full subcategory of pure acyclic complexes~\cite[def.~4(2)]{buehler_kuenzer:2009:some_elementary_considerations_in_exact_categories} in the sense of definition~\ref{def:quasi-isomorphisms_in_an_exact_category_having_enough_formal_cones}. Moreover, the full subcategory of pure acyclic complexes is closed under pure short exact sequences~\cite[cor.~29]{buehler_kuenzer:2009:some_elementary_considerations_in_exact_categories} and under summands~\cite[lem.~10.7]{buehler:2010:exact_categories}, and the exact category \(\Com(\mathcal{E})\) has enough formal cones with respect to this full subcategory~\cite[def.~9.2, rem.~9.9, prop.~10.9]{buehler:2010:exact_categories}. Thus the assertion follows from example~\ref{ex:exact_categories_with_enough_formal_cones_with_respect_to_subcategories_closed_under_pure_short_exact_sequences}.
\end{proof}

\subsection*{Classical examples} \label{ssec:classical_examples}

We finish this article by considering some classical examples, which yield in fact a 2-arrow calculus, but nonetheless fit in our framework. Note that example~\ref{ex:thick_subcategories_of_abelian_and_triangulated_categories_yield_uni-fractionable_categories}\ref{ex:thick_subcategories_of_abelian_and_triangulated_categories_yield_uni-fractionable_categories:abelian_category_with_factorised_structure} implicitly occurs in \eigenname{Grothendieck}'s T{\^o}hoku article~\cite[sec.~1.11]{grothendieck:1957:sur_quelques_points_d_algebre_homologique}. This example differs from the others since here the S-denominators are epimorphisms and the T{\nbd}denominators are monomorphisms, while in all our other examples the S-denominators are ``mono-like'' (certain cofibrations resp.\ monomorphisms resp.\ pure monomorphisms) and the T-denominators are ``epi-like'' (certain fibrations resp.\ epimorphisms resp.\ pure epimorphisms).

Recall that a \newnotion{thick subcategory} of an abelian category \(\mathcal{A}\) is a non-empty full (abelian) subcategory \(\mathcal{U}\) that is closed under extensions, subobjects and quotient objects. Recall that a \newnotion{thick subcategory} of a Verdier triangulated category \(\mathcal{V}\) is a (non-empty) full triangulated subcategory \(\mathcal{U}\) that is closed under taking summands.

\begin{example} \label{ex:thick_subcategories_of_abelian_and_triangulated_categories_yield_uni-fractionable_categories} \
\begin{enumerate}
\item \label{ex:thick_subcategories_of_abelian_and_triangulated_categories_yield_uni-fractionable_categories:abelian_category_with_obvious_structure} We suppose given an abelian category \(\mathcal{A}\) and a thick subcategory \(\mathcal{U}\) in \(\mathcal{A}\). Then \(\mathcal{A}\) carries the structure of an additive uni-fractionable category, where
\[\Denominators \mathcal{A} = \SDenominators \mathcal{A} = \TDenominators \mathcal{A} = \{f \in \Mor \mathcal{A} \mid \text{\(\Kernel f\) and \(\Cokernel f\) are in \(\mathcal{U}\)}\}.\]
In particular, the Serre quotient of \(\mathcal{A}\) by \(\mathcal{U}\) is isomorphic to \(\FractionCategory \mathcal{A}\).
\item \label{ex:thick_subcategories_of_abelian_and_triangulated_categories_yield_uni-fractionable_categories:abelian_category_with_factorised_structure} We suppose given an abelian category \(\mathcal{A}\) and a thick subcategory \(\mathcal{U}\) in \(\mathcal{A}\). Then \(\mathcal{A}\) carries the structure of an additive uni-fractionable category, where
\begin{align*}
& \Denominators \mathcal{A} = \{f \in \Mor \mathcal{A} \mid \text{\(\Kernel f\) and \(\Cokernel f\) are in \(\mathcal{U}\)}\}, \\
& \SDenominators \mathcal{A} = \{s \in \Denominators \mathcal{A} \mid \text{\(s\) is an epimorphism}\}, \\
& \TDenominators \mathcal{A} = \{t \in \Denominators \mathcal{A} \mid \text{\(t\) is a monomorphism}\}.
\end{align*}
In particular, the Serre quotient of \(\mathcal{A}\) by \(\mathcal{U}\) is isomorphic to \(\FractionCategory \mathcal{A}\).
\item \label{ex:thick_subcategories_of_abelian_and_triangulated_categories_yield_uni-fractionable_categories:verdier_triangulated_category} We suppose given a Verdier triangulated category \(\mathcal{V}\) and a thick subcategory \(\mathcal{U}\) in \(\mathcal{V}\). Then \(\mathcal{V}\) carries the structure of an additive uni-fractionable category, where
\[\Denominators \mathcal{V} = \SDenominators \mathcal{V} = \TDenominators \mathcal{V} = \{f \in \Mor \mathcal{V} \mid \text{a cone of \(f\) is in \(\mathcal{U}\)}\}.\]
In particular, the Verdier quotient of \(\mathcal{V}\) by \(\mathcal{U}\) is isomorphic to \(\FractionCategory \mathcal{V}\).
\end{enumerate}
\end{example}
\begin{proof} \
\begin{enumerate}
\item We set \(D := \{f \in \Mor \mathcal{A} \mid \text{\(\Kernel f\) and \(\Cokernel f\) are in \(\mathcal{U}\)}\}\) and have to verify the axioms (Cat), (2\,of\,3) and (WU).
\begin{itemize}
\item[(Cat)] We suppose given morphisms \(f\colon X \map Y\) and \(g\colon Y \map Z\) in \(\mathcal{A}\) with \(f, g \in D\), so that \(\Kernel f\), \(\Cokernel f\), \(\Kernel g\), \(\Cokernel g\) are in \(\mathcal{U}\). By the circumference lemma~\cite[Lem.~132]{kuenzer:2010:homologische_algebra}, we have an exact sequence
\[0 \morphism \Kernel f \morphism \Kernel(f g) \morphism \Kernel g \morphism \Cokernel f \morphism \Cokernel(f g) \morphism \Cokernel g \morphism 0.\]
Since \(\Kernel f\) and \(\Kernel g\) are in \(\mathcal{U}\), it follows that \(\Kernel(f g)\) is in \(\mathcal{U}\), and since \(\Cokernel f\) and \(\Cokernel g\) are in \(\mathcal{U}\), it follows that \(\Cokernel(f g)\) is in \(\mathcal{U}\). Thus we have \(f g \in D\).

Moreover, \(0 \in \Ob \mathcal{U}\) and therefore \(1_X \in D\) for all \(X \in \Ob \mathcal{A}\).
\item[(2\,of\,3)] We suppose given morphisms \(f\colon X \map Y\) and \(g\colon Y \map Z\) in \(\mathcal{A}\) with \(f, f g \in D\), so that the objects \(\Kernel f\), \(\Cokernel f\), \(\Kernel(f g)\), \(\Cokernel(f g)\) are in \(\mathcal{U}\). The circumference lemma~\cite[Lem.~132]{kuenzer:2010:homologische_algebra} implies that \(\Kernel g\) is in \(\mathcal{U}\) since \(\Kernel(f g)\) and \(\Cokernel f\) are in \(\mathcal{U}\), and that \(\Cokernel g\) is in \(\mathcal{U}\) since \(\Cokernel(f g)\) is in \(\mathcal{U}\). Thus we have \(g \in D\).

The other case follows by duality.
\item[(WU)] We suppose given morphisms \(d\colon X \map X'\) and \(f\colon X \map Y\) in \(\mathcal{A}\) with \(d \in D\), and we let
\[\begin{tikzpicture}[baseline=(m-2-1.base)]
  \matrix (m) [diagram without objects]{
    X' & Y' \\
    X & X' \\};
  \path[->, font=\scriptsize]
    (m-1-1) edge node[above] {\(f'\)} (m-1-2)
    (m-2-1) edge node[above] {\(f\)} (m-2-2)
            edge node[left] {\(d\)} (m-1-1)
    (m-2-2) edge node[right] {\(d'\)} (m-1-2);
\end{tikzpicture}\]
be a pushout rectangle in \(\mathcal{A}\). Then the induced morphism \(\Cokernel d \map \Cokernel d'\) is an isomorphism and the induced morphism \(\Kernel d \map \Kernel d'\) is an epimorphism. Thus since \(\Kernel d\) and \(\Cokernel d\) are in \(\mathcal{U}\), it follows that \(\Kernel d'\) and \(\Cokernel d'\) are in \(\mathcal{U}\), that is, \(d' \in D\).

The other property follows by duality.
\end{itemize}
Altogether, there is a structure of a uni-fractionable category on \(\mathcal{A}\) with \(\Denominators \mathcal{A} = \SDenominators \mathcal{A} = \TDenominators \mathcal{A} = D\). Moreover, \(\Denominators \mathcal{A}\) is closed under finite sums since \(\mathcal{U}\) is an additive subcategory, and hence \(\mathcal{A}\) is an additive uni-fractionable category.
\item The axioms (Cat), (2\,of\,3) and (WU) as well as additivity follow from~\ref{ex:thick_subcategories_of_abelian_and_triangulated_categories_yield_uni-fractionable_categories:abelian_category_with_obvious_structure}, taking into account that epimorphisms are stable under composition and pushouts, and dually. Moreover, (Fac) holds since every morphism in an abelian category factorises into an epimorphism with the same kernel followed by a monomorphism with the same cokernel.
\item We set \(D := \{f \in \Mor \mathcal{V} \mid \text{a cone of \(f\) is in \(\mathcal{U}\)}\}\) and have to verify the axioms of a uni-fractionable category. In the following, given a morphism \(f \in \Mor \mathcal{V}\), we denote by \(C_f\) a chosen cone of \(f\). Since every cone of \(f\) is isomorphic to \(C_f\)~\cite[IV.1.4 b)]{gelfand_manin:2003:methods_of_homological_algebra}, we have \(f \in D\) if and only if \(C_f\) is in \(\mathcal{U}\). The shift functor in \(\mathcal{V}\) is denoted by \(\ShiftFunctor\).
\begin{itemize}
\item[(Cat)] We suppose given morphisms \(f\colon X \map Y\) and \(g\colon Y \map Z\) with \(f, g \in D\), so that \(C_f\) and \(C_g\) are in \(\mathcal{U}\). By the octahedral axiom, we get a distinguished triangle
\[\dots \morphism[\ShiftFunctor^{- 1} v] \ShiftFunctor^{- 1} C_g \morphism[\ShiftFunctor^{- 1} w] C_f \morphism[u] C_{f g} \morphism[v] C_g \morphism[w] \ShiftFunctor C_f \morphism[\ShiftFunctor u] \dots,\]
and in particular, \(C_{f g}\) is a cone of \(\ShiftFunctor^{- 1} w\). Since \(C_g\) is in \(\mathcal{U}\), it follows that \(\ShiftFunctor^{- 1} C_g\) is in \(\mathcal{U}\). But then \(\ShiftFunctor^{- 1} w\) is a morphism in \(\mathcal{U}\) and thus \(C_{f g}\) is an object in \(\mathcal{U}\).

Moreover, \(0 \in \Ob \mathcal{U}\) implies \(1_X \in D\) for all \(X \in \Ob \mathcal{V}\).
\item[(2\,of\,3)] We suppose given morphisms \(f\colon X \map Y\) and \(g\colon Y \map Z\) in \(\mathcal{V}\), and we use the distinguished triangle obtained by the octahedral axiom from above. If \(f, f g \in D\), then \(C_f\) and \(C_{f g}\) are in \(\mathcal{U}\), hence \(u\) is a morphism in \(\mathcal{U}\) and therefore \(C_g\) is in \(\mathcal{U}\), that is, \(g \in D\). Analogously, \(f g, g \in D\) implies that \(C_{f g}\) and \(C_g\) are in \(\mathcal{U}\), hence \(v\) is a morphism in \(\mathcal{U}\) and \(\ShiftFunctor C_f\) is an object in \(\mathcal{U}\), and therefore \(C_f\) is in \(\mathcal{U}\), that is, \(f \in D\).
\item[(WU)] We suppose given morphisms \(d\colon X \map X'\) and \(f\colon X \map Y\) in \(\mathcal{V}\) with \(d \in D\), and we let
\[\begin{tikzpicture}[baseline=(m-2-1.base)]
  \matrix (m) [diagram without objects]{
    X' & Y' \\
    X & Y \\};
  \path[->, font=\scriptsize]
    (m-1-1) edge node[above] {\(f'\)} (m-1-2)
    (m-2-1) edge node[above] {\(f\)} (m-2-2)
            edge node[left] {\(d\)} (m-1-1)
    (m-2-2) edge node[right] {\(d'\)} (m-1-2);
\end{tikzpicture}\]
be a dweak square in \(\mathcal{V}\), that is, a quadrangle whose diagonal sequence fits in a distinguished triangle, and hence in particular a weak pushout. Then \(C_d\) is a cone of \(d'\), cf.~\cite[lem.~1.4.4]{neeman:2001:triangulated_categories}. Hence \(d' \in D\) as \(C_d\) is in \(\mathcal{U}\).

The other property follows by duality.
\end{itemize}
Altogether, there is a structure of a uni-fractionable category on \(\mathcal{V}\) with \(\Denominators \mathcal{V} = \SDenominators \mathcal{V} = \TDenominators \mathcal{V} = D\). Moreover, additivity of \(\mathcal{V}\) follows from the additivity of \(\mathcal{U}\). \qedhere
\end{enumerate}
\end{proof}

\appendix

\section{Formal cones in exact categories} \label{sec:formal_cones_in_exact_categories}

In this appendix, we develop a theory about ``formal cones'' and ``quasi-isomorphisms'' in an exact category relative to a suitable subcategory. This will be used to generalise example~\ref{ex:category_of_complexes_in_an_abelian_category_with_quasi-isomorphisms_as_denominators}, where we have shown that the category of complexes in an abelian category carries a uni-fractionable category structure in such a way that the fraction category becomes the derived category, to the case of idempotent splitting exact categories, see example~\ref{ex:category_of_complexes_in_an_idempotent_splitting_exact_category_with_quasi-isomorphisms_as_denominators}.

We consider an exact category \(\mathcal{E}\) in the sense of \eigenname{Quillen}~\cite[\S 2, pp.~99--100]{quillen:1973:higher_algebraic_k-theory_I}, cf.\ also~\cite[app.~A]{keller:1990:chain_complexes_and_stable_categories},~\cite[def.~2.1]{buehler:2010:exact_categories}. The distinguished short exact sequences in \(\mathcal{E}\) will be called \newnotion{pure short exact sequences}. Likewise, the monomorphisms occurring in a pure short exact sequence are called \newnotion{pure monomorphisms}, and the epimorphisms occurring in a pure short exact sequence are called \newnotion{pure epimorphisms}.

During this appendix, we suppose given an exact category \(\mathcal{E}\) and a non-empty full subcategory \(\mathcal{U}\) of \(\mathcal{E}\). From remark~\ref{rem:for_subcategories_closed_under_pure_short_exact_sequences_the_set_of_quasi-isomorphisms_is_closed_under_sums} on, we suppose that \(\mathcal{U}\) is closed under pure short exact sequences, see definition~\ref{def:subcategories_closed_under_pure_short_exact_sequences}, and that \(\mathcal{E}\) has enough formal \(\mathcal{U}\)-cones, see definition~\ref{def:exact_category_having_enough_formal_cones}.

\begin{definition}[formal cone] \label{def:formal_cone_with_respect_to_a_subcategory_of_an_exact_category} \
\begin{enumerate}
\item \label{def:formal_cone_with_respect_to_a_subcategory_of_an_exact_category:objects} We suppose given an object \(X\) in \(\mathcal{E}\). A \newnotion{formal cone} with respect to \(\mathcal{U}\) (or \newnotion{formal \(\mathcal{U}\)-cone} or just \newnotion{formal cone}) of \(X\) consists of an object \(C\) in \(\mathcal{U}\) together with a pure monomorphism \(i\colon X \map C\). By abuse of notation, we denote the formal cone as well as its underlying object by \(C\). The pure monomorphism \(i\) is called the \newnotion{insertion} in \(C\). Given a formal cone \(C\) of \(X\) with insertion \(i\), we write \(\ins = \ins^C := i\).
\item \label{def:formal_cone_with_respect_to_a_subcategory_of_an_exact_category:morphisms} We suppose given a morphism \(f\colon X \map Y\) in \(\mathcal{E}\). Given a formal \(\mathcal{U}\)-cone \(C_X\) of \(X\), a \newnotion{formal cone} with respect to \(\mathcal{U}\) (or \newnotion{formal \(\mathcal{U}\)-cone} or just \newnotion{formal cone}) of \(f\) \newnotion{corresponding to \(C_X\)} consists of an object \(C\) in \(\mathcal{U}\) together with a pure monomorphism \(i\colon Y \map C\) such that there exists a pushout rectangle
\[\begin{tikzpicture}[baseline=(m-2-1.base)]
  \matrix (m) [diagram]{
    C_X & C \\
    X & Y \\};
  \path[->, font=\scriptsize]
    (m-1-1) edge node[above] {\(f'\)} (m-1-2)
    (m-2-1) edge node[above] {\(f\)} (m-2-2)
            edge node[left] {\(\ins^{C_X}\)} (m-1-1)
    (m-2-2) edge node[right] {\(i\)} (m-1-2);
\end{tikzpicture}\]
in \(\mathcal{E}\). By abuse of notation, we denote the formal cone as well as its underlying object by \(C\). The pure monomorphism \(i\colon Y \map C\) is called the \newnotion{insertion} in \(C\). Given a formal cone \(C\) of \(f\) corresponding to \(C_X\) with insertion \(i\), we write \(\ins = \ins^C := i\).

A \newnotion{formal cone} with respect to \(\mathcal{U}\) (or \newnotion{formal \(\mathcal{U}\)-cone} or just \newnotion{formal cone}) of \(f\) is a formal \(\mathcal{U}\)-cone of \(f\) corresponding to some formal \(\mathcal{U}\)-cone of \(X\).
\end{enumerate}
\end{definition}

The cone of a complex resp.\ of a morphism of complexes with entries in an additive category as defined for example in~\cite[def.~9.2]{buehler:2010:exact_categories} is a formal cone with respect to the full subcategory of acyclic complexes (or even of split acyclic complexes).

\begin{remark} \label{rem:for_additive_subcategories_the_sum_of_formal_cones_is_a_formal_cone}
We suppose that \(\mathcal{U}\) is an additive subcategory of \(\mathcal{E}\). 
\begin{enumerate}
\item \label{rem:for_additive_subcategories_the_sum_of_formal_cones_is_a_formal_cone:objects} We suppose given objects \(X_1\) and \(X_2\) in \(\mathcal{E}\), a formal \(\mathcal{U}\)-cone \(C_1\) of \(X_1\) and a formal \(\mathcal{U}\)-cone \(C_2\) of \(X_2\). Then \(C_1 \directsum C_2\) is a formal \(\mathcal{U}\)-cone of \(X_1 \directsum X_2\) with \(\ins^{C_1 \directsum C_2} = \ins^{C_1} \directsum \ins^{C_2}\).
\item \label{rem:for_additive_subcategories_the_sum_of_formal_cones_is_a_formal_cone:morphisms} We suppose given morphisms \(f_1\colon X_1 \map Y_1\) and \(f_2\colon X_2 \map Y_2\) in \(\mathcal{E}\), a formal \(\mathcal{U}\)-cone \(C_1\) of \(f_1\) and a formal \(\mathcal{U}\)-cone \(C_2\) of \(f_2\). Then \(C_1 \directsum C_2\) is a formal \(\mathcal{U}\)-cone of \(f_1 \directsum f_2\) with \(\ins^{C_1 \directsum C_2} = \ins^{C_1} \directsum \ins^{C_2}\).
\end{enumerate}
\end{remark}

\begin{definition}[having enough formal cones] \label{def:exact_category_having_enough_formal_cones}
The exact category \(\mathcal{E}\) is said to \newnotion{have enough formal cones} with respect to \(\mathcal{U}\) (or to \newnotion{have enough formal \(\mathcal{U}\)-cones} or just to \newnotion{have enough formal cones}) if there exists a formal \(\mathcal{U}\)-cone of every object in \(\mathcal{E}\).
\end{definition}

\begin{definition}[quasi-isomorphism] \label{def:quasi-isomorphisms_in_an_exact_category_having_enough_formal_cones}
We suppose that \(\mathcal{E}\) has enough formal \(\mathcal{U}\)-cones. A \newnotion{quasi-isomorphism} with respect to \(\mathcal{U}\) (or \newnotion{\(\mathcal{U}\)-quasi-isomorphism} or just \newnotion{quasi-isomorphism}) is a morphism \(f\) in \(\mathcal{E}\) such that there exists a formal cone of \(f\) that is in \(\mathcal{U}\).
\end{definition}

\begin{definition}[closed under pure short exact sequences] \label{def:subcategories_closed_under_pure_short_exact_sequences} \
The full subcategory \(\mathcal{U}\) of \(\mathcal{E}\) is said to be \newnotion{closed under pure short exact sequences} in \(\mathcal{E}\) if it fulfills the following axiom.
\begin{itemize}
\item[(Seq)] Given a pure short exact sequence
\[X' \morphism X \morphism X''\]
in \(\mathcal{E}\) such that two out of the objects \(X'\), \(X\), \(X''\) are in \(\mathcal{U}\), then so is the third.
\end{itemize}
\end{definition}

\begin{remark} \label{rem:subcategories_closed_under_pure_short_exact_sequences_contain_a_null_object_and_are_closed_under_isomorphisms}
If \(\mathcal{U}\) is closed under pure short exact sequences, then \(\mathcal{U}\) is closed under isomorphisms and an additive subcategory of \(\mathcal{E}\).
\end{remark}
\begin{proof}
Since \(\mathcal{U}\) is supposed to be non-empty, there exists an object \(X\) in \(\mathcal{U}\). But then
\[X \morphism[1_X] X \morphism[0] 0\]
is a pure short exact sequence~\cite[def.~2.1]{buehler:2010:exact_categories} and hence \(0\) is in \(\mathcal{U}\). Given objects \(X_1\) and \(X_2\) in \(\mathcal{U}\), we have the pure short exact sequence
\[X_1 \morphism[\begin{smallpmatrix} 1 & 0 \end{smallpmatrix}] X_1 \directsum X_2 \morphism[\begin{smallpmatrix} 0 \\ 1 \end{smallpmatrix}] X_2\]
by~\cite[lem.~2.7]{buehler:2010:exact_categories} and therefore \(X_1 \directsum X_2\) is in \(\mathcal{U}\). Thus \(\mathcal{U}\) is an additive subcategory of \(\mathcal{E}\).

Finally, given an object \(X\) in \(\mathcal{U}\) and an isomorphism \(f\colon X \map Y\) in \(\mathcal{E}\), we have the pure short exact sequence
\[X \morphism[f] Y \morphism[0] 0\]
and hence \(Y\) is in \(\mathcal{U}\).
\end{proof}

From now on, we suppose that \(\mathcal{U}\) is closed under pure short exact sequences and that \(\mathcal{E}\) has enough formal \(\mathcal{U}\)-cones.

\begin{remark} \label{rem:for_subcategories_closed_under_pure_short_exact_sequences_the_set_of_quasi-isomorphisms_is_closed_under_sums}
The set of \(\mathcal{U}\)-quasi-isomorphisms in \(\mathcal{E}\) is closed under finite sums.
\end{remark}
\begin{proof}
This follows from remark~\ref{rem:subcategories_closed_under_pure_short_exact_sequences_contain_a_null_object_and_are_closed_under_isomorphisms} and remark~\ref{rem:for_additive_subcategories_the_sum_of_formal_cones_is_a_formal_cone}\ref{rem:for_additive_subcategories_the_sum_of_formal_cones_is_a_formal_cone:morphisms}.
\end{proof}

\begin{proposition} \label{prop:criterion_for_pure_monos_resp_pure_epis_being_quasi-isomorphism} \
\begin{enumerate}
\item \label{prop:criterion_for_pure_monos_resp_pure_epis_being_quasi-isomorphism:pure_mono} A pure monomorphism \(i\) in \(\mathcal{E}\) is a \(\mathcal{U}\)-quasi-isomorphism if and only if \(\Cokernel i\) is in \(\mathcal{U}\).
\item \label{prop:criterion_for_pure_monos_resp_pure_epis_being_quasi-isomorphism:pure_epi} A pure epimorphism \(p\) in \(\mathcal{E}\) is a \(\mathcal{U}\)-quasi-isomorphism if and only if \(\Kernel p\) is in \(\mathcal{U}\).
\end{enumerate}
\end{proposition}
\begin{proof}
We let \(f\colon X \map Y\) be a morphism in \(\mathcal{E}\) and we let \(C_X\) be a formal cone of \(X\). Moreover, we let \(C_f\) be a formal cone of \(f\) corresponding to \(C_X\), that is, we suppose that there exists a pushout rectangle
\[\begin{tikzpicture}[baseline=(m-2-1.base)]
  \matrix (m) [diagram]{
    C_X & C_f \\
    X & Y \\};
  \path[->, font=\scriptsize]
    (m-1-1) edge node[above] {\(f'\)} (m-1-2)
    (m-2-1) edge node[above] {\(f\)} (m-2-2)
            edge node[left] {\(\ins^{C_X}\)} (m-1-1)
    (m-2-2) edge node[right] {\(\ins^{C_f}\)} (m-1-2);
\end{tikzpicture}\]
in \(\mathcal{E}\). By~\cite[prop.~2.12]{buehler:2010:exact_categories}, this rectangle is also a pullback.

Let us first suppose that \(f\) is a pure monomorphism. Then \(f'\) is also a pure monomorphism by~\cite[def.~2.1]{buehler:2010:exact_categories}, and since \(\mathcal{U}\) is closed under pure short exact sequences, the formal cone \(C_f\) is in \(\mathcal{U}\) if and only if \(\Cokernel f'\) is in \(\mathcal{U}\). But by~\cite[prop.~2.12]{buehler:2010:exact_categories}, we have \(\Cokernel f \isomorphic \Cokernel f'\), so \(C_f\) is in \(\mathcal{U}\) if and only if \(\Cokernel f\) is in \(\mathcal{U}\). Since \(C_X\) and \(C_f\) were chosen arbitrarily, this means that \(f\) is a quasi-isomorphism if and only if \(\Cokernel f\) is in \(\mathcal{U}\).

Next, let us suppose that \(f\) is a pure epimorphism. Then \(f'\) is also a pure epimorphism by~\cite[dual of prop.~2.15]{buehler:2010:exact_categories}, and since \(\mathcal{U}\) is closed under pure short exact sequences, the formal cone \(C_f\) is in \(\mathcal{U}\) if and only if \(\Kernel f'\) is in \(\mathcal{U}\). By~\cite[dual of prop.~2.12]{buehler:2010:exact_categories}, we have \(\Kernel f \isomorphic \Kernel f'\), so \(C_f\) is in \(\mathcal{U}\) if and only if \(\Kernel f\) is in \(\mathcal{U}\). Since \(C_X\) and \(C_f\) were chosen arbitrarily, this means that \(f\) is a quasi-isomorphism if and only if \(\Kernel f\) is in \(\mathcal{U}\).
\end{proof}

\begin{corollary} \label{cor:quasi-isomorphisms_that_are_pure_mono_or_pure_epi_are_stable_under_bicartesian_squares} \
\begin{enumerate}
\item Given a pushout rectangle
\[\begin{tikzpicture}[baseline=(m-2-1.base)]
  \matrix (m) [diagram without objects]{
    X' & Y' \\
    X & Y \\};
  \path[->, font=\scriptsize]
    (m-1-1) edge node[above] {\(f'\)} (m-1-2)
    (m-2-1) edge node[above] {\(f\)} (m-2-2)
            edge node[left] {\(i\)} (m-1-1)
    (m-2-2) edge node[right] {\(i'\)} (m-1-2);
\end{tikzpicture}\]
in \(\mathcal{E}\) with pure monomorphism \(i\), then \(i\) is a \(\mathcal{U}\)-quasi-isomorphism if and only if \(i'\) is a \(\mathcal{U}\)-quasi-isomorphism.
\item Given a pullback rectangle
\[\begin{tikzpicture}[baseline=(m-2-1.base)]
  \matrix (m) [diagram without objects]{
    X' & Y' \\
    X & Y \\};
  \path[->, font=\scriptsize]
    (m-1-1) edge node[above] {\(f'\)} (m-1-2)
            edge node[left] {\(p'\)} (m-2-1)
    (m-1-2) edge node[right] {\(p\)} (m-2-2)
    (m-2-1) edge node[above] {\(f\)} (m-2-2);
\end{tikzpicture}\]
in \(\mathcal{E}\) with pure epimorphism \(p\), then \(p\) is a \(\mathcal{U}\)-quasi-isomorphism if and only if \(p'\) is a \(\mathcal{U}\)-quasi-isomorphism.
\end{enumerate}
\end{corollary}
\begin{proof}
This follows by proposition~\ref{prop:criterion_for_pure_monos_resp_pure_epis_being_quasi-isomorphism} since these pushouts induce isomorphisms on the cokernels of the pure monomorphisms and these pullbacks induce isomorphisms on the kernels of the pure epimorphisms~\cite[prop.~2.12 and its dual]{buehler:2010:exact_categories}.
\end{proof}

\begin{proposition} \label{prop:formal_cone_factorisation}
We let \(f\colon X \map Y\) be a morphism in \(\mathcal{E}\). For every formal \(\mathcal{U}\)-cone \(C_X\) of \(X\), we have the factorisation \(f = \begin{smallpmatrix} f & \ins \end{smallpmatrix} \begin{smallpmatrix} 1 \\ 0 \end{smallpmatrix}\), where \(\begin{smallpmatrix} f & \ins \end{smallpmatrix}\colon X \map Y \directsum C_X\) is a pure monomorphism such that \(\Cokernel \begin{smallpmatrix} f & \ins \end{smallpmatrix}\) carries the structure of a formal \(\mathcal{U}\)-cone of \(f\) corresponding to \(C_X\) and where \(\begin{smallpmatrix} 1 \\ 0 \end{smallpmatrix}\colon Y \directsum C_X \map Y\) is a split pure epimorphism with \(\Kernel \begin{smallpmatrix} 1 \\ 0 \end{smallpmatrix} \isomorphic C_X\). In particular, \(\begin{smallpmatrix} 1 \\ 0 \end{smallpmatrix}\) is a \(\mathcal{U}\)-quasi-isomorphism.
\[\begin{tikzpicture}[baseline=(m-3-1.base)]
  \matrix (m) [diagram=0.9em]{
    & Y \directsum C_X & \\
    & & \\
    X & & Y \\};
  \path[->, font=\scriptsize]
    (m-1-2) edge node[right] {\(\begin{smallpmatrix} 1 \\ 0 \end{smallpmatrix}\)} (m-3-3)
    (m-3-1) edge node[above] {\(f\)} (m-3-3)
            edge node[left] {\(\begin{smallpmatrix} f & \ins \end{smallpmatrix}\)} (m-1-2);
\end{tikzpicture}\]
\end{proposition}
\begin{proof}
We let \(C_X\) be a formal cone of \(X\) and we let \(C_f\) be a formal cone of \(f\) corresponding to \(C_X\), so that there exists a pushout rectangle
\[\begin{tikzpicture}[baseline=(m-2-1.base)]
  \matrix (m) [diagram]{
    C_X & C_f \\
    X & Y \\};
  \path[->, font=\scriptsize]
    (m-1-1) edge node[above] {\(f'\)} (m-1-2)
    (m-2-1) edge node[above] {\(f\)} (m-2-2)
            edge node[left] {\(\ins^{C_X}\)} (m-1-1)
    (m-2-2) edge node[right] {\(\ins^{C_f}\)} (m-1-2);
\end{tikzpicture}\]
in \(\mathcal{E}\). We get the factorisation \(f = \begin{smallpmatrix} f & \ins^{C_X} \end{smallpmatrix} \begin{smallpmatrix} 1 \\ 0 \end{smallpmatrix}\).
\[\begin{tikzpicture}[baseline=(m-3-1.base)]
  \matrix (m) [diagram=0.9em]{
    & Y \directsum C_X & \\
    & & \\
    X & & Y \\};
  \path[->, font=\scriptsize]
    (m-1-2) edge[exists] node[right=2pt] {\(\begin{smallpmatrix} 1 \\ 0 \end{smallpmatrix}\)} (m-3-3)
    (m-3-1) edge node[above] {\(f\)} (m-3-3)
            edge[exists] node[left] {\(\begin{smallpmatrix} f & \ins^{C_X} \end{smallpmatrix}\)} (m-1-2);
\end{tikzpicture}\]
By~\cite[prop.~2.12]{buehler:2010:exact_categories}, we have the pure short exact sequence
\[X \morphism[\begin{smallpmatrix} f & \ins^{C_X} \end{smallpmatrix}] Y \directsum C_X \morphism[\begin{smallpmatrix} \ins^{C_f} \\ - f' \end{smallpmatrix}] C_f.\]
Hence \(\begin{smallpmatrix} f & \ins^{C_X} \end{smallpmatrix}\colon X \map Y \directsum C_X\) is a pure monomorphism and \(\Cokernel \begin{smallpmatrix} f & \ins^{C_X} \end{smallpmatrix} \isomorphic C_f\). Moreover, since \(C_X\) is in \(\mathcal{U}\) and the split short exact sequence
\[C_X \morphism[\begin{smallpmatrix} 0 & 1 \end{smallpmatrix}] Y \directsum C_X \morphism[\begin{smallpmatrix} 1 \\ 0 \end{smallpmatrix}] Y\]
is a pure short exact sequence by~\cite[lem.~2.7]{buehler:2010:exact_categories}, the morphism \(\begin{smallpmatrix} 1 \\ 0 \end{smallpmatrix}\colon Y \directsum C_X \map Y\) is a quasi-isomorphism by proposition~\ref{prop:criterion_for_pure_monos_resp_pure_epis_being_quasi-isomorphism}\ref{prop:criterion_for_pure_monos_resp_pure_epis_being_quasi-isomorphism:pure_epi}.
\end{proof}

\begin{proposition} \label{prop:set_of_quasi-isomorphisms_in_an_exact_category_is_semi-saturated}
The set of \(\mathcal{U}\)-quasi-isomorphisms in \(\mathcal{E}\) is a semi-saturated denominator set in \(\mathcal{E}\).
\end{proposition}
\begin{proof}
For every object \(X\) in \(\mathcal{E}\), the identity \(1_X\) is a pure monomorphism~\cite[def.~2.1]{buehler:2010:exact_categories} with \(\Cokernel 1_X \isomorphic 0\). As \(0\) is in \(\mathcal{U}\) by remark~\ref{rem:subcategories_closed_under_pure_short_exact_sequences_contain_a_null_object_and_are_closed_under_isomorphisms}, it follows that \(1_X\) is a quasi-isomorphism for all \(X \in \Ob \mathcal{E}\) by proposition~\ref{prop:criterion_for_pure_monos_resp_pure_epis_being_quasi-isomorphism}\ref{prop:criterion_for_pure_monos_resp_pure_epis_being_quasi-isomorphism:pure_mono}.

We suppose given morphisms \(f\colon X_0 \map X_1\) and \(g\colon X_1 \map X_2\) in \(\mathcal{E}\). By proposition~\ref{prop:formal_cone_factorisation}, there exists for every formal cone \(C_f\) of \(f\) a pure monomorphism \(i\colon X_0 \map Y_0\) and a pure epimorphism \(p\colon Y_0 \map X_1\) with \(f = i p\) and such that \(\Cokernel i \isomorphic C_f\) and \(p\) is a quasi-isomorphism. Analogously, there exists for every formal cone \(C_g\) of \(g\) a pure monomorphism \(j\colon X_1 \map Y_1\) and a pure epimorphism \(q\colon Y_1 \map X_2\) with \(g = j q\) and such that \(\Cokernel j \isomorphic C_g\) and \(q\) is a quasi-isomorphism. Every formal cone \(C_0\) for \(X_0\) leads to a diagram as follows, where all quadrangles are pushout rectangles and hence also pullback rectangles~\cite[prop.~2.12]{buehler:2010:exact_categories}, and where \(D_0\) is a formal cone of \(i\) corresponding to \(C_0\), where \(C_1\) is a formal cone of \(i p = f\) corresponding to \(C_0\), where \(D_1\) is a formal cone of \(i p j\) corresponding to \(C_0\), and where \(C_2\) is a formal cone of \(i p j q = f g\) corresponding to \(C_0\).
\[\begin{tikzpicture}[baseline=(m-2-1.base)]
  \matrix (m) [diagram]{
    C_0 & D_0 & C_1 & D_1 & C_2 \\
    X_0 & Y_0 & X_1 & Y_1 & X_2 \\};
  \path[->, font=\scriptsize]
    (m-1-1) edge node[above] {\(i'\)} (m-1-2)
    (m-1-2) edge node[above] {\(p'\)} (m-1-3)
    (m-1-3) edge node[above] {\(j'\)} (m-1-4)
    (m-1-4) edge node[above] {\(q'\)} (m-1-5)
    (m-2-1) edge node[above] {\(i\)} (m-2-2)
            edge node[right] {\(\ins^{C_0}\)} (m-1-1)
    (m-2-2) edge node[above] {\(p\)} (m-2-3)
            edge node[right] {\(\ins^{D_0}\)} (m-1-2)
    (m-2-3) edge node[above] {\(j\)} (m-2-4)
            edge node[right] {\(\ins^{C_1}\)} (m-1-3)
    (m-2-4) edge node[above] {\(q\)} (m-2-5)
            edge node[right] {\(\ins^{D_1}\)} (m-1-4)
    (m-2-5) edge node[right] {\(\ins^{C_2}\)} (m-1-5);
\end{tikzpicture}\]
In particular, \(i'\) and \(j'\) are pure monomorphisms~\cite[def.~2.1]{buehler:2010:exact_categories} and \(p'\) and \(q'\) are pure epimorphisms~\cite[dual of prop.~2.15]{buehler:2010:exact_categories}.

So let us first suppose that \(f\) and \(g\) are quasi-isomorphisms. We choose formal cones \(C_f\) of \(f\) and \(C_g\) of \(g\) such that \(C_f\) and \(C_g\) are in \(\mathcal{U}\), and we choose an arbitrary formal cone \(C_0\) of \(X_0\). Then \(C_0\) is in \(\mathcal{U}\) and hence \(D_0\) is in \(\mathcal{U}\) since \(\Cokernel i' \isomorphic \Cokernel i \isomorphic C_f\) is in \(\mathcal{U}\) by~\cite[prop.~2.12]{buehler:2010:exact_categories} and \(\mathcal{U}\) is closed under pure short exact sequences. By proposition~\ref{prop:criterion_for_pure_monos_resp_pure_epis_being_quasi-isomorphism}\ref{prop:criterion_for_pure_monos_resp_pure_epis_being_quasi-isomorphism:pure_epi}, \(\Kernel p\) is in \(\mathcal{U}\) as \(p\) is a quasi-isomorphism. But then \(C_1\) is in \(\mathcal{U}\) since \(D_0\) and \(\Kernel p' \isomorphic \Kernel p\) are in \(\mathcal{U}\). Analogously, \(C_1\) in \(\mathcal{U}\) implies that \(D_1\) is in \(\mathcal{U}\), and this in turn implies that \(C_2\) is in \(\mathcal{U}\). But \(C_2\) is a formal cone of \(f g\) corresponding to \(C_0\), whence \(f g\) is a quasi-isomorphism.

Next, we suppose that \(f\) and \(f g\) are quasi-isomorphisms. We choose formal cones \(C_f\) of \(f\) and \(C_{f g}\) of \(f g\) such that \(C_f\) and \(C_{f g}\) are in \(\mathcal{U}\). Moreover, we choose the formal cone \(C_0\) of \(X_0\) such that \(C_{f g}\) is corresponding to \(C_0\). As shown above, \(C_0\) in \(\mathcal{U}\) implies that \(D_0\) is in \(\mathcal{U}\), and \(D_0\) in \(\mathcal{U}\) implies that \(C_1\) is in \(\mathcal{U}\). But then \(C_1\) is a formal cone of \(X_2\) and hence \(C_2\) is a formal cone of \(j q = g\) of corresponding to \(C_1\). Since \(C_2 \isomorphic C_{f g}\) is in \(\mathcal{U}\) by our choice of \(C_0\), this implies that \(g\) is a quasi-isomorphism.

Finally, let us suppose that \(g\) and \(f g\) are quasi-isomorphisms. We choose formal cones \(C_g\) of \(g\) and \(C_{f g}\) of \(f g\) such that \(C_g\) and \(C_{f g}\) are in \(\mathcal{U}\). Moreover, we choose the formal cone \(C_0\) of \(X_0\) such that \(C_{f g}\) is corresponding to \(C_0\). Then \(C_2 \isomorphic C_{f g}\) is in \(\mathcal{U}\) and \(\Kernel q' \isomorphic \Kernel q\) is in \(\mathcal{U}\), and therefore \(D_1\) is in \(\mathcal{U}\). This in turn implies that \(C_1\) is in \(\mathcal{U}\) since \(\Cokernel j' \isomorphic \Cokernel j \isomorphic C_g\) is in \(\mathcal{U}\). But \(C_1\) is a formal cone of \(f\) corresponding to \(C_0\), whence \(f\) is a quasi-isomorphism.

Altogether, the set of quasi-isomorphisms is a semi-saturated denominator set in \(\mathcal{E}\).
\end{proof}

\begin{corollary} \label{cor:equivalent_descriptions_for_quasi-isomorphisms_in_exact_categories_having_enough_formal_cones}
We suppose given a morphism \(f\) in \(\mathcal{E}\). The following conditions are equivalent.
\begin{enumerate}
\item \label{cor:equivalent_descriptions_for_quasi-isomorphisms_in_exact_categories_having_enough_formal_cones:quasi-isomorphism} The morphism \(f\) is a quasi-isomorphism with respect to \(\mathcal{U}\).
\item \label{cor:equivalent_descriptions_for_quasi-isomorphisms_in_exact_categories_having_enough_formal_cones:every_formal_cone_is_in_the_subcategory} Every formal \(\mathcal{U}\)-cone of \(f\) is in \(\mathcal{U}\).
\item \label{cor:equivalent_descriptions_for_quasi-isomorphisms_in_exact_categories_having_enough_formal_cones:factorisation_property} There exist a pure monomorphism \(i\) and a pure epimorphism \(p\) with \(f = i p\) and such that \(i\) and \(p\) are \(\mathcal{U}\)-quasi-isomorphisms.
\end{enumerate}
\end{corollary}
\begin{proof}
First, we let \(f\) be a quasi-isomorphism and we suppose given an arbitrary formal cone \(C_f\) of \(f\). By proposition~\ref{prop:formal_cone_factorisation}, there exist a pure monomorphism \(i\) and a pure epimorphism \(p\) with \(f = i p\) and such that \(p\) is a quasi-isomorphism and \(\Cokernel i \isomorphic C_f\). But \(i\) is a quasi-isomorphism since \(f\) is a quasi-isomorphism and since the set of quasi-isomorphisms is semi-saturated by proposition~\ref{prop:set_of_quasi-isomorphisms_in_an_exact_category_is_semi-saturated}, so \(C_f \isomorphic \Cokernel i\) is in \(\mathcal{U}\) by proposition~\ref{prop:criterion_for_pure_monos_resp_pure_epis_being_quasi-isomorphism}\ref{prop:criterion_for_pure_monos_resp_pure_epis_being_quasi-isomorphism:pure_mono}.

If every formal cone of \(f\) is an object of \(\mathcal{U}\), then \(f\) is a \(\mathcal{U}\)-quasi-isomorphism since \(\mathcal{E}\) has enough formal \(\mathcal{U}\)-cones.

Finally, if there exist a pure monomorphism \(i\) and a pure epimorphism \(p\) with \(f = i p\) and such that \(i\) and \(p\) are quasi-isomorphisms, then \(f\) is a quasi-isomorphism since the set of quasi-isomorphisms is multiplicative by proposition~\ref{prop:set_of_quasi-isomorphisms_in_an_exact_category_is_semi-saturated}. \qedhere
\end{proof}

Corollary~\ref{cor:equivalent_descriptions_for_quasi-isomorphisms_in_exact_categories_having_enough_formal_cones}\ref{cor:equivalent_descriptions_for_quasi-isomorphisms_in_exact_categories_having_enough_formal_cones:factorisation_property} and proposition~\ref{prop:criterion_for_pure_monos_resp_pure_epis_being_quasi-isomorphism} show that the notion of a quasi-isomorphism is self-dual, provided \(\mathcal{E}\) fulfills also the dual of definition~\ref{def:exact_category_having_enough_formal_cones}.

\begin{proposition} \label{prop:for_closed_under_taking_summands_subcategories_the_set_of_quasi-isomorphisms_in_an_exact_category_is_weakly_saturated}
If \(\mathcal{U}\) is closed under taking summands, then the set of \(\mathcal{U}\)-quasi-isomorphisms in \(\mathcal{E}\) is a weakly saturated denominator set in \(\mathcal{E}\).
\end{proposition}
\begin{proof}
We suppose that \(\mathcal{U}\) is closed under taking summands, and we suppose given morphisms \(f\colon X_0 \map X_1\), \(g\colon X_1 \map X_2\), \(h\colon X_2 \map X_3\) in \(\mathcal{E}\) such that \(f g\) and \(g h\) are quasi-isomorphisms. We let \(C_0\) be a formal cone of \(X_0\) and construct iteratively pushouts as in the following diagram, so that \(C_1\) is a formal cone of \(f\) corresponding to \(C_0\), so that \(C_2\) is a formal cone of \(f g\) corresponding to \(C_0\), and so that \(C_3\) is a formal cone of \(f g h\) corresponding to \(C_0\).
\[\begin{tikzpicture}[baseline=(m-2-1.base)]
  \matrix (m) [diagram]{
    C_0 & C_1 & C_2 & C_3 \\
    X_0 & X_1 & X_2 & X_3 \\};
  \path[->, font=\scriptsize]
    (m-1-1) edge node[above] {\(f'\)} (m-1-2)
    (m-1-2) edge node[above] {\(g'\)} (m-1-3)
    (m-1-3) edge node[above] {\(h'\)} (m-1-4)
    (m-2-1) edge node[above] {\(f\)} (m-2-2)
            edge node[right] {\(\ins^{C_0}\)} (m-1-1)
    (m-2-2) edge node[above] {\(g\)} (m-2-3)
            edge node[right] {\(\ins^{C_1}\)} (m-1-2)
    (m-2-3) edge node[above] {\(h\)} (m-2-4)
            edge node[right] {\(\ins^{C_2}\)} (m-1-3)
    (m-2-4) edge node[right] {\(\ins^{C_3}\)} (m-1-4);
\end{tikzpicture}\]
Next, we let \(D_1\) be a formal cone of \(C_1\) and construct again iteratively pushouts as in the following diagram, so that \(D_2\) is a formal cone of \(g'\) corresponding to \(D_1\), and so that \(D_3\) is a formal cone of \(g' h'\) corresponding to~\(D_1\).
\[\begin{tikzpicture}[baseline=(m-3-1.base)]
  \matrix (m) [diagram]{
    & D_1 & D_2 & D_3 \\
    C_0 & C_1 & C_2 & C_3 \\
    X_0 & X_1 & X_2 & X_3 \\};
  \path[->, font=\scriptsize]
    (m-1-2) edge node[above] {\(g''\)} (m-1-3)
    (m-1-3) edge node[above] {\(h''\)} (m-1-4)
    (m-2-1) edge node[above] {\(f'\)} (m-2-2)
    (m-2-2) edge node[above] {\(g'\)} (m-2-3)
            edge node[right] {\(\ins^{D_1}\)} (m-1-2)
    (m-2-3) edge node[above] {\(h'\)} (m-2-4)
            edge node[right] {\(\ins^{D_2}\)} (m-1-3)
    (m-2-4) edge node[right] {\(\ins^{D_3}\)} (m-1-4)
    (m-3-1) edge node[above] {\(f\)} (m-3-2)
            edge node[right] {\(\ins^{C_0}\)} (m-2-1)
    (m-3-2) edge node[above] {\(g\)} (m-3-3)
            edge node[right] {\(\ins^{C_1}\)} (m-2-2)
    (m-3-3) edge node[above] {\(h\)} (m-3-4)
            edge node[right] {\(\ins^{C_2}\)} (m-2-3)
    (m-3-4) edge node[right] {\(\ins^{C_3}\)} (m-2-4);
\end{tikzpicture}\]
Then \(D_1\) is also a formal cone of \(X_1\) and therefore \(D_3\) is a formal cone of \(g h\) corresponding to \(D_1\). Since \(f g\) and \(g h\) are quasi-isomorphisms, \(C_2\) and \(D_3\) are in \(\mathcal{U}\) by corollary~\ref{cor:equivalent_descriptions_for_quasi-isomorphisms_in_exact_categories_having_enough_formal_cones}. But since we have the pure short exact sequence
\[C_2 \morphism[\begin{smallpmatrix} h' & \ins^{D_2} \end{smallpmatrix}] C_3 \directsum D_2 \morphism[\begin{smallpmatrix} \ins^{D_3} \\ - h'' \end{smallpmatrix}] D_3\]
by~\cite[prop.~2.12]{buehler:2010:exact_categories}, we conclude that \(C_3 \directsum D_2\) is in \(\mathcal{U}\) and therefore that \(D_2\) is in \(\mathcal{U}\) since \(\mathcal{U}\) is closed under pure short exact sequences and taking summands. Thus \(g\) is a quasi-isomorphism as \(D_2\) is a formal cone of \(g\) corresponding to \(D_1\), and hence also \(f\), \(h\), \(f g h\) are quasi-isomorphisms by proposition~\ref{prop:set_of_quasi-isomorphisms_in_an_exact_category_is_semi-saturated}.
\end{proof}


\bigskip

{\raggedleft Sebastian Thomas \\ Lehrstuhl D f{\"u}r Mathematik \\ RWTH Aachen University \\ Templergraben 64 \\ 52062 Aachen \\ Germany \\ sebastian.thomas@math.rwth-aachen.de \\ \url{http://www.math.rwth-aachen.de/~Sebastian.Thomas/} \\}

\end{document}